\pdfoutput=1

\documentclass[11pt]{amsart}

\usepackage{amsmath, amssymb, amsthm, amsfonts}
\usepackage{graphicx}
\usepackage{pinlabel}
\usepackage{subfigure}
\usepackage[final]{listofsymbols}

\theoremstyle{plain}
\newtheorem{thm}{Theorem}[section]
\newtheorem{cor}[thm]{Corollary}
\newtheorem{lem}[thm]{Lemma}
\newtheorem{prop}[thm]{Proposition}

\theoremstyle{definition}
\newtheorem{defn}[thm]{Definition}
\newtheorem{rem}[thm]{Remark}

\newcommand{\df}{\ensuremath{\partial}}
\newcommand{\rr}{\ensuremath{\mathbb{R}}}
\newcommand{\zz}{\ensuremath{\mathbb{Z}}}
\newcommand{\aac}{\ensuremath{\mathcal{A}}}
\newcommand{\bbc}{\ensuremath{\mathcal{B}}}
\newcommand{\xxc}{\ensuremath{\mathcal{X}}}
\newcommand{\sMCS}{\ensuremath{\mathcal{C}}}
\newcommand{\sbMCS}{\ensuremath{[\mathcal{C}]}}
\newcommand{\sDMCSeq}{\ensuremath{\widehat{MCS}_b(\Sigma)}}
\newcommand{\IIinv}{\ensuremath{\mbox{II}^{-1}}}
\newcommand{\sHSM}{\ensuremath{\mathcal{H}}}
\newcommand{\sdmatrix}{\ensuremath{\mathcal{D}}}
\newcommand{\sFMCSeq}{\ensuremath{\widehat{MCS}(\Sigma)}}

\opensymdef
	\newsym[A Legendrian knot. We also use $K$ to refer to a Legendrian isotopy class.]{sK}{K}
	\newsym[The front projection of a Legendrian knot $K$.]{sfront}{\Sigma}	
	\newsym[The Lagrangian projection of a Legendrian knot $K$.]{sLagr}{L}
	\newsym[The Ng resolution of the front projection $\Sigma$.]{sNgres}{L_{\Sigma}}
	\newsym[A dipped diagram of the Ng resolution of $\Sigma$.]{sNgresD}{L_{\Sigma}^d}
	\newsym[An augmentation on a Lagrangian projection.]{saug}{\epsilon}
	\newsym[The chain homotopy equivalence relation on augmentations.]{schequiv}{\simeq}
	\newsym[A chain homotopy between augmentations.]{sH}{H}
	\newsym[An augmentation chain homotopy class.]{sbaug}{[\epsilon]}				
	\newsym[The set of augmentations of $L$.]{sAugL}{Aug(L)}	
	\newsym[The set of augmentations of $L_{\Sigma}$.]{sAugNgres}{Aug(L_{\Sigma})}		
	\newsym[The set of augmentation chain homotopy classes of $L_{\Sigma}$.]{sAugNgresch}{Aug^{ch}(L_{\Sigma})}		
	\newsym[The set of occ-simple augmentations of $L_{\Sigma}^d$.]{sAugsimple}{Aug_{occ}(L_{\Sigma}^d)}		
	\newsym[The set of minimal occ-simple augmentations of $L_{\Sigma}^d$.]{sAugminsimple}{Aug_{occ}^m (L_{\Sigma})}		
	\newsym[The set of MCSs of $\Sigma$.]{sDMCS}{MCS_b(\Sigma)}	
	\newsym[The full set of MCSs of $\Sigma$.]{sFMCS}{MCS(\Sigma)}
	\newsym[The equivalence relation on MCSs.]{sMCSeq}{\sim}	
	\newsym[A linear-quadratic generating family for $\Sigma$.]{sF}{F}
	\newsym[A Maslov potential on $\Sigma$.]{smu}{\mu}
	\newsym[A dip region of $L_\Sigma^d$.]{sdip}{D}
	\newsym[The $a$-lattice of a dip region of $L_\Sigma^d$.]{sA}{A}	
	\newsym[The $b$-lattice of a dip region of $L_\Sigma^d$.]{sB}{B}	
	\newsym[An insert region between dips of $L_\Sigma^d$.]{sI}{I}
	\newsym[A graded normal ruling on $\Sigma$.]{sgnr}{N}
	\newsym[The set of graded normal rulings of $\Sigma$.]{sSgnr}{N(\Sigma)}
	\newsym[Notation for the Augmentation form of an MCS.]{sAugform}{$A$}
	\newsym[Notation for moves in the SR and A algorithms.]{sVmove}{$V$}
\closesymdef

\begin{document}

\title[Connections between Legendrian Knot Invariants]{Connections between Floer-type Invariants and Morse-type Invariants of Legendrian Knots}

\author[Michael B. Henry]{Michael B. Henry} 
\address{The University of Texas at Austin, Austin, TX 78712} 
\email{mbhenry@math.utexas.edu}

\begin{abstract}

We define an algebraic/combinatorial object on the front projection $\Sigma$ of a Legendrian knot called a Morse complex sequence, abbreviated MCS. This object is motivated by the theory of generating families and provides new connections between generating families, normal rulings, and augmentations of the Chekanov-Eliashberg DGA. In particular, we place an equivalence relation on the set of MCSs on $\Sigma$ and construct a surjective map from the equivalence classes to the set of chain homotopy classes of augmentations of $L_\Sigma$, where $L_\Sigma$ is the Ng resolution of $\Sigma$. In the case of Legendrian knot classes admitting representatives with two-bridge front projections, this map is bijective. We also exhibit two standard forms for MCSs and give explicit algorithms for finding these forms. The definition of an MCS, the equivalence relation, and the statements of some of the results originate from unpublished work of Petya Pushkar.

\end{abstract}

\maketitle

\section{Introduction}
\label{ch:intro}

Legendrian knot theory is a rich refinement of smooth knot theory with deep connections to low-dimensional topology, symplectic and contact geometry, and singularity theory. In this article, we investigate connections between Legendrian knot invariants derived from Symplectic Field Theory and from the theory of generating families. Specifically, we relate augmentations, derived from Symplectic Field Theory, and Morse complex sequences, derived from generating families. A Legendrian knot $\sK$ in $\rr^3$ is a smooth knot whose tangent space sits in the standard contact structure $\xi$ on $\rr^3$, where $\xi$ is the kernel of the 1-form $ dz - ydx$. Legendrian knot theory is the study of Legendrian knots up to isotopy through Legendrian knots. 

We begin by recalling existing connections between Legendrian knot invariants. The Chekanov-Eliashberg differential graded algebra (abbreviated CE-DGA) of a Legendrian knot $\sK$ is a differential graded algebra $(\aac(\sLagr), \df)$ associated to the $xy$-projection $\sLagr$ of $\sK$. The CE-DGA is derived from the Symplectic Field Theory of \cite{Eliashberg,Eliashberg2000} and is developed in \cite{Eliashberg2000} and \cite{Chekanov2002a}. The homology of $(\aac(\sLagr), \df)$ is a Legendrian invariant and, if we consider $(\aac(\sLagr), \df)$ up to a certain algebraic equivalence, then the resulting DGA class is also a Legendrian invariant. Geometrically, the CE-DGA is Floer theoretic in nature. We refer the interested reader to \cite{Chekanov2002} and \cite{Etnyre2002} for a more detailed introduction.

An \emph{augmentation} is a type of algebra homomorphism from the CE-DGA of a Legendrian knot to the base field $\zz_2$. We denote the set of augmentations of $(\aac(\sLagr), \df)$ by $\sAugL$. There is a natural algebraic equivalence relation on $\sAugL$ and we denote the set of equivalence classes of augmentations of $\sLagr$ by $Aug^{ch}(\sLagr)$. The cardinality of $Aug^{ch}(\sLagr)$ is a Legendrian isotopy invariant. 

A second source of Legendrian invariants is the theory of generating families; see any of \cite{Chekanov2005,Jordan2006,L.Traynor2004,Traynor1997,Traynor2001}. A \emph{generating family} for a Legendrian knot $\sK$ encodes the $xz$-projection of $\sK$ as the Cerf diagram of a one-parameter family of functions $\sF_x$. In particular, let $W$ be a smooth manifold and $F: \rr \times W \to \rr$ be a smooth function. Let $\mathcal{C}_x \subset \{x\} \times W$ denote the set of critical points of $F_x = F(x, \cdot)$ and $\mathcal{C}_F = \bigcup_{x \in \rr}\mathcal{C}_x$. If the rank of the matrix of second derivatives of $F$ is maximal at all points in $\mathcal{C}_F$, then $\sfront_F =  \{(x,F(x,w)) | (x,w) \in \mathcal{C}_F \}$ is the $xz$-projection of an immersed Legendrian submanifold in $(\rr^3, \xi)$ and we say $F$ is a generating family for this submanifold; see Figure~\ref{f:gen-fam-example2}. If we restrict our attention to generating families that are sufficiently nice outside a compact set of the domain, then it can be shown that the existence of a generating family for a Legendrian knot is a Legendrian isotopy invariant; see \cite{Jordan2006}. 

Throughout this article, we will let $\sfront$ denote the $xz$-projection of a Legendrian knot and call it the \emph{front projection} of $\sK$. Every Legendrian knot can be Legendrian isotoped in an arbitrarily small neighborhood of itself so that the singularities of $\sfront$ are left cusps, right cusps, and transverse double points and the $x$-coordinates of these singularities are all distinct. We say a front is \emph{$\sigma$-generic} if its singularities are arranged in this manner. 

\begin{figure}[t]
\labellist
\small\hair 2pt
\pinlabel {$\sfront$} [br] at 170 270
\pinlabel {$z$} [br] at 16 318
\pinlabel {$x$} [tl] at 419 118
\pinlabel {$F_x$} [tl] at 442 107
\endlabellist
\centering
\includegraphics[scale=.35]{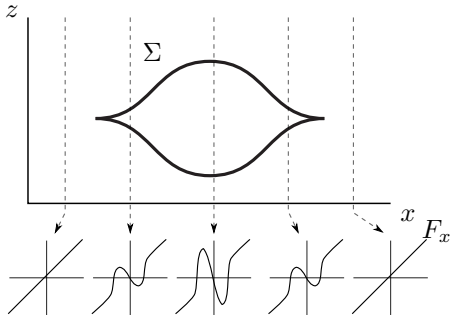}
\caption[Generating family example.]{A generating family for a Legendrian unknot.}
\label{f:gen-fam-example2}
\end{figure}

From a generating family, we may derive a combinatorial object called a graded normal ruling. Suppose a Legendrian knot $\sK$ with $\sigma$-generic front $\sfront$ admits a generating family $\sF : \rr \times W \to \rr$. The generating family may be chosen so that for all but finitely many values of $x$, $F_x$ is a Morse function whose critical points have distinct critical values. By placing an appropriate metric $g$ on $\rr \times W$, we may construct the Morse-Smale chain complex $(C_x, \df_x, g_x)$ of the pair $(F_x, g_x)$. A \emph{graded normal ruling} is a combinatorial object on $\sfront$ that encodes a certain pairing of the generators of $(C_x, \df_x, g_x)$ as $x$ varies. In \cite{Chekanov2005}, Chekanov and Pushkar work with a more general object called a pseudo-involution. Section 12 of \cite{Chekanov2005} provides a detailed explanation of the connection between generating families and pseudo-involutions, including the restrictions placed on the types of generating families considered and on the metric $g$. 

Many connections exist between augmentations, generating families, and graded normal rulings; see any of \cite{Chekanov2005,Fuchs2003,Fuchs2004,Fuchs2008,Kalman2006,Ng2006,Sabloff2005}. For a fixed Legendrian knot $\sK$ with Lagrangian projection $\sLagr$ and front projection $\sfront$, the following results are known.

\begin{thm}[\cite{Fuchs2003,Fuchs2004,Ng2006,Sabloff2005}]
\label{thm:many-to-one}
The CE-DGA (\aac(\sLagr), \df) admits a graded augmentation if and only if $\sfront$ admits a graded normal ruling. In particular, there exists a many-to-one map from the set of graded augmentations of $(\aac(\sLagr), \df)$ to the set of graded normal rulings of $\sfront$. 
\end{thm}

The backward direction of the first statement was proved by Fuchs in \cite{Fuchs2003}. Fuchs and Ishkhanov prove the forward direction in \cite{Fuchs2004}. Sabloff independently proves the forward direction in \cite{Sabloff2005}. Ng and Sabloff prove the second statement in \cite{Ng2006}.

\begin{thm}[\cite{Chekanov2005,Fuchs2008}]
A Legendrian knot $\sK$ admits a generating family if and only if $\sfront$ admits a graded normal ruling.
\end{thm}

Chekanov and Pushkar prove the forward direction in \cite{Chekanov2005} and state the backwards direction without proof. Fuchs and Rutherford prove the backwards direction in \cite{Fuchs2008}. 

By encoding the Morse theory data inherent in a generating family, we hope to refine the many-to-one map in Theorem~\ref{thm:many-to-one}. We encode this data in a finite sequence of chain complexes. The resulting algebraic object is called a \emph{Morse complex sequence}. Geometrically, it should be thought of as a sequence from the 1-parameter family of chain complexes $(C_x, \df_x, g_x)$ from a generating family. In this article we will not work explicitly with generating families, though they provide important geometric intuition. We let $\sFMCS$ denote the set of MCSs of $\sfront$ and $\sFMCSeq$ denote the set of MCSs of $\sfront$ up to a natural equivalence. 

In 1999, Petya Pushkar began a program to combinatorialize the Morse theory data coming from a generating family. The work with pseudo-involutions in \cite{Chekanov2005} may be considered the first step in this program. In \cite{Pushkar'a, Pushkar'}, Pushkar outlines his ``Spring Morse theory,'' which encodes the sequence $(C_x, \df_x, g_x)$ coming from a generating family and provides an equivalence relation on the resulting objects. The equivalence relation is the result of understanding the evolution of one-parameter families of functions and metrics. The ideas behind Morse complex sequences and the equivalence relation we define in this article originate with Petya Pushkar.
 
\subsection{Results} 
\label{sec:results}

Given a Legendrian knot $\sK$ with $\sigma$-generic front projection $\sfront$, we form the Ng resolution of $\sfront$, denoted $\sNgres$, by resolving the cusps and crossings as indicated in Figure~\ref{f:Ng-res}. The CE-DGA of $\sNgres$ is equal to the CE-DGA of an $xy$-projection for the Legendrian knot class of $\sK$. Therefore we may use $\sNgres$ to compare objects defined on $\sfront$ with objects derived from the CE-DGA of $\sK$. Given $\sfront$ and $\sNgres$, we have the following results.

\begin{figure}[t]
\labellist
\small\hair 2pt
\pinlabel {(a)} [tl] at 48 7
\pinlabel {(b)} [tl] at 164 7
\pinlabel {(c)} [tl] at 307 7
\endlabellist
\centering
\includegraphics[scale=.7]{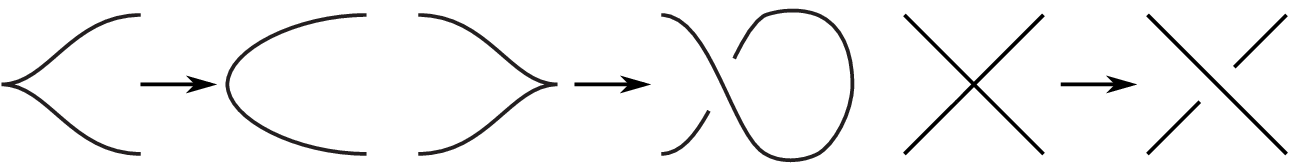}
\caption[Cusps and crossings in the Ng resolution procedure.]{Cusps and crossings in the Ng resolution procedure.}
\label{f:Ng-res}
\end{figure}

\begin{thm}
	\label{thm:intro-main-result}
	For a fixed Legendrian knot $\sK$ with $\sigma$-generic front projection $\sfront$ and Ng resolution $\sNgres$, there exists a surjective map $ \widehat{\Psi} : \sFMCSeq \to \sAugNgresch$.
\end{thm}

\begin{thm}
	\label{thm:intro-two-cusps}
	If $\sfront$ has exactly two left cusps, then $ \widehat{\Psi} : \sFMCSeq \to \sAugNgresch$ is a bijection.
\end{thm}

In \cite{Pushkar'}, Pushkar announced, without proof, results very similar to Theorem~\ref{thm:intro-main-result}.  

In the case of a front projection with exactly two left cusps, we can explicitly calculate $|\sAugNgresch| = |\sFMCSeq|$. The language in this corollary is defined in Section~\ref{ch:two-bridge}.

\begin{cor}
	\label{cor:intro-two-cusps}
	Suppose $\sfront$ has exactly two left cusps and let $\sSgnr$ denote the set of graded normal rulings on $\sfront$. For each $\sgnr \in \sSgnr$, define $\nu(\sgnr)$ to be the number of graded departure-return pairs in $\sgnr$. Then $|\sAugNgresch| = |\sFMCSeq| = \sum_{N \in \sSgnr} 2^{\nu(\sgnr)}$.
\end{cor}

In Section~\ref{sec:std-forms}, we describe two standard forms for MCSs on $\sfront$. The algorithms used to find these forms allow us to understand $\widehat{\Psi}$, while avoiding the lengthy algebra arguments required to prove Theorem~\ref{thm:intro-main-result}. We use the $S\bar{R}$-form to calculate bounds on the number of MCS classes associated to a fixed graded normal ruling. The $\sAugform$-form of an MCS $\sMCS$ allows us to easily compute the augmentation class $\widehat{\Psi}([\sMCS])$. 

\begin{thm}
	\label{thm:intro-std-forms}
	Every MCS is equivalent to an MCS in $S\bar{R}$-form and an MCS in $\sAugform$-form.
\end{thm}

\subsection{Outline of the rest of the article} 

In Section~\ref{ch:Background}, we provide the necessary background material in Legendrian knot theory. The front and Lagrangian projections of a Legendrian knot are used to develop combinatorial descriptions of the Chekanov-Eliashberg DGA, augmentations, and graded normal rulings. The definition of a Morse complex sequence (MCS) is given in Section~\ref{ch:Defining-MCS}, along with an equivalence relation on MCSs. Section~\ref{ch:Chain-Homotopy} reviews properties of differential graded algebras, DGA morphisms and DGA chain homotopies and applies them to the case of the CE-DGA and augmentations. We also sketch the proof that $Aug^{ch}(\sLagr)$ is a Legendrian knot invariant. In Section~\ref{ch:Dipped-Diagrams} we use a variation of the splash construction first developed in \cite{Fuchs2003} to write down the boundary map of the CE-DGA of a ``dipped'' version of $\sNgres$ as a system of local matrix equations. This gives us local control over augmentations and chain homotopies of augmentations. A number of lemmata are proved involving extending augmentations to dipped diagrams. In Section~\ref{ch:MCS-Aug} we develop the connections between MCSs and augmentations. We use the lemmata on dipped diagrams from Section~\ref{ch:Dipped-Diagrams} and the lemmata concerning chain homotopies from Section~\ref{ch:Chain-Homotopy} to explicitly construct $ \widehat{\Psi}$ and prove Theorem~\ref{thm:intro-main-result}. In Section~\ref{sec:std-forms}, we describe two standard forms for MCSs on $\sfront$ and prove Theorem~\ref{thm:intro-std-forms} using explicit algorithms. In Section~\ref{ch:two-bridge} we prove Theorem~\ref{thm:intro-two-cusps} and Corollary~\ref{cor:intro-two-cusps}.

\subsection{Acknowledgments}

My graduate advisors, Rachel Roberts and Joshua Sabloff, provided invaluable guidance and advice as I completed this work. In addition, I thank Sergei Chmutov, Dmitry Fuchs, Victor Goryunov, Paul Melvin, Dan Rutherford, and Lisa Traynor for many fruitful discussions and the American Institute of Mathematics (AIM), which held the September 2008 workshop conference where many of these discussions took place. 

It was at the AIM workshop that I first learned of the possibility of extending graded normal rulings by including handleslide data. Petya Pushkar had outlined such a program in a personal correspondence with Dmitry Fuchs in 2000 \cite{Pushkar'a}. Fuchs gave me copies of \cite{Pushkar'a} and the more recent \cite{Pushkar'} after the AIM workshop. Along with the work done at AIM, these notes directly influenced the definition of a Morse complex sequence and the equivalence relation on the set of MCSs given in this article. The results claimed by Pushkar in \cite{Pushkar'} also directed my efforts. It is my understanding that proofs of the claims in \cite{Pushkar'a} and \cite{Pushkar'} have not yet appeared.

\section{Background}
\label{ch:Background}

We assume the reader is familiar with the basic concepts in Legendrian knot theory, including front and Lagrangian projections, and the classical invariants. Throughout this article $\sfront$ and $\sLagr$ denote the front and Lagrangian projections of a knot $\sK$ respectively. Legendrian knots with non-zero rotation number do not admit augmentations, generating families, and graded normal rulings. Thus, we will always assume the rotation number of $\sK$ is $0$. Many of the statements of definitions and results in this section come from \cite{Sabloff}. An in-depth survey of Legendrian knot theory can be found in \cite{Etnyre2005}.

\subsubsection{The Ng resolution}
\label{subs:Ng-res}

In \cite{Ng2003}, Ng algorithmically constructs a Legendrian isotopy of $\sK$ so that the Lagrangian projection $\sLagr'$ of the resulting Legendrian knot $\sK'$ is topologically similar to $\sfront$. Definition~\ref{defn:Ng-res} gives a combinatorial description of the Lagrangian projection produced by Ng's resolution algorithm. 

\begin{defn}
	\label{defn:Ng-res}
	Given a Legendrian knot $\sK$ with front projection $\sfront$, we form the \emph{Ng resolution}, denoted $\sNgres$, by smoothing the left cusps as in Figure~\ref{f:Ng-res} (a), smoothing and twisting the right cusps as in Figure~\ref{f:Ng-res} (b), and resolving the double points as in Figure~\ref{f:Ng-res} (c).
\end{defn} 

The projection $\sNgres$ is regularly homotopic to $\sLagr'$ and the CE-DGA of $\sNgres$ is equal to the CE-DGA of $\sLagr'$. Given that $\sfront$ and $\sNgres$ are combinatorially very similar, the Ng resolution algorithm provides a natural first step towards finding connections between generating families and the CE-DGA.

\subsection{The Chekanov-Eliashberg DGA}
\label{sec:CE-DGA}

In \cite{Chekanov2002a} and \cite{Eliashberg2000}, Chekanov and Eliashberg develop a differential graded algebra, henceforth referred to as the CE-DGA, that has led to the discovery of several new Legendrian isotopy invariants. 

\subsubsection{The Algebra}

Label the crossings of a Lagrangian projection $\sLagr$ by $q_1, \hdots, q_n$. Let $A(\sLagr)$ denote the $\zz_2$ vector space freely generated by the elements of $Q = \{q_1, \hdots, q_n \}$. The algebra $\aac(\sLagr)$ is the unital tensor algebra $T A(\sLagr)$. We consider $\aac(\sLagr)$ to be a \emph{based} algebra since the algebra basis $Q$ is part of the data of $\aac(\sLagr)$. An element of $\aac(\sLagr)$ looks like the sum of noncommutative words in the letters $q_i$.

\subsubsection{The Grading}
	\label{sec:grading}

We define a $\zz$-grading $|q_i|$ on the generators $q_i$ and extend it to all monomials in $\aac(\sLagr)$ by requiring $|\prod_{j=1}^{l} q_{i_j}| = \sum_{j=1}^{l} |q_{i_j}|$. We isotope $\sK$ slightly so that the two strands meeting at each crossing of $\sLagr$ are orthogonal. Let $\gamma_i$ be a path in $\sK$ that begins on the overstrand of $\sLagr$ at $q_i$ and follows $\sLagr$ until it first returns to $q_i$. We let $r(\gamma_i)$ denote the fractional winding number of the tangent space of $\gamma_i$ with respect to the trivialization $\{ \partial_x, \partial_y \}$ of the tangent space of $\rr^2$. The \emph{grading of a crossing $q_i$} is defined to be	$|q_i| = 2 r (\gamma_i) - \frac{1}{2}$. The grading is well-defined since we have assumed $\sK$ has rotation number 0.

If $\sNgres$ is the Ng resolution of a front projection $\sfront$, then we can calculate the grading of the crossings of $\sNgres$ from $\sfront$ using a Maslov potential.    

\begin{defn}
	\label{defn:Maslov}	A \emph{strand} in $\sfront$ is a smooth path in $\sfront$ from a left cusp to a right cusp. A \emph{Maslov potential} on $\sfront$ is a map $\smu$ from the strands of $\sfront$ to $ \zz$ satisfying the relation shown in Figure~\ref{f:maslov-reeb} (a).
\end{defn}

\begin{figure}[t]
\labellist
\small\hair 2pt
\pinlabel {i+1} [tl] at -10 30
\pinlabel {i} [tl] at 2 8
\pinlabel {i+1} [tl] at 90 30
\pinlabel {i} [tl] at 90 8
\pinlabel {$+$} [tl] at 125 20
\pinlabel {$+$} [tl] at 157 20
\pinlabel {$-$} [tl] at 141 30
\pinlabel {$-$} [tl] at 141 7
\pinlabel {(a)} [tl] at 39 0
\pinlabel {(b)} [tl] at 138 0
\endlabellist
\centering
\includegraphics[scale=1]{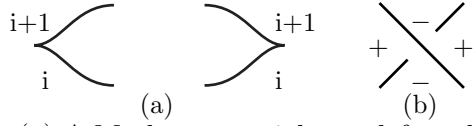}
\caption{(a) A Maslov potential near left and right cusps. (b) The Reeb sign of a crossing.}
\label{f:maslov-reeb}
\end{figure}

Given a crossing $q$ in $\sfront$, the grading of the corresponding resolved crossing $q$ in $\sNgres$ is computed by $|q| = \smu(T) - \smu(B)$ where $T$ and $B$ are the strands crossing at $q$ and $T$ has smaller slope. The crossings created by resolving right cusps have grading $1$.

\begin{figure}[t]
\centering
\includegraphics[scale=.65]{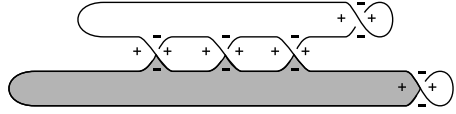}
\caption[A convex immersed polygon for a Lagrangian trefoil.]{A convex immersed polygon contributing $q_3 q_2 q_1$ to $\df q_5$. Crossings are labeled from left to right.}
\label{f:lagr-trefoil-DGA}
\end{figure}

\subsubsection{The Differential}
\label{subs:differential}

The differential of the CE-DGA counts certain disks in the $xy$-plane with polygonal corners at the crossings of $\sLagr$. We begin by decorating each corner of $q_i$ with a $+$ or $-$ sign as in Figure~\ref{f:maslov-reeb} (b). 

\begin{defn}
	\label{defn:convex-poly}
	Let $D$ be the unit disk in $\rr^2$ and let $\xxc = \{ x_0, \hdots, x_n \}$ be a set of distinct points along $\df D$ in counter-clockwise order. A \emph{convex immersed polygon} is a continuous map $f : D \to \rr^2$ so that:
	\begin{enumerate}
		\item $f$ is an orientation-preserving immersion on the interior of $D$;
		\item The restriction of $f$ to $\df D \setminus \xxc$ is an immersion whose image lies in $\sLagr$; and
		\item The image of each $x_i \in \xxc$ under $f$ is a crossing of $\sLagr$, and the image of a neighborhood of $x_i$ covers a convex corner at the crossing. Call a corner of an immersed polygon \emph{positive} if it covers a $+$ sign and \emph{negative} otherwise.
	\end{enumerate}
\end{defn}

\begin{rem}
	\label{rem:height-function}
Given a crossing $q_i$, let $\gamma_i$ be the path in $\rr^3$ that begins and ends on $\sK$, has constant $x$ and $y$ coordinates, and projects to the crossing $q_i$. Such paths are called \emph{Reeb chords}. We define a height function $h : Q \to \rr^{+}$ on the crossings of $\sLagr$ by defining $h(q_i)$ to be length of the path $\gamma_i$.

In the Ng resolution algorithm, we may arrange the Legendrian isotopy so that the height function on $\sLagr'$ is strictly increasing as we move from left to right along the $x$-axis. Thus the height function provides an ordering on the crossings of $\sNgres$ corresponding to the ordering coming from the $x$-axis. 
\end{rem}

\begin{lem}
	\label{lem:height-eq}
	Let $f$ be a convex immersed polygon and let $\gamma_i$ be the Reeb chord that lies over the corner $f(x_i)$. The following relationship holds:
	\begin{equation*}
		\sum_{x_i \text{ positive}} h(\gamma_i) - \sum_{x_j \text{ negative}} h(\gamma_j) = Area(f(D)).
	\end{equation*}
\end{lem}

As a consequence, every convex immersed polygon has at least one positive corner and in the case of an immersed polygon with a single positive corner, the height of the positive corner is greater than the height of each of the negative corners. The differential $\df q_i$ of the generator $q_i$ is a mod 2 count of convex immersed polygons with a single positive corner at $q_i$. 

\begin{defn}
	\label{defn:Delta-set}
	Let $\widetilde{\Delta} (q_i; q_{j_1}, \hdots, q_{j_k})$ be the set of convex immersed polygons with a positive corner at $q_i$ and negative corners at $q_{j_1}, \hdots, q_{j_k}$. The negative corners are ordered by the counter-clockwise order of the marked points along $\df D$. Let $\Delta (q_i; q_{j_1}, \hdots, q_{j_k})$ be $\widetilde{\Delta} (q_i; q_{j_1}, \hdots, q_{j_k})$ modulo smooth reparameterization.
\end{defn}

\begin{defn}
	\label{defn:DGA-boundary}
	The differential $\df$ on the algebra $\aac(\sLagr)$ is defined on a generator $q_i \in Q$ by the formula:
	\begin{equation*}
		\label{eq:DGA-diff}
		\df q_i = \sum_{\Delta (q_i; q_{j_1}, \hdots, q_{j_k})} \# (\Delta (q_i; q_{j_1}, \hdots, q_{j_k})) q_{j_1} \hdots q_{j_k}
	\end{equation*}
	\noindent where $\# (\Delta (\hdots))$ is the mod 2 count of the elements in $\Delta (\hdots)$. We extend $\df$ to all of $\aac(\sLagr)$ by linearity and the Leibniz rule. 
\end{defn}

The convex immersed polygon contributing the monomial $q_3 q_2 q_1$ to $\df q_5$ is given in Figure~\ref{f:lagr-trefoil-DGA}. Given the differential and grading defined above, $(\aac(\sLagr), \df)$ is a differential graded algebra.  

\begin{thm}[\cite{Chekanov2002a}]
	\label{thm:CE-DGA-diff}
	The differential $\df$ satisfies:
	\begin{enumerate}
		\item $| \df q | = |q| - 1 $ modulo $2 r(K)$, and 
		\item $\df \circ \df = 0$.
	\end{enumerate}
\end{thm}

\begin{rem}
In the case of an Ng resolution $\sNgres$, we noted in Remark~\ref{rem:height-function} that the heights of the crossings increase as we move from left to right along the $x$-axis. Thus the negative corners of a convex immersed polygon contributing to $\df$ in $(\aac(\sNgres), \df)$ always appear to the left of the positive corner. In Section~\ref{ch:Dipped-Diagrams}, this fact will allow us to find all convex immersed polygons contributing to $\df$. 
\end{rem}

In \cite{Chekanov2002a}, Chekanov defines an algebraic equivalence on DGAs, called \emph{stable tame isomorphism}, and proves that, up to this equivalence, the CE-DGA is a Legendrian isotopy invariant. In general, it is difficult to determine if two CE-DGAs are stable tame isomorphic.

\begin{defn}
	\label{defn:elem-iso}
	Given two algebras $\aac$ and $\aac'$, a grading-preserving identification of their generating sets $Q \leftrightarrow Q'$, and a generator $q_j$ for $\aac$, an \emph{elementary isomorphism} $\phi$ is an graded algebra map defined by:
	
	\begin{equation*} 
  \phi(q_i) = \begin{cases}
    q_i' & i \neq j \\
    q_j' + u &\mbox{ $i=j$, $u$ a term in $\aac'$ not containing $q_j'$}.
  \end{cases}
\end{equation*}
A composition of elementary isomorphisms is called a \emph{tame isomorphism}. A \emph{tame isomorphism of DGAs} between $(\aac, \df)$ and $(\aac', \df')$ is a tame isomorphisms $\Phi$ that is also a chain map, i.e. $\Phi \circ \df = \df' \circ \Phi$.
\end{defn}

\begin{defn}
	\label{defn:stabilization}
	Given a DGA $(\aac, \df)$ with generating set $Q$, the \emph{degree $i$ stabilization} $S_i (\aac, \df)$ is the differential graded algebra generated by the set $Q \cup \{e_1, e_2\}$, where $	|e_1| = i$ and $|e_2| = i-1$ and with the differential extended by $\df e_1 = e_2$ and $\df e_2 = 0.$
\end{defn}

\begin{defn}
	\label{defn:stable-tame-iso}
	Two DGAs $(\aac, \df)$ and $(\aac', \df')$ are \emph{stable tame isomorphic} if there exists stabilizations $S_{i_1}, \hdots, S_{i_m}$ and $S_{j_1}', \hdots, S_{j_n}'$ and a tame isomorphism of DGAs
	\begin{equation*}
	\psi : S_{i_1} ( \hdots S_{i_m} (\aac) \dots ) \to S_{j_1}' ( \hdots S_{j_n}' (\aac') \dots )
	\end{equation*}
	\noindent so that the composition of maps is a chain map.
\end{defn}

A stable tame isomorphism preserves the homology of the DGA and so the homology of $(\aac(\sLagr), \df)$ is also a Legendrian isotopy invariant, called the \emph{Legendrian contact homology} of $\sK$. 

\subsection{Augmentations}
\label{sec:Augs}

In \cite{Chekanov2002a}, Chekanov implicitly defines a class of DGA chain maps called augmentations. 

\begin{defn}
	\label{defn:aug}
	An \emph{augmentation} is an algebra map $\saug : (\aac(\sLagr), \df) \to \zz_2$ satisfying $\saug(1) = 1$, $\saug \circ \df = 0$, and if $\saug (q_i) = 1$ then $ |q_i| = 0 $. We let $Aug(\sLagr)$ denote the set of augmentations of $(\aac(\sLagr), \df)$.
\end{defn}

A tame isomorphism between DGAs induces a bijection on the corresponding sets of augmentations. Stabilizing a DGA may double the number of augmentations, depending on the grading of the new generators. It is possible to normalize the number of augmentations by an appropriate power of two and obtain an integer Legendrian isotopy invariant; see \cite{Mishachev2003,Ng2006}. 

\subsection{Graded normal rulings}
\label{sec:Rulings}

The geometric motivation for a graded normal ruling on $\sfront$ comes from examining the one-parameter family of Morse-Smale chain complexes from a suitably generic generating family $\sF$ for $\sfront$. Section 12 of \cite{Chekanov2005} provides a detailed explanation of the connection between generating families and graded normal rulings, including the restrictions placed on the types of generating families considered. In the language of \cite{Chekanov2005}, a graded normal ruling is a positive Maslov pseudo-involution. As combinatorial objects, graded normal rulings are defined as follows.

\begin{defn}
	\label{defn:ruling}
	A \emph{ruling} on the front diagram $\sfront$ is a one-to-one correspondence between the set of left cusps and the set of right cusp and, for each corresponding pair of cusps, two paths in $\sfront$ that join them. The paths satisfy:	
\begin{enumerate}
	\item Any two paths in the ruling meet only at crossings or at cusps; and
	\item The two paths joining corresponding cusps meet only at the cusps, hence their interiors are disjoint.
\end{enumerate}
\end{defn}

The two paths joining corresponding cusps are called \emph{companions} of one another. Together the two paths bound a disk in the plane called the \emph{ruling disk}. At a crossing, two paths either pass through each other or one path lies entirely above the other. In the latter case, we call the crossing a \emph{switch}.  

\begin{defn}
	\label{defn:graded-normal-ruling}
	We say a ruling is \emph{graded} if each switched crossing has grading 0, where the grading comes from a Maslov potential as defined in Section~\ref{sec:grading}. A ruling is \emph{normal} if at each switch the two paths at the crossing and their companion strands are arranged as in Figure~\ref{f:normal-switches}(a). Let $\sSgnr$ denote the set of graded normal ruling of a Legendrian knot with front projection $\sfront$.
\end{defn}

Figure~\ref{f:normal-switches}(b) gives a graded normal ruling of the standard Legendrian trefoil. In \cite{Chekanov2005}, Chekanov and Pushkar prove the following. 

\begin{thm}[\cite{Chekanov2005}]
	\label{thm:ruling-invt}
	If $\sK$ and $\sK'$ are Legendrian isotopic and $\sfront$ and $\sfront'$ are $\sigma$-generic, then $\sfront$ and $\sfront'$ admit the same number of graded normal rulings.
\end{thm}

In a normal ruling, there are two types of unswitched crossings. A \emph{departure} is an unswitched crossing in which, to the left of the crossing, the two ruling disks are either disjoint or one is nested inside the other. A \emph{return} is an unswitched crossing in which the two ruling disks partially overlap to the left of the crossing. 

\begin{figure}[t]
\labellist
\small\hair 2pt
\pinlabel {(a)} [tl] at 88 7
\pinlabel {(b)} [tl] at 309 7
\endlabellist
\centering\includegraphics[scale=.7]{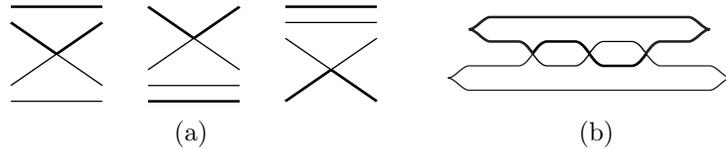}
\caption[Configurations of a normal switch and a graded normal ruling on a Legendrian trefoil.]{(a) The three possible configurations of a normal switch. (b) A graded normal ruling on the standard Legendrian trefoil.}
\label{f:normal-switches}
\end{figure}

\section{Defining Morse Complex Sequences}
\label{ch:Defining-MCS}

As indicated in the Introduction, the ideas behind \emph{Morse complex sequences} and the equivalence relation we define in this section originate with Petya Pushkar. In the language of \cite{Pushkar'}, equivalence classes of Morse complex sequences correspond to Pushkar's ``combinatorial generating families.''

A \emph{Morse complex sequence} is a finite sequence of chain complexes and a set of chain maps relating consecutive chain complexes. Its definition is geometrically motivated by the sequence of Morse-Smale chain complexes $(C_x, \df_x, g_x)$ coming from a suitably generic generating family $\sF$ and metric $g$. The local moves used to define an equivalence relation are found by considering two-parameter families of function/metric pairs $(F_x^t, g_x^t)$. In \cite{Hatcher1973}, Hatcher and Wagoner describe possible relationships between the sequences $(C_x, \df_x, g^0_x)$ and $(C_x, \df_x, g^1_x)$. In \cite{Pushkar'a}, Pushkar identifies the relationships that are necessary for working with Legendrian front projections. 

\subsection{Ordered chain complexes}

We begin by defining the chain complexes that comprise a Morse complex sequence.  

\begin{defn}
	\label{defn:occ}
	An \emph{ordered chain complex} is a $ \zz_2 $ vector space $C$ with ordered basis $y_1 < y_2 < \hdots < y_m$, a $\zz$ grading on $y_1, \hdots, y_m$, denoted $|y_j|$ and a linear map $\df : C \to C$, that satisfies:
	
	\begin{enumerate}
		\item $ \df \circ \df = 0 $, 
		\item $ | \df y_j | = | y_j | - 1 $, and
		\item $ \df y_j = \sum_{i<j} a_{j,i} y_i $, where $a_{j,i} \in \zz_2$.
	\end{enumerate}	
We denote an ordered chain complex by $(C, \df)$ when the ordered basis and grading are understood. We let $\langle \df y_j | y_i \rangle$ denote the contribution of the generator $y_i$ to $\df y_j$, i.e. $\langle \df y_j | y_i \rangle = a_{j,i}$.  
\end{defn}

\begin{rem}
	\label{rem:matrix-rep}
The $m \times m$ lower triangular matrix $\sdmatrix$ defined by $(\sdmatrix)_{j,i} = a_{j,i}$ for $j > i$ is a matrix representative of the map $\df$. Indeed, $\df^2 = 0$ implies $\sdmatrix^2 = 0$ and with respect to the basis $\{y_m\}$ the $\zz_2$ coefficients of $\df y_j$ are given by $e_j \sdmatrix$, where $e_j$ is the $j^{\mbox{th}}$ standard basis row vector. In Section~\ref{ch:MCS-Aug}, we use the matrix representatives of a sequence of ordered chain complexes to associate an augmentation to an MCS.
\end{rem}

In Figure~\ref{f:occ-example} we give an example of the graphical encoding of a ordered chain complex $(C, \df)$ used in \cite{Barannikov1994}. The vertical lines indicate the gradings of the generators $y_1, \hdots, y_6$, the height of the vertices on the vertical lines indicates the ordering of the generators, and the sloped lines connecting vertices represent the boundary map $\df$. 

\begin{figure}
\labellist
\small\hair 2pt
\pinlabel {3} [tl] at 1 4
\pinlabel {2} [tl] at 26 4
\pinlabel {1} [tl] at 50 4
\pinlabel {0} [tl] at 74 4
\pinlabel {$y_6$} [bl] at 4 63
\pinlabel {$y_5$} [bl] at 29 54
\pinlabel {$y_4$} [bl] at 29 42
\pinlabel {$y_3$} [bl] at 53 31
\pinlabel {$y_2$} [tl] at 53 17
\pinlabel {$y_1$} [bl] at 78 13
\endlabellist
\centering
\includegraphics[scale=1]{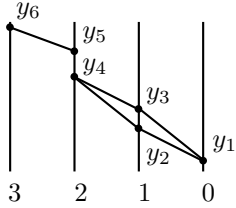}
\caption[An ordered chain complex.]{The graphical presentation of an ordered chain complex. The sloped lines from $y_4$ to $y_3$ and $y_2$ indicate that $\df y_4 = y_2 + y_3$.}
\label{f:occ-example}
\end{figure}

\subsubsection{Matrix Notation}
\label{sec:matrices}

All of the matrices in this article have entries in $\zz_2$ and matrix operations are done mod $2$. For $k > l$, we let $\sHSM_{k,l}$ denote a square matrix with 1 in the $(k,l)$ position and zeros everywhere else. We let $E_{k,l} = I + \sHSM_{k,l}$ where $I$ denotes the identity matrix.	We let $P_{i+1,i}$ denote the square permutation matrix obtained by interchanging rows $i$ and $i+1$ of the identity matrix. Finally, we let $J_{i-1}$ denote the matrix obtained by inserting two columns of zeros after column $i-1$ in the identity matrix. The matrix $J_{i-1}^{T}$ is the transpose of $J_{i-1}$.

\subsection{Morse complex sequences}

Consecutive ordered chain complexes in a Morse complex sequence are related by one of the following four chain maps.

\begin{defn}
	\label{defn:MCS-maps}
Suppose $(C, \df)$ and $(C', \df')$ are ordered chain complexes with ordered generating sets $y_1 < \hdots < y_{n}$ and $y'_1 < \hdots < y'_{m}$ respectively. We define four types of chain isomorphisms. 

\begin{enumerate}
	\item Suppose $n=m$, $1 \leq l < k \leq n$, and $|y_k| = |y_l|$. We say $(C, \df)$ and $(C', \df')$ are \emph{related by a handleslide between $k$ and $l$} if and only if the linear extension of the map on generators defined by 
				
			\begin{equation*}
			\phi_1 ( y_i ) =
			\begin{cases}
 y'_i &\mbox{ if $i \neq k $} \\
  y'_k + y'_l &\mbox{ if $i = k $}
      \end{cases}
      \end{equation*}
      
     \noindent is a chain isomorphism from $(C, \df)$ to $(C', \df')$. As matrices $\sdmatrix = E_{k,l} \sdmatrix' E^{-1}_{k,l}$.
     
	\item Suppose $n=m$ and $1 \leq k < n$. We say $(C, \df)$ and $(C', \df')$ are \emph{related by interchanging critical values at $k$} if and only if the linear extension of the map on generators defined by 
				
			\begin{equation*}
			\phi_2 ( y_i ) =
			\begin{cases}
  y'_i &\mbox{ if $i \notin \{k, k+1\} $} \\
  y'_{k+1} &\mbox{ if $i = k $} \\
  y'_{k} &\mbox{ if $i = k+1 $} \\
      \end{cases}
      \end{equation*}
      
      \noindent is a chain isomorphism from $(C, \df)$ to $(C', \df')$. As matrices $\sdmatrix = P_{k+1} \sdmatrix' P^{-1}_{k+1}$.
	
	\item Suppose $n = m-2$, $1 \leq k < m$ and $\langle \df' y'_{k+1} | y'_k \rangle = 1$. We say $(C, \df)$ and $(C', \df')$ are \emph{related by the birth of two generators at $k$} if and only if $(C, \df)$ is chain isomorphic to the quotient of $(C', \df')$ by the acyclic subcomplex generated by $\{ y'_{k+1}, \df' y'_{k+1} \}$ by the map 
			\begin{equation*}
			\phi_3 ( y_i ) =
			\begin{cases}
  [y'_i] &\mbox{ if $i < k $} \\
  [y'_{i+2}] &\mbox{ if $i > k+1 $}.
      \end{cases}
      \end{equation*}
The matrix $\sdmatrix$ is computed explicitly as follows. 
\begin{enumerate}
		\item Let $y'_{k+1} < y'_{u_1} < y'_{u_2} < \hdots < y'_{u_s}$ denote the generators of $C'$ satisfying  $\langle \df' y' | y'_{k} \rangle = 1$.  
		\item Let  $y'_{v_r} <  \hdots < y'_{v_1} < y'_k$ denote the generators of $C'$ satisfying $\langle \df' y'_{k+1} | y' \rangle = 1$.
		\item Let $E = E_{k, v_r} \hdots E_{k, v_1} E_{u_1, k+1} \hdots E_{u_s, k+1}I$.
	\end{enumerate}
	
	Then, as matrices:
      \begin{equation*}
      \label{eq:type-3-MCS}	
	 			\sdmatrix = J_{i-1} E \sdmatrix' E^{-1} J_{i-1}^{T}.
	 		\end{equation*}

We say $(C, \df)$ and $(C', \df')$ are \emph{related by the death of two generators at $k$} if the roles of $(C, \df)$ and $(C', \df')$ are exchanged in the map above. In particular, $n = m+2$, $\langle \df y_{k+1} | y_k \rangle = 1$ for some $1 \leq k < n$, and $(C', \df')$ is chain isomorphic to the quotient of $(C, \df)$ by the acyclic subcomplex generated by $\{ y_{k+1}, \df y_{k+1} \}$ by the map $\phi_4$ given by $y'_i \mapsto [y_i]$ if $i < k $ and $y'_i \mapsto [y_{i+2}]$ otherwise. 	 			
\end{enumerate}
\end{defn}      

In the matrix equation in (3), $E \sdmatrix' E^{-1}$ represents a series of handleslide moves on the chain complex $(C', \df')$. In $E \sdmatrix' E^{-1}$, $y_{k+1}$ and $y_{k}$ form a trivial acyclic subcomplex and $\sdmatrix = J_{i-1} E \sdmatrix' E^{-1} J_{i-1}^{T}$ is the result of quotienting out this subcomplex.

A Morse complex sequence encodes possible algebraic changes resulting from a generic deformation between two Morse functions without critical points. A generating family for a Legendrian knot is such a deformation.

\begin{defn}
	\label{defn:MCS}
	A \emph{Morse complex sequence $\sMCS$} is a finite sequence of ordered chain complexes $(C_1, \df_1) \hdots (C_m, \df_m)$ with ordered generating sets $y_1^j < \hdots < y_{n_j}^j$ for each $(C_j, \df_j)$, $1 \leq j \leq m$, and $\tau_j \in \zz \oplus \zz$ for $1\leq j < m$ satisfying:	
	
\begin{enumerate}
		\item For $(C_1, \df_1)$ and $(C_m, \df_m)$, $n_1 = n_m = 2$, $\langle \df_1 y_2^1 | y_1^1 \rangle = 1$ and $\langle \df_m y_2^m | y_1^m \rangle = 1$. In particular, both $(C_1, \df_1)$ and $(C_m, \df_m)$ have trivial homology.
		\item For each $1 \leq j < m$, $| n_{j+1} - n_j | \in \{ 0, 2 \}$.
		\item If $n_j = n_{j+1} - 2$, then $\tau_j = (k,0)$ for some $k$ and $(C_j, \df_j)$ and $(C_{j+1}, \df_{j+1})$ are related by the birth of two generators at $k$.
		\item If $n_j = n_{j+1} + 2$, then $\tau_j = (k,0)$ for some $k$ and $(C_j, \df_j)$ and $(C_{j+1}, \df_{j+1})$ are related by the death of two generators at $k$.
		\item If $n_j = n_{j+1}$, then either:
			\begin{enumerate}
				\item $\tau_j = (k,0)$ for some $k$ and $(C_j, \df_j)$ and $(C_{j+1}, \df_{j+1})$ are related by interchanging critical values at $k$, or
					\item $\tau_j = (k,l)$ for some $1 \leq l < k \leq n_j$ and $(C_j, \df_j)$ and $(C_{j+1}, \df_{j+1})$ are related by a handleslide between $k$ and $l$.
			\end{enumerate}		
\end{enumerate}

\end{defn}

\subsection{Associating a Marked Front Projection to an MCS}
\label{sec:mcs-marked-front}

We may encode an MCS graphically using a front projection and certain vertical line segments. The vertical marks encode the handleslides occurring in the chain isomorphisms of type (1) and (3).

\begin{defn}
	\label{defn:handleslide-mark}
	A \emph{handleslide mark} on a $\sigma$-generic front projection $\sfront$ with Maslov potential $\smu$ is a vertical line segment in the $xz$-plane with endpoints on $\sfront$. We require that the line segment not intersect the crossings or cusps of $\sfront$ and the endpoints sit on strands of $\sfront$ with the same Maslov potential. A \emph{marked front projection} is a $\sigma$-generic front projection with a collection of handleslide marks; see Figure~\ref{f:MCS-example2}.
\end{defn}

\begin{figure}
\labellist
\small\hair 2pt
\pinlabel {$x_1$} [tr] at 90 105
\pinlabel {$x_2$} [tr] at 140 105
\pinlabel {$x_3$} [tr] at 230 105
\pinlabel {$x_4$} [tr] at 310 105
\pinlabel {$x_5$} [tr] at 350 105
\pinlabel {$x_6$} [tr] at 400 105
\pinlabel {$x_7$} [tr] at 470 105
\endlabellist\centering
\includegraphics[scale=.5]{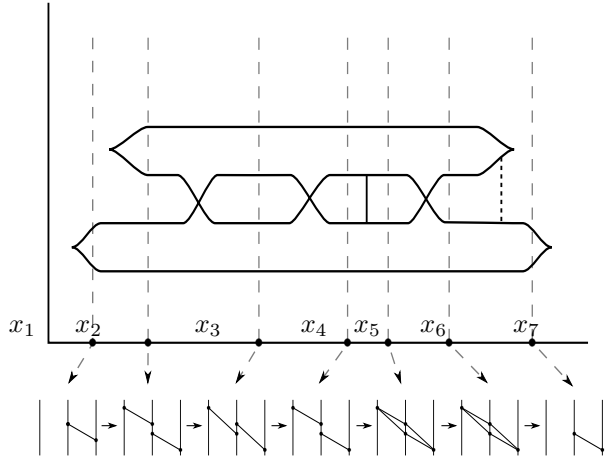}
\caption[An MCS on a Legendrian trefoil.]{An example of an MCS and its associated marked front projection.}
\label{f:MCS-example2}
\end{figure}

\begin{figure}
\labellist
\small\hair 2pt
\pinlabel {(a)} [tl] at 26 3
\pinlabel {(b)} [tl] at 119 3
\pinlabel {(c)} [tl] at 210 3
\pinlabel {(d)} [tl] at 295 3
\endlabellist
\centering
\includegraphics[scale=.6]{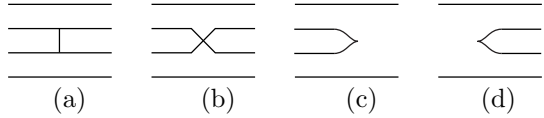}
\caption{The four possible tangles used in Section~\ref{sec:mcs-marked-front}.}
\label{f:mcs-tangle}
\end{figure}

Let $(C_1, \df_1) \hdots (C_m, \df_m)$ and $\tau_1, \hdots, \tau_{m-1} \in \zz \times \zz$ be an MCS $\sMCS$. We build the marked front projection for $\sMCS$ inductively beginning with a single left cusp and building to the right by adjoining tangles of the types in Figure~\ref{f:mcs-tangle}. 

For each $1 \leq j < m$, we adjoin a tangle from Figure~\ref{f:mcs-tangle}. At each step, we assume the strands of our tangle are numbered from top to bottom.  If $(C_j, \df_j)$ and $(C_{j+1}, \df_{j+1})$ are related by a handleslide between $k$ and $l$, then we adjoin a tangle of type (a) with a handleslide mark between strands $k$ and $l$. If $(C_j, \df_j)$ and $(C_{j+1}, \df_{j+1})$ are related by interchanging critical values at $k$, then we adjoin a tangle of type (b) with a crossing between strands $k$ and $k+1$. If $(C_j, \df_j)$ and $(C_{j+1}, \df_{j+1})$ are related by the death of two generators at $k$, then we adjoin a tangle of type (c) where the right cusp connects strands $k$ and $k+1$. If $(C_j, \df_j)$ and $(C_{j+1}, \df_{j+1})$ are related by the birth of two generators at $k$, then we adjoin a tangle of type (d) where the left cusp sits above strand $k-1$.

We will refer to the marks arising from handleslides between consecutive chain complexes as \emph{explicit handleslide marks}. At the cusps corresponding to births or deaths, we also place marks in a small neighborhood of the cusp. These marks correspond to the composition of handleslides represented by the matrix product $E$ in (3) of Definition~\ref{defn:MCS-maps}. We call these \emph{implicit handleslide marks} and use dotted vertical lines to distinguish them from the explicit handleslide marks.

The front projection we have constructed is $\sigma$-generic and has a natural Maslov potential coming from the index of the generators in $(C_1, \df_1) \hdots (C_m, \df_m)$. In addition, the handleslide marks connect strands of the same Maslov potential. Hence, we have constructed a marked front projection. Figure~\ref{f:MCS-example2} gives an MCS and its associated marked front projection.

\begin{rem}
\label{rem:mcs-defn}
We note the following:
\begin{enumerate}
	\item If a marked front projection is associated to $\sMCS$, then, by reversing the construction defined above, $(C_1, \df_1) \hdots (C_m, \df_m)$ and $\tau_1, \hdots, \tau_{m-1}$ are uniquely reconstructed from the cusps, crossings and handleslide marks of the marked front projection using part (1) of Definition~\ref{defn:MCS} and the chain isomorphisms from Definition~\ref{defn:MCS-maps}. Thus, a marked front projection may be associated with at most one MCS.
	\item It follows from Definition~\ref{defn:MCS-maps} that if $(C_j, \df_j)$ and $(C_{j+1}, \df_{j+1})$ are related by the death of two generators at $i$, then $\langle \df_{j}y_{i+1}^j | y_{i}^j \rangle = 1$ and if $(C_j, \df_j)$ and $(C_{j+1}, \df_{j+1})$ are related by interchanging critical values at $i$, then $\langle \df_{j}y_{i+1}^j | y_{i}^j \rangle = 0$.
	\item In an MCS $\sMCS$, the homologies of $(C_j, \df_j)$ and $(C_{j+1}, \df_{j+1})$ are isomorphic. Thus, by (1) in Definition~\ref{defn:MCS}, all of the homologies in $\sMCS$ are trivial.
\end{enumerate}
\end{rem}

\begin{defn}
Given a $\sigma$-generic front projection $\sfront$, an MCS $\sMCS$ is in the set $\sFMCS$ if and only if the marked front projection associated to $\sMCS$ has a front projection that is planar isotopic to $\sfront$ by a planar isotopy through $\sigma$-generic front projections. Such an isotopy pulls back to a Legendrian isotopy in $\rr^3$.
\end{defn}

In \cite{Pushkar'a} and \cite{Pushkar'}, Pushkar defines a ``spring sequence'' to be a sequence of ordered chain complexes over a commutative ring $\mathbb{E}$ with consecutive complexes connected by maps similar to those defined in Definition~\ref{defn:MCS-maps}. Originally defined in Section~12 of \cite{Chekanov2005}, these chain complexes are called $M$-complexes. Pushkar encodes a spring sequence by adding vertical marks, called ``springs,'' to the front projection of an associated Legendrian knot.

\subsection{An Equivalence Relation on $\sFMCS$}

In this section, we describe a set of local moves used to define an equivalence relation on $\sFMCS$. Since an MCS is completely determined by its associated marked front projection, these moves are defined as graphical changes in the handleslide marks. The equivalence relation encodes local algebraic changes resulting from a generic deformation of a 1-parameter family of Morse functions. Such deformations are studied extensively in \cite{Hatcher1973}.

Figures~\ref{f:MCS-equiv-2}, \ref{f:MCS-equiv-3}, and \ref{f:MCS-equiv-explosion} describe local changes in the handleslide marks of a marked front projection. We call these \emph{MCS moves}. Other strands may appear in a local neighborhood of an MCS move, although we assume no other crossings or cusps appear.  The ordering of implicit handleslide marks at a birth or death is irrelevant, thus we will not consider analogues of moves 1 - 6 for implicit handleslide marks. In moves 11 - 14, there may be other implicit marks that the indicated explicit handleslide mark commutes past without incident. Additional implicit marks may also appear at the birth or death in moves 15 and 16.

MCS move 17 requires explanation. Let $\sMCS \in \sFMCS$ and suppose $(C, \df)$ is an ordered chain complex in $\sMCS$ with generators $y_1 < \hdots < y_m$  and a pair of generators $y_l < y_k$ such that $|y_l| = |y_k| + 1$. Let $y_{u_1} < y_{u_2} < \hdots < y_{u_s}$ denote the generators of $C$ satisfying  $\langle \df y | y_{k} \rangle = 1$; see the left three arrows in Figure~\ref{f:MCS-equiv-explosion}. Let  $y_{v_r} <  \hdots < y_{v_1} < y_i$ denote the generators of $C$ appearing in $\df y_{l}$; see the right two arrows in Figure~\ref{f:MCS-equiv-explosion}. Let $E = E_{k, v_r} \hdots E_{k, v_1} E_{u_1, l} \hdots E_{u_s, l}$. Then over $\zz_2$, the following formula holds:

\begin{equation}
	\label{eq:explosion}
\sdmatrix = E \sdmatrix E^{-1}.
\end{equation}

The MCS move in Figure~\ref{f:MCS-equiv-explosion} says that we may either introduce or remove the handleslides represented by $E$ and this move is local in the sense that it does not change any of the other chain complexes in $\sMCS$. The next proposition shows that all of the MCS moves are local in this sense. 

\begin{figure}
\labellist
\small\hair 2pt
\pinlabel {1} [br] at 115 365
\pinlabel {3} [br] at 115 286
\pinlabel {5} [br] at 115 206
\pinlabel {7} [br] at 115 126
\pinlabel {9} [br] at 115 45
\pinlabel {2} [br] at 345 365
\pinlabel {4} [br] at 345 286
\pinlabel {6} [br] at 345 206
\pinlabel {8} [br] at 345 126
\pinlabel {10} [br] at 347 45
\endlabellist
\centering
\includegraphics[scale=.5]{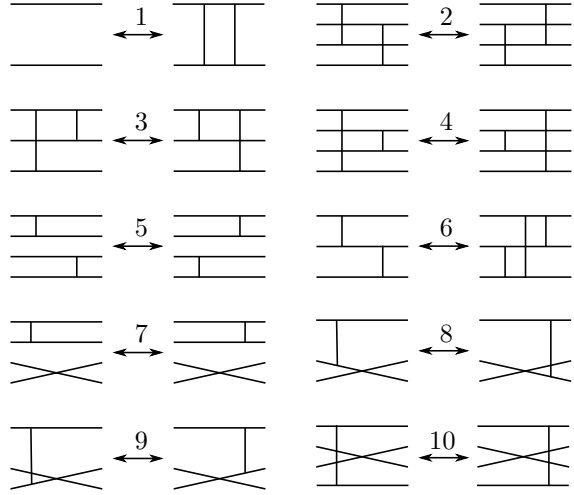}
\caption[MCS moves 1 - 10]{MCS moves 1 - 10. We also allow reflections about the horizontal axis.}
\label{f:MCS-equiv-2}
\end{figure}

\begin{figure}
\labellist
\small\hair 2pt
\pinlabel {11} [br] at 117 185
\pinlabel {12} [br] at 345 185
\pinlabel {13} [br] at 117 110
\pinlabel {14} [br] at 345 110
\pinlabel {15} [br] at 117 30
\pinlabel {16} [br] at 345 37
\endlabellist
\centering
\includegraphics[scale=.5]{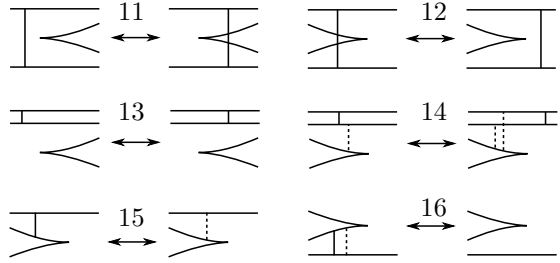}
\caption[MCS moves 11 - 16]{MCS moves 11 - 16. We also allow reflections about the horizontal and vertical axes.}
\label{f:MCS-equiv-3}
\end{figure}

\begin{figure}
\labellist
\small\hair 2pt
\pinlabel {17} [br] at 220 70
\pinlabel {$y_{u_3}$} [br] at 5 120
\pinlabel {$y_{u_2}$} [br] at 5 100
\pinlabel {$y_{u_1}$} [br] at 5 80
\pinlabel {$y_k$} [br] at 5 60
\pinlabel {$y_l$} [br] at 5 40
\pinlabel {$y_{v_1}$} [br] at 5 20
\pinlabel {$y_{v_2}$} [br] at 5 0
\endlabellist
\centering
\includegraphics[scale=.6]{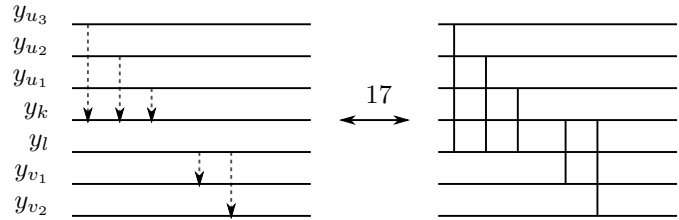}
\caption[MCS moves 17.]{MCS move 17.}
\label{f:MCS-equiv-explosion}
\end{figure}

\begin{prop}
	\label{p:MCS-equiv}
Suppose $\sMCS = (C_1, \df_1) \hdots (C_m, \df_m)$ is an MCS on $\sfront$ and that in the interval $[a,b]$ of the $x$-axis we modify the marked front projection of $\sMCS$ by one of the MCS moves in Figures~\ref{f:MCS-equiv-2}, \ref{f:MCS-equiv-3}, or \ref{f:MCS-equiv-explosion}. Then the resulting marked front projection $\sMCS'$ determines an MCS. In addition, the ordered chain complexes of $\sMCS$ and $\sMCS'$ agree outside of the interval $[a,b]$.
\end{prop}

\begin{proof}
As marked front projections, $\sMCS$ and $\sMCS'$ agree outside of $[a,b]$ and so by Remark~\ref{rem:mcs-defn}~(1) the chain complexes they determine agree to the left of $[a,b]$. The chain complexes to the right of $[a,b]$ depend on the handleslide marks of each MCS and the chain complex to the immediate right of $[a,b]$, so it suffices to show that $\sMCS$ and $\sMCS'$ determine the same chain complex to the immediate right of $[a,b]$. This follows from the matrix equations in Definition~\ref{defn:MCS-maps} and the following matrix equations:

\begin{description}
	\item[Move 1] For $k>l$, $E_{k, l}E_{k, l} = I $.
	\item[Moves 2 - 5] For $k_1 > l_1$ and $k_2 > l_2$ with $k_1 \neq l_2$ and $k_2 \neq l_1$,  $E_{k_1, l_1}E_{k_2, l_2} = E_{k_2, l_2}E_{k_1, l_1} $.
	\item[Move 6] For $a>b>c$, $E_{a, b}E_{b, c} = E_{b, c}E_{a, c}E_{a,b} $.
	\item[Moves 7 and 10] For $k > l$ with $k, l \notin \{i+1,i\}$, $E_{k, l}P_{i+1, i} = P_{i+1,i}E_{k, l} $.
	\item[Moves 8 and 9] For $k > i+1$, $E_{k, i}P_{i+1, i} = P_{i+1,i}E_{k, i+1}$.	
	\item[Moves 15 and 16] Suppose the explicit handleslide mark occurs between strands $k$ and $i+1$. If $E$ (respectively $F$) represents the sequence of implicit handleslide moves in the formula for the differential of the chain complex of $\sMCS$ (respectively $\sMCS'$) occurring after the right cusp, then $F = E E_{k,i+1}$.
\end{description}

In moves 11 - 13, the explicit handleslide mark does not involve the generators of the acyclic subcomplex that is quotiented out, so the chain complexes of $\sMCS$ and $\sMCS'$ to the right of $[a,b]$ are equal. Move 14 is a composition of moves 6, 13, and 15. Indeed, moves 15 and 6 allow us to commute the explicit mark past the implicit mark and then move 13 pushes the explicit mark past the cusp. Finally, the localness of move 17 was detailed in the discussion surrounding Equation~\ref{eq:explosion}.
\end{proof}

\begin{defn}
	\label{defn:MCS-equiv}
We say $\sMCS$ and $\sMCS'$ in $\sFMCS$ are \emph{equivalent}, denoted $\sMCS \sMCSeq \sMCS'$, if there exists a finite sequence $\sMCS = \sMCS_1, \sMCS_2, \hdots, \sMCS_n = \sMCS' \in \sFMCS$ where for each $1 \leq i < n$ the marked front projections of $\sMCS_i$ and $\sMCS_{i+1}$ are related by an MCS move. 

We let $\sFMCSeq$ denote the equivalence classes of $\sFMCS \diagup \sim$ and let $\sbMCS$ denote an equivalence class in $\sFMCSeq$. 
\end{defn}

\subsection{Associating a Normal Ruling to an MCS}

\begin{defn}[\cite{Barannikov1994}]
	\label{defn:simple-complex}
		An ordered chain complex $(C, \df)$ with generators $y_1, \hdots, y_m$ is in \emph{simple form} if:
\begin{enumerate}
	\item For all $i$, either $\df y_i = 0 $ or $\df y_i = y_j$ for some $j$; and
	\item If $\df y_i = \df y_k = y_j$, then $k = i$.
\end{enumerate}
\end{defn}

\begin{lem}[\cite{Barannikov1994}]
	\label{lem:Barannikov}
		Let $(C, \df)$ be an ordered chain complex with generators $y_1, \hdots, y_m$. Then after a series of handleslide moves (see Definition~\ref{defn:MCS-maps}~(1))  we can reduce $(C, \df)$ to simple form. In addition, this simple form is unique. 
\end{lem}

The proof of Lemma~\ref{lem:Barannikov} is given in Lemma~2 of \cite{Barannikov1994} where Barannikov uses an inductive construction to find the simple forms of each of the subcomplexes $(C_k, \df)$ generated by $y_1, \hdots, y_k$ for $k \leq m$.

\begin{rem}
\label{rem:simple-form} 
The following two observations follow directly from Lemma~\ref{lem:Barannikov}.
\begin{enumerate}
	\item If $(C, \df)$ and $(C', \df')$ are chain isomorphic by a handleslide move, i.e. $\sdmatrix = E_{k,l} \sdmatrix' E_{k,l}^{-1}$, then they have the same simple form. 
	\item For consecutive generators in $(C, \df)$, a handleslide move does not change the value of $\langle \df y_{j+1} | y_j \rangle$. Thus, if $\langle \df y_{j+1} | y_j \rangle = 1$, then $\df y_{j+1} = y_j$ in the simple form for $(C, \df)$.
\end{enumerate}
\end{rem}

\begin{defn}
	\label{defn:pairing}
	Suppose $(C, \df)$ is an ordered chain complex with generators $y_1, \hdots, y_m$ and trivial homology. By Lemma~\ref{lem:Barannikov}, there exists a fixed-point free involution $\tau : \{1,\hdots,m\} \to \{1,\hdots,m\}$ so that: in the simple form of $(C, \df)$, for all $i$ either $\df y_i = y_{\tau(i)}$ or $\df y_{\tau(i)} = y_{i}$. We call $\tau$ the \emph{pairing} of $(C, \df)$. 
\end{defn}	

The following lemma assigns a normal ruling to an MCS. In Section 12.4 of \cite{Chekanov2005}, Chekanov and Pushkar prove this lemma for a generating family using the language of pseudo-involutions.

\begin{lem}
	\label{lem:MCS-ruling}
		The pairings determined by the simple forms of the ordered chain complexes in an MCS $\sMCS = (C_1, \df_1) \hdots (C_m, \df_m)$ determine a graded normal ruling $\sgnr_{\sMCS}$ on $\sfront$.  
\end{lem}

\begin{proof}
Each $(C_j, \df_j)$ has trivial homology and, thus, a pairing $\tau_i$. Given a pair of consecutive singularities $x_i$ and $x_{i+1}$ of $\sfront$, we define a pairing on the strands of the tangle $\sfront \cap (x_i,x_{i+1})$ using the pairing of a chain complex of $\sMCS$ between $x_i$ and $x_{i+1}$. There may be several chain complexes of $\sMCS$ between $x_i$ and $x_{i+1}$. However, these chain complexes differ by handleslide moves and so by Remark~\ref{rem:simple-form}~(1) their pairings are identical. Remark~\ref{rem:simple-form}~(2) justifies that two strands entering a cusp are paired. 

It remains to check that the switched crossings are graded and normal. Normality follows from Lemma~4 in \cite{Barannikov1994}. The two strands meeting at a switch exchange companion strands and the Maslov potentials of two companioned strands differ by $1$. Hence, two strands meeting at a switch must have the same Maslov potential and so the grading of a switched crossing is $0$.
\end{proof}

If $\sMCS_1$ and $\sMCS_2$ are equivalent by a single MCS move, then by Proposition~\ref{p:MCS-equiv} the chain complexes in $\sMCS_1$ and $\sMCS_2$ are equal outside of a neighborhood of the MCS move. Thus, we may use chain isomorphic chain complexes in $\sMCS_1$ and $\sMCS_2$ to determine the graded normal rulings $\sgnr_{\sMCS_1}$ and $\sgnr_{\sMCS_1}$. As a consequence we have the following. 

\begin{prop}
	\label{prop:gnr-uniq}
If $\sMCS_1 \sim \sMCS_2$ then $\sgnr_{\sMCS_1} = \sgnr_{\sMCS_2}$. 
\end{prop}

We let $\sgnr_{\sbMCS} \in N(\sfront)$ denote the graded normal ruling associated to the MCS class $\sbMCS$.

\subsection{MCSs with simple births}

A \emph{simple birth} in an MCS is a birth with no implicit handleslides. In the language of Definition~\ref{defn:MCS-maps}~(3), this says $E = I$. From now on we will restrict our attention to MCSs with only simple births and MCSs equivalence moves that do not involve implicit handleslide marks at births. The connection between MCSs and augmentations is clearer under this assumption. For the sake of simplicity, the results stated in the Introduction were given in terms of $\sFMCSeq$. The corresponding results proven in the remainder of the article are given in terms of MCSs with simple births. In light of Proposition~\ref{p:simple-mcs-eq} below, there is no harm in this ambiguity. 

\begin{defn}
We define $\sDMCS$ to be set of MCSs in $\sFMCS$ which have only simple births. Two MCSs $\sMCS, \sMCS' \in \sDMCS$ are equivalent if they are equivalent in $\sFMCS$ by a sequence of MCS moves $\sMCS = \sMCS_1, \sMCS_2, \hdots, \sMCS_n = \sMCS'$ such that $\sMCS_i \in \sDMCS$ for all $i$. 

We let $\sDMCSeq$ denote the equivalence classes of $\sDMCS \diagup \sim$ and let $\sbMCS$ denote an equivalence class in $\sDMCSeq$. 
\end{defn}

\begin{prop}
	\label{p:simple-mcs-eq}
The inclusion map $i : \sDMCS \to \sFMCS$ determines a bijection $\widehat{i} : \sDMCSeq \to \sFMCSeq$ by $\sbMCS \mapsto [i(\sMCS)]$.
\end{prop}

We sketch the argument for injectivity, but leave the details to the reader. Suppose $[i(\sMCS)] = [i(\sMCS')]$ in $\sFMCS$ and $\mathcal{Z} = \{ i(\sMCS) = \sMCS_1, \sMCS_2, \hdots, \sMCS_{n-1}, i(\sMCS') = \sMCS_n \}$ is a sequence of MCS moves. Since $i(\sMCS)$ and $i(\sMCS')$ have simple births, each implicit handleslide appearing at a birth in an MCS in $\mathcal{Z}$ is introduced and eventually eliminated using one of moves 14, 15, or 16. We may modify the sequence of MCS moves in $\mathcal{Z}$ inductively beginning with $\sMCS_1$ and $\sMCS_2$ so that $\sMCS_i \in \sDMCS$ for all $i$. We do this by replacing occurrences of move 16 at lefts cusps with move 1 and removing the occurrences of move 15 at left cusps. As a result, we must also change the occurrences of move 14 to a composition of moves 6 and 13 and possibly include additional moves 2 - 5. The map $\widehat{i}$ is surjective since MCS move 15 may be used to find a representative of $\sbMCS \in \sFMCSeq$ with simple births. 

\begin{rem}
If we assume all births are simple, then we do not need to indicate the implicit handleslide marks at deaths in the associated marked front projection, since they can be determined by reconstructing the ordered chain complexes as in Remark~\ref{rem:mcs-defn}~(1). Thus, from now on, marked front projections will include only explicit handleslide marks and we will no longer indicate implicit handleslide marks in the MCS moves. 
\end{rem}

\section{Chain Homotopy Classes of Augmentations}
\label{ch:Chain-Homotopy}

We begin by defining DGA morphisms and chain homotopy on arbitrary DGAs and then restrict to the case of CE-DGAs. Definition~\ref{defn:dga-morphism} and Lemma~\ref{lem:chain-htpy} follow directly from Section 2.3 of \cite{Kalman2005}.

\begin{defn}
	\label{defn:dga-morphism}
	Let $(\aac, \df)$ and $(\bbc, \df')$ be differential graded algebras over a commutative ring $R$ and graded by a cyclic group $\Gamma$. A \emph{DGA morphism} $\varphi : (\aac, \df) \to (\bbc, \df')$ is a grading-preserving algebra homomorphism satisfying $\varphi \circ \df = \df' \circ \varphi$. Given two DGA morphisms $\varphi, \psi : (\aac, \df) \to (\bbc, \df')$, a \emph{chain homotopy between $\varphi$ and $\psi$} is a linear map $H: (\aac, \df) \to (\bbc, \df') $ satisfying:
	
	\begin{enumerate}
		\item For all $a \in \aac$, $|H(a)| = |a| + 1$,
		\item For all $a, b \in \aac$, $H(ab) = H(a) \psi(b) + (-1)^{|a|}\varphi(a) H(b)$, and
		\item $\varphi - \psi = H \circ \df + \df' \circ H$ .
	\end{enumerate}
We refer to Condition 2 as the \emph{derivation product property}. 	
\end{defn}

Augmentations are DGA morphisms between the CE-DGA and the DGA whose only nonzero chain group is a copy of $\zz_2$ in grading 0. By Lemma~2.18 in \cite{Kalman2005}, a chain homotopy between augmentations is determined by its action on the generators of the CE-DGA.

\begin{lem}
	\label{lem:chain-htpy}
	Let $(\aac (\sLagr), \df)$ be the CE-DGA of the Lagrangian projection $\sLagr$ with generating set $Q$ and let $\epsilon_1, \epsilon_2 \in \sAugL$. If a map $H: Q \to \zz_2$ has support on generators of grading $-1$, then $H$ can be uniquely extended by linearity and the derivation product property to a map, also denoted $H$, on all of $(\aac (\sLagr), \df)$ that has support on elements of grading $-1$. Moreover, if the extension satisfies $\epsilon_1(q) - \epsilon_2(q) = H \circ \df(q)$ for all $q \in Q$, then $\epsilon_1 - \epsilon_2 = H \circ \df$ on all of $\aac (\sLagr)$ and, thus, $H$ is a chain homotopy between $\epsilon_1$ and $\epsilon_2$.
\end{lem}

We say $\epsilon_1$ and $\epsilon_2$ are \emph{chain homotopic} and write $\epsilon_1 \schequiv \epsilon_2$ if a chain homotopy $H$ exists between $\epsilon_1$ and $\epsilon_2$. 

\begin{lem}[\cite{YvesFelix1995}]
	\label{lem:ch-equiv-rel}
	The relation $\schequiv$ is an equivalence relation.
\end{lem}

  We let $Aug^{ch}(\sLagr)$ denote $Aug(\sLagr) \diagup \schequiv$ and $[\saug]$ denote an augmentation class in $Aug^{ch}(\sLagr)$. The following lemma will be necessary for our later work connecting augmentation classes and MCS classes. 
  
\begin{lem}
	\label{lem:aug-ch}
	Suppose $L$ and $L'$ are Lagrangian projections with CE-DGAs $(\aac(\sLagr), \df)$ and $(\aac(\sLagr'), \df')$ respectively. Let $f : (\aac(\sLagr), \df) \to (\aac(\sLagr'), \df')$ be a DGA morphism. Then $f$ induces a map $F : Aug^{ch}(\sLagr') \to Aug^{ch}(\sLagr)$ by $[\saug] \mapsto [\saug \circ f]$. If $g : (\aac(\sLagr), \df) \to (\aac(\sLagr'), \df')$ is also a DGA morphism and $\sH$ is a chain homotopy between $f$ and $g$, then $F = G$.
\end{lem}

\begin{proof}
We begin by checking $F$ is well-defined, i.e. $\saug \circ f$ is an augmentation and $\saug \schequiv \saug'$ implies $\saug \circ f \schequiv \saug' \circ f$. Since $f$ is a degree preserving chain map and $\saug \circ \df' = 0$, $\saug \circ f (q) = 1$ implies $|q| = 1$ and $\saug \circ f \circ \df = \saug \circ \df' \circ f = 0$. Therefore $\saug \circ f$ is an augmentation. If $\sH$ is a chain homotopy between $\saug$ and $\saug'$, then $\sH \circ f$ is a chain homotopy between $\saug \circ f$ and $\saug' \circ f$.

Suppose $g : (\aac(\sLagr), \df) \to (\aac(\sLagr'), \df')$ is also a DGA morphism and $\sH$ is a chain homotopy between $f$ and $g$. The map $\sH' = \saug \circ \sH : (\aac(\sLagr), \df) \to \zz_2$ is a chain homotopy between $\saug \circ f$ and $\saug \circ g$. Thus, $F = G$ since

	\begin{equation*}
	F([\saug]) = [\saug \circ f] = [\saug \circ g] = G([\saug]).
	\end{equation*}	
	
\end{proof}
  
In the case of the CE-DGA, the number of chain homotopy classes is a Legendrian invariant.

\begin{prop}
	\label{prop:chain-homotopy-inv}
	If $L$ and $L'$ are Lagrangian projections of Legendrian isotopic knots $\sK$ and $K'$, then there is a bijection between $Aug^{ch}(\sLagr)$ and $Aug^{ch}(\sLagr')$. Thus, $|Aug^{ch}(\sLagr)|$ is a Legendrian invariant of the Legendrian isotopy class of $\sK$.
\end{prop}

We sketch the proof of Proposition~\ref{prop:chain-homotopy-inv}. Since $\sK$ and $\sK'$ are Legendrian isotopic, $(\aac(\sLagr), \df)$ and $(\aac(\sLagr'), \df')$ are stable tame isomorphic. Thus, we need only consider the case of a single elementary isomorphism and a single stabilization. Suppose $(S_{i} (\aac), \df)$ is an index $i$ stabilization of a DGA $(\aac, \df)$. If $i \neq 0$, then $\saug \in Aug(\aac)$ extends uniquely to an augmentation $\tilde{\saug} \in Aug(S_i (\aac, \df))$ by sending both $e_1$ and $e_2$ to $0$. If $i = 0$, then $|e_1| = 0$ and $\saug \in Aug(\aac)$ extends to two different augmentations in $S_0 (\aac, \df)$; the first, denoted $\tilde{\saug}$, sends both $e_1$ and $e_2$ to $0$ and the second sends $e_1$ to 1 and $e_2$ to $0$. However, these two augmentations are chain homotopic by the chain homotopy $\sH$ that sends $e_2$ to 1 and all of the other generators to 0. Regardless of $i$, the map defined by $ [\saug] \mapsto [\tilde{\saug}]$ is the desired bijection. Finally, if $\phi: (\aac, \df) \to (\bbc, \df')$ is an elementary isomorphism between two DGAs, then the map $\Phi : Aug^{ch}(\bbc) \to Aug^{ch}(\aac)$ defined by $ [\saug'] \mapsto [\saug' \circ \phi]$ is the desired bijection.

In the next section, we will concentrate on understanding the chain homotopy classes of a fixed Lagrangian projection. By using a procedure called ``adding dips,'' we may reformulate the augmentation condition $\saug \circ \df$ and the chain homotopy condition $\saug_1 - \saug_2 = \sH \circ \df$ as a system of matrix equations. Understanding the chain homotopy classes of augmentations then reduces to understanding solutions to these matrix equations.

\section{Dipped Resolution Diagrams}
\label{ch:Dipped-Diagrams}

In \cite{Fuchs2003}, Fuchs modifies the Lagrangian projection of a Legendrian knot so that the differential in the CE-DGA is easy to compute, at the cost of increasing the number of generators. This technique has proved to be very useful; see \cite{Fuchs2004,Fuchs2008,Ng2006,Sabloff2005}. We will use the version of this philosophy implemented in \cite{Sabloff2005}.

\subsection{Adding dips to a Ng resolution diagram}
\label{sec:adding-dips}

By performing a series of Lagrangian Reidemeister type II moves on $\sNgres$, we can limit the types of convex immersed polygon appearing in the differential of the CE-DGA. A \emph{dip}, denoted $\sdip$, is the collection of crossings created by performing type II moves on the strands of $\sNgres$ as in Figure~\ref{f:dip-creation}. At the location of the dip, label the strands of $\sNgres$ from bottom to top with the integers $1, \hdots, n$. For all $k > l$ there is a type II move that pushes strand $k$ over strand $l$. If $k < i$ then $k$ crosses over $l$ before $i$ crosses over any strand. If $l < j < k$ then $k$ crosses over $l$ before $k$ crosses over $j$. The notation $(k,l) \prec (i,j)$ denotes that $k$ crosses over $l$ before $i$ crosses over $j$. 

\begin{figure}
\labellist
\small\hair 2pt
\pinlabel {4} [tl] at 203 116
\pinlabel {3} [tl] at 203 96
\pinlabel {2} [tl] at 203 76
\pinlabel {1} [tl] at 203 56
\pinlabel {$b^{3,1}$} [tl] at 211 22
\pinlabel {$a^{4,2}$} [tl] at 315 22
\endlabellist
\centering
\includegraphics[scale=.5]{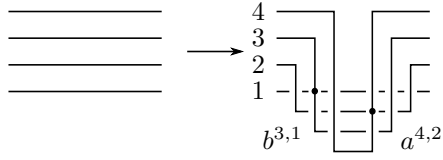}
\caption[Adding a dip to $\sNgres$.]{Adding a dip to $\sNgres$ and labeling the resulting crossings. Crossings $b^{3,1}$ and $a^{4,2}$ are indicated.}
\label{f:dip-creation}
\end{figure}

The \emph{$a$-lattice} of $\sdip$ is comprised of all crossings to the right of the vertical line of symmetry and the \emph{$b$-lattice} is comprised of all crossings to the left. The label $a^{k,l}$ denotes the crossing in the $a$-lattice of strand $k$ over strand $l$ where $k > l$; see Figure~\ref{f:dip-creation}.  The label $b^{k,l}$ is similarly defined. The gradings of the crossings in a dip are computed by $|b^{k,l}| = \smu(k) - \smu(l)$ and $|a^{k,l}| = |b^{k,l}| - 1$ where $\mu$ is a Maslov potential on $\sfront$.

The type II moves may be arranged so when strand $k$ is pushed over strand $l$, a crossing $q$ has height less than $b^{k,l}$ if and only if $q$ appears to the left of the dip, or $q = b^{r,s}$ or $q = a^{r,s}$ where $r-s \leq k-l$. Similarly, a crossing $q$ has height less than $a^{k,l}$ if and only if $q$ appears to the left of the dip, or $q = b^{r,s}$ or $q = a^{r,s}$ where $r-s \leq k-l$.

\begin{figure}
\centering
\includegraphics[scale=.6]{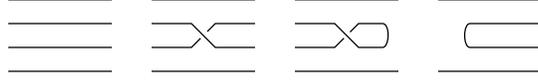}
\caption[Inserts in a sufficiently dipped diagram.]{The four possible inserts in a sufficiently dipped diagram $\sNgres^d$; (1) parallel lines, (2) a single crossing, (3) a resolved right cusp, and (4) a resolved left cusp. In each case, any number of horizontal strands may exist.}
\label{f:inserts}
\end{figure}

\begin{defn}
	\label{defn:dipped-diagram}
	Given a Legendrian knot $\sK$ with front projection $\sfront$ and Ng resolution $\sNgres$, a \emph{dipped diagram}, denoted $\sNgresD$, is the result of adding some number of dips to $\sNgres$. We require that during the process of adding dips, we not allow a dip to be added between the crossings of an existing dip and we not add dips in the loop of a resolved right cusp.
\end{defn}

We let $ \sdip_1$, $\hdots$, $\sdip_{m}$ denote the $m$ dips of $\sNgresD$, ordered from left to right with respect to the $x$-axis. For consecutive dips $\sdip_{j-1}$ and $\sdip_j$ we let $I_j$ denote the tangle between $\sdip_{j-1}$ and $\sdip_j$. We define $I_1$ to be the tangle to the left of $\sdip_1$ and $I_{m+1}$ to be the tangle to the right of $\sdip_m$. We call $I_1, \hdots, I_{m+1}$ the \emph{inserts} of $\sNgresD$. 

\begin{defn}
	\label{defn:suff-dipped}
	We say a dipped diagram $\sNgresD$ is \emph{sufficiently dipped} if each insert $I_1, \hdots, I_{n+1}$ is isotopic to one of those in Figure~\ref{f:inserts}.
\end{defn}

\begin{rem}
\label{rem:sliding-dips}
If we slide a dip of $\sNgresD$ left or right without passing it by a crossing, resolved cusp or another dip, then the resulting dipped diagram is topologically identical to the original. In particular, they determine the same CE-DGA. Thus, we will not distinguish between two dipped diagrams that differ by such a change.
\end{rem}

\subsection{The CE-DGA on a sufficiently dipped diagram}
\label{sec:boundary-dipped}

A convex immersed polygon cannot pass completely through a dip. Therefore, we may calculate the boundary map of $(\aac(\sNgresD), \df)$ by classifying convex immersed polygons that lie between consecutive dips or that sit entirely within a single dip. This classification is aided by our understanding of the heights of the crossings in $\sNgresD$. 

Let $\sNgresD$ be a sufficiently dipped diagram of $\sNgres$ with dips $\sdip_1, \hdots, \sdip_m$ and inserts $I_1, \hdots, I_{m+1}$. Let $q_i, \hdots, q_M$ and $z_j, \hdots, z_N$ denote the crossings found in the inserts of type (2) and (3) respectively. The subscripts on $q$ and $z$ correspond to the subscripts of the inserts in which the crossings appear. 

\begin{rem}
	\label{rem:warning}
In Section~\ref{sec:adding-dips}, we labeled the crossings in the $a$-lattice and $b$-lattice of a dip using a labeling of the strands at the location of the dip. When calculating $\df$ on the crossings in $\sdip_j$, it will make our computations clearer and cleaner if we change slightly the labeling we use for the strands in $\sdip_{j-1}$ and $\sdip_{j}$. 

If $I_j$ is of type (3) and the right cusp occurs between strands $i+1$ and $i$, then we will label the strands of $\sdip_{j-1}$ by $1, \hdots, n$ and the strands of $\sdip_{j}$ by $1, \hdots, i-1, i+2, \hdots, n$. If $I_j$ is of type (4) and the left cusp occurs between strands $i+1$ and $i$, then we will label the strands of $\sdip_{j-1}$ by $1, \hdots, i-1, i+2, \hdots, n$ and the strands of $\sdip_{j}$ by $1, \hdots, n$. If $I_j$ is of type (1) or (2), we do not change the labeling of the strands. These changes are local. In particular, the labeling of the strands in $\sdip_{j}$ may vary depending on whether we are calculating $\df$ on the crossings in $\sdip_j$ or $\sdip_{j+1}$. 
\end{rem}

Fix an insert $I_{j}$, $j \neq m+1$, and label the strands of $\sdip_j$ and $\sdip_{j-1}$ as indicated in Remark~\ref{rem:warning}. For $i \in \{j, j-1 \}$, $A_i$ (resp. $B_i$) is the strictly lower triangular square matrices with rows and columns labeled the same as the strands of $D_i$ and entries given by $(A_i)_{k,l} = a^{k,l}_i$ (resp. $(B_i)_{k,l} = b^{k,l}_i$) for $k > l$. The following formulae define a matrix $\widetilde{A}_{j-1}$ using the $a_{j-1}$-lattice and possible crossings in $I_j$. The subscript ${j-1}$ has been left off the $a$'s to make the text more readable.

\begin{enumerate}
\item For $j = 1$, $\widetilde{A}_{j-1} = \sHSM_{2,1}$ where $\mbox{dim}(\sHSM_{2,1}) = 2$.
\item Suppose $I_j$ is of type (1). Then $\widetilde{A}_{j-1} = A_{j-1}$.
	\item Suppose $I_j$ is of type (2), $A_j$ has dimension $n$ and $q_j$ involves strands $i+1$ and $i$. Then $\widetilde{A}_{j-1}$ is the $n \times n$ matrix with $(u,v)$ entry $\widetilde{a}_{j-1}^{u,v}$ defined by:
	
				\begin{equation*}
				\begin{array}{l}
				\displaystyle  \widetilde{a}^{u,v} = \begin{cases}
				    a^{u,v} & u,v \notin \{i,i+1\} \\
				    0 & u=i+1, v=i \\
				  \end{cases} \\
				\displaystyle \widetilde{a}^{i+1,v} = a^{i,v} \\
				\displaystyle \widetilde{a}^{i,v} = a^{i+1,v} + q_j a^{i,v} \\
				\displaystyle \widetilde{a}^{u,i} = a^{u,i+1} \\
				\displaystyle \widetilde{a}^{u,i+1} = a^{u,i} + a^{u,i+1} q_j. 
				\end{array} 
				\end{equation*}	 
				
	\item Suppose $I_j$ is of type (3), $A_j$ has dimension $n$, and $z_j$ involves strands $i+1$ and $i$. Then $\widetilde{A}_{j-1}$ is an $n \times n$ matrix with rows and columns numbered $1,\hdots, i-1, i+2, \hdots, n+2$ and $(u,v)$ entry $\widetilde{a}_{j-1}^{u,v}$ defined by:
	
				\begin{equation*}  
				  \widetilde{a}^{u,v} = \begin{cases}
				    a^{u,v} & u < i \text{ or } v > i+1 \\
				    a^{u,v} + a^{u,i}a^{i+1,v} + a^{u,i+1} z_j a^{i+1,v} & u > i+1 > i > v. \\
				    \hspace{10 mm} + a^{u,i} z_j a^{i,v} + a^{u,i+1} z_j z_j a^{i,v} & 
				  \end{cases}
				\end{equation*} 				
	
	\item Suppose $I_j$ is of type (4), $A_j$ has dimension $n$ and the resolved left cusp strands are $i+1$ and $i$. Then $\widetilde{A}_{j-1}$ is the $n \times n$ matrix with $(u,v)$ entry  $\widetilde{a}_{j-1}^{u,v}$ defined by:
	
				\begin{equation*}  
				  \widetilde{a}^{u,v} = \begin{cases}
				    a^{u,v} & u,v \notin \{ i,i+1 \} \\
				    1 & u = i+1, v = i \\
				    0 & \mbox{otherwise.}
				  \end{cases}
				\end{equation*}	 

\end{enumerate}

We are now in a position to give matrix equations for the boundary map of $(\aac(\sNgresD), \df)$.

\begin{lem}
	\label{lem:dipped-df}
	The CE-DGA boundary map $\df$ of a sufficiently dipped diagram of a Ng resolution $\sNgres$ is computed by: 

\begin{enumerate}
	\item For a crossings $q_s$ between strands $i+1$ and $i$, $\df q_s = a_{s-1}^{i+1,i}$;
	\item For a crossings $z_r$ between strands $i+1$ and $i$, $\df z_r = 1 + a_{r-1}^{i+1,i}$, where $z_r$ is a crossing between strands $i+1$ and $i$;
	\item $\df A_j = A_j^2$ for all $j$; and
	\item $\df B_j = (I + B_j)A_j + \widetilde{A}_{j-1}(I + B_j)$ for all $j$
\end{enumerate}

\noindent where $I$ is the identity matrix of appropriate dimension and the matrices $\df A_j$ and $\df B_j$ are the result of applying $\df$ to $A_j$ and $B_j$ entry-by-entry.
\end{lem}

This result appears in \cite{Fuchs2008} in the language of ``splashed diagrams.'' In the next three sections we justify these formulae.

\subsubsection{Computing $\df$ on $q_s$ and $z_r$}
 
The height ordering on the crossings of $\sNgresD$ ensures that $q_s$ is the right-most convex corner of any non-trivial disk in $\df q_s$. Thus, the disk in Figure~\ref{f:dip-crossing-right-cusp} is the only disk contributing to $\df q_s$. The case of $\df z_r$ is similar.

\begin{figure}
\centering
\includegraphics[scale=.65]{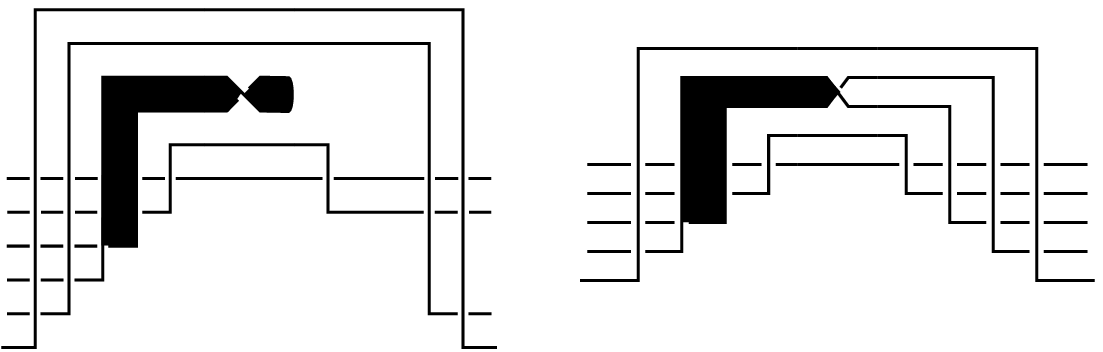}
\caption[Disks contributing to $\df z_r$ and $\df q_s$.]{Disks contributing to $\df z_r$ and $\df q_s$.}
\label{f:dip-crossing-right-cusp}
\end{figure}
  
\subsubsection{Computing $\df$ on the $A_j$ lattice}

The combinatorics of the dip $D_j$ along with the height ordering on the crossings of $\sNgresD$ require that disks in $\df a^{k,l}_j$ sit within the $A_j$ lattice. In addition, a convex immersed polygon sitting within the $A_j$ lattice must have exactly two negative convex corners; see Figure~\ref{f:d-a-d} (a). From this we compute $\df a^{k,l}_j = \sum_{l<i<k} a^{k,i}_j a^{i,l}_j$, which is the $(k,l)$-entry in the matrix $A_j^2$. 

\begin{figure}
\labellist
\small\hair 2pt
\pinlabel {(a)} [tl] at 100 0
\pinlabel {(b)} [tl] at 375 0
\endlabellist
\centering
\includegraphics[scale=.35]{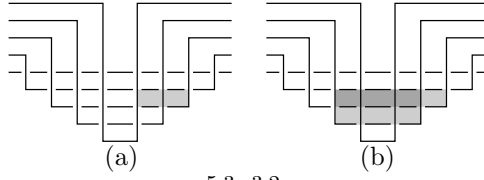}
\caption{(a) The disk $a^{5,3}_j a^{3,2}_j$ appearing in $\partial a^{5,2}_j$. (b) The disks $b^{4,3}_j a^{3,2}_j$ and $a^{4,2}_j $ appearing in $R(\df b^{4,2}_j)$.}
\label{f:d-a-d}
\end{figure}

\subsubsection{Computing $\df$ on the $B_j$ lattice}

The disks in $\df b^{k,l}_j$ are of two types. The first type sit within $D_j$ and include the lower right corner of $b^{k,l}_j$ as their positive convex corner. We denote these disks by $R(\df b^{k,l}_j)$. Figure~\ref{f:d-a-d} (b) shows the two types of convex immersed polygons found in $R(\df b^{k,l}_j)$. From this, we compute:

\begin{equation*}
R(\df b^{k,l}_j) = a^{k,l}_j + \sum_{l<i<k} b^{k,i}_j a^{i,l}_j.
\end{equation*} 

\noindent Disks of the second type include the upper left corner of $b^{k,l}_j$ as their right-most convex corner. Thus these disks sit between the $B_j$ lattice and the $A_{j-1}$ lattice. We let $L(\df b^{k,l}_j)$ denote these disks. The possible disks appear in Figures~\ref{f:d-l-d-d}, \ref{f:d-lc-d-d}, \ref{f:d-c-d-d}, and \ref{f:d-rc-d-d}. Computing $L(\df b^{k,l}_j)$ requires understanding which of these disks appear with a positive convex corner at $b^{k,l}_j$. The entries in the matrix $\widetilde{A}_{j-1}$ are defined to encode these disks. In fact, for a fixed $b^{k,l}_j$:

\begin{equation*}
L(\df b^{k,l}_j) = \widetilde{a}^{k,l}_{j-1} + \sum_{l<i<k} \widetilde{a}^{k,i}_{j-1} b^{i,l}_j.
\end{equation*} 

\noindent Regardless of $I_j$'s type, $R(\df b^{k,l}_j) + L(\df b^{k,l}_j)$ is the $(k,l)$-entry in the matrix $(I + B_j)A_j + \widetilde{A}_{j-1}(I + B_j)$.

\begin{figure}
     \centering
     \subfigure[]{
					\labellist
					\small\hair 2pt
					\pinlabel {$a^{4,2}_{j-1}$} [tl] at 84 18
					\pinlabel {$a^{4,3}_{j-1} b^{3,2}_j$} [tl] at 270 18
					\endlabellist     
          \label{f:d-l-d-d}
          \includegraphics[scale=.5]{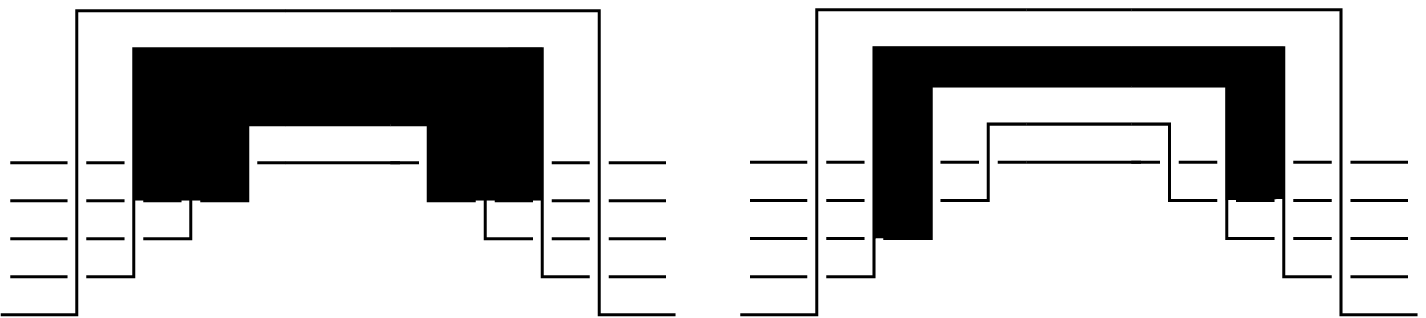}}\\    
     \subfigure[]{
     			\labellist
					\small\hair 2pt
					\pinlabel {$1$} [tl] at 84 18
					\pinlabel {$b^{3,1}_j$} [tl] at 290 18
					\endlabellist
          \label{f:d-lc-d-d}
          \includegraphics[scale=.5]{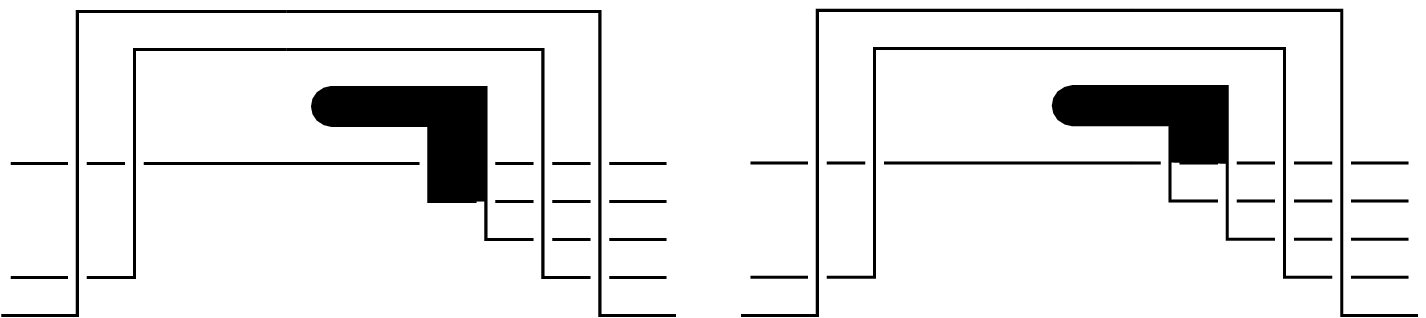}}
     \caption[Disk contributing to $L(\df B_j)$ when $I_{j}$ is of type (1) or (4).]{Disks contributing to $L(\df b^{k,l}_j)$ when $I_{j}$ is of type (1) or (4). The disks in (a) may occur in an insert of any type.}
\end{figure}

\begin{figure}
\centering
\includegraphics[scale=.65]{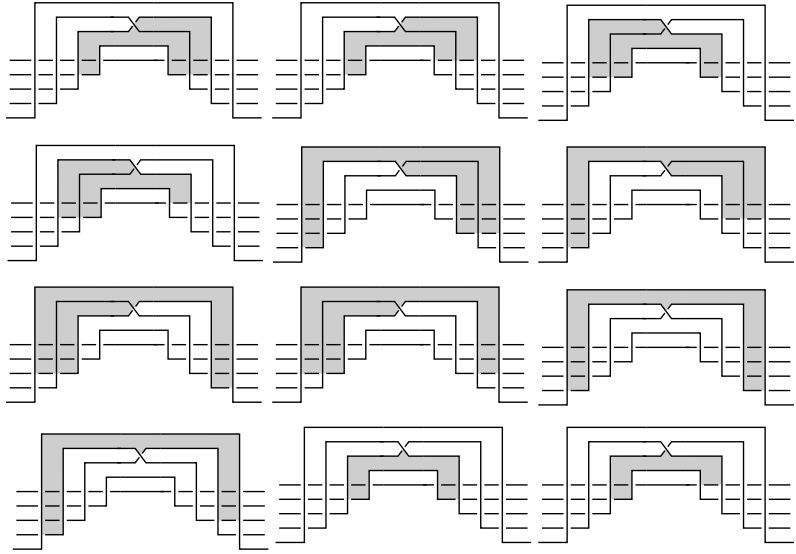}
\caption[Disks contributing to $L(\df B_j)$ when $I_{j}$ is of type (2).]{Disks contributing to $L(\df B_j)$ when $I_{j}$ is of type (2).}
\label{f:d-c-d-d}
\end{figure}

\begin{figure}
\centering
\includegraphics[scale=.6]{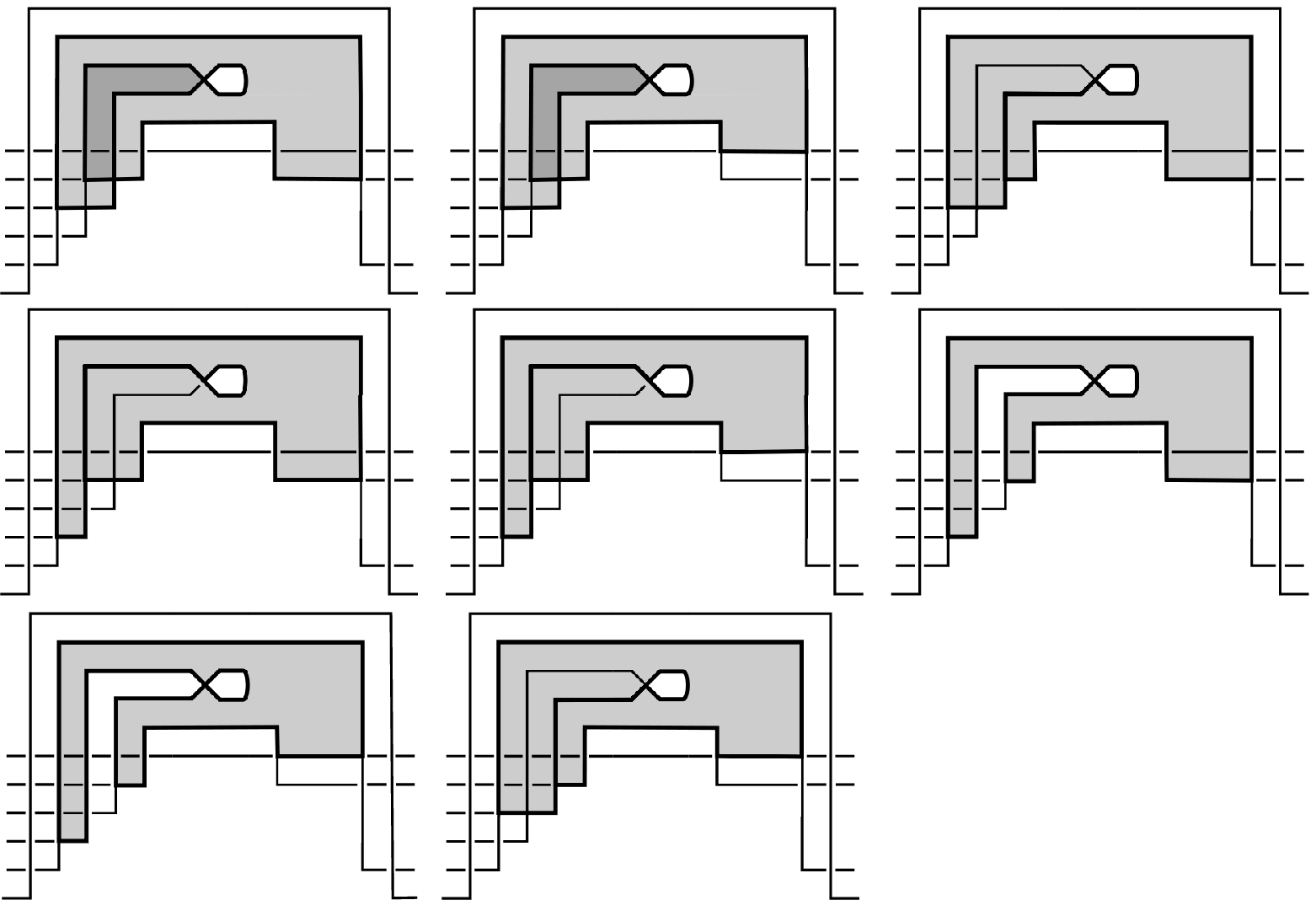}
\caption[Disks contributing to $L(\df B_j)$ when $I_{j}$ is of type (3).]{Disks contributing to $L(\df B_j)$ when $I_{j}$ is of type (3). The boundary of each disk is highlighted.}
\label{f:d-rc-d-d}
\end{figure}

In Section~\ref{sec:tracking-augs}, we study the extension of augmentations across individual type II moves and then across dips and, thus, understand the connections between augmentations on $\sNgres$ and augmentations on a sufficiently dipped diagram $\sNgresD$. In Section~\ref{ch:MCS-Aug}, we use an argument from \cite{Fuchs2008} to associate an augmentation of $\sNgresD$ to an MCS of $\sfront$.

\subsection{Extending an augmentation across a dip}
\label{sec:tracking-augs}

If the Lagrangian projection $\sLagr_2$ is obtained from $\sLagr_1$ by a single type II move introducing new crossings $a$ and $b$ with $|b| = i$ and $|a| = i-1$, then there exists a stable tame isomorphism from $(\aac (\sLagr_2), \df)$ to $(\aac (\sLagr_1), \df')$. As we demonstrated in Proposition~\ref{prop:chain-homotopy-inv}, such a stable tame isomorphism induces a map from $Aug(\sLagr_1)$ to $Aug(\sLagr_2)$. By understanding this map we may keep track of augmentations as we add dips to $\sNgres$. This careful analysis will allow us in Section~\ref{ch:MCS-Aug} to assign an MCS class of $\sfront$ to an augmentation class of $\sNgres$ and, hence, prove the surjectivity of the map $\widehat{\Psi} : \sDMCSeq \to Aug^{ch}(\sNgres)$. 

The stable tame isomorphism from $(\aac (\sLagr_2), \df)$ to $(\aac (\sLagr_1), \df')$ was first explicitly written down in \cite{Chekanov2002a}. We will use the formulation found in \cite{Sabloff2005}. The stable tame isomorphism is a single stabilization $S_i$ introducing generators $\alpha$ and $\beta$ followed by a tame isomorphism $\psi$. The tame isomorphism $\psi: (\aac(\sLagr_2), \df) \to S_i(\aac (\sLagr_1), \df')$ is defined as follows. Let $\{x_1, \ldots, x_N\}$ denote the generators with height less than $h(a)$ and labeled so that $h(x_1) < \hdots < h(x_N)$. Let $\{y_1, \ldots, y_M\}$ denote the generators with height greater than $h(b)$ and labeled so that $h(y_1) < \hdots < h(y_M)$. Recall that \df\ lowers height, so we may write $\df b$ as $\df b = a + v$ where $v$ consists of words in the letters $x_1, \ldots, x_N$. We begin by defining a vector space map $H$ on $S_i(\aac (\sLagr_1), \df')$ by:

\begin{equation*} 
  H(w) = \begin{cases}
    0 & w \in \aac (\sLagr_1) \\
    0 & w=Q\beta R \quad \text{with\ } Q \in \aac (\sLagr_1), R \in S_i(\aac (\sLagr_1)) \\
    Q\beta R & w=Q\alpha R \quad \text{with\ } Q \in \aac (\sLagr_1), R \in S_i(\aac (\sLagr_1)).
  \end{cases}
\end{equation*} 

We build $\psi$ up from a sequence of maps $\psi_i: (\aac(\sLagr_2), \df) \to S_i(\aac (\sLagr_1), \df')$ for $0 \leq i \leq M$: 

\begin{equation*}
  \psi_0 (w) = \begin{cases}
    \beta & w = b \\
    \alpha + v & w = a \\
    w & \text{otherwise}
  \end{cases}
\end{equation*}

and

\begin{equation*}
  \psi_i (w) = \begin{cases}
    y_i + H \circ \psi_{i-1}(\df y_i) & w=y_i \\
    \psi_{i-1}(w) & \text{otherwise.}
  \end{cases}
\end{equation*}

\noindent If we extend the map $\psi_M$ by linearity from a map on generators to a map on all of $(\aac(\sLagr_2), \df)$, then the resulting map $\psi : (\aac(\sLagr_2), \df) \to S_i(\aac (\sLagr_1), \df')$ is the desired DGA isomorphism; see \cite{Chekanov2002a,Sabloff2005}.

If $i \neq 0$, then an augmentation $\saug \in Aug(\sLagr_1)$ is extended to an augmentation $\saug \in Aug(S_i (\sLagr_1))$ by $\saug(\beta) = \saug(\alpha) = 0$. If $i = 0$, then we may choose to extend $\saug \in Aug(\sLagr_1)$ either by $\saug(\beta) = \saug(\alpha) = 0$ or by $\saug(\beta) = 1$ and $\saug(\alpha) = 0$. Given the tame isomorphism $\psi$ defined above, we would like to understand how our choice of an extension from  $Aug(\sLagr_1)$ to $Aug(S_i (\sLagr_1))$ affects the induced map from $Aug(\sLagr_1)$ to $Aug(\sLagr_2)$. Recall $\psi : (\aac(\sLagr_2), \df) \to S_i(\aac (\sLagr_1), \df')$ induces a bijection $\Psi : Aug(S_i (\sLagr_1)) \to Aug(\sLagr_2)$ by $\saug \mapsto \saug \circ \psi$. Let $\saug' = \saug \circ \psi$. From the formulae for $\psi$ we note that $\saug'(x_i) = \saug(x_i)$ for all $x_i$, $\saug' (b) = \saug ( \beta )$, and $\saug' (a) = \saug(v)$ where $\df b = a + v$. The next two lemmata follow from Lemma~3.2 in \cite{Sabloff2005} and detail the behavior of $\saug'$ on $y_1, \hdots, y_M$.  

\begin{lem}
	\label{lem:dip-aug-1}
	If we extend an augmentation $\saug \in Aug(\sLagr_1)$ to an augmentation $\saug \in Aug(S_i (\sLagr_1))$ by $\saug(\beta) = \saug(\alpha) = 0$ then the augmentation $\saug' \in Aug(\sLagr_2)$ given by $\saug' = \saug \circ \psi$ satisfies:
	
\begin{enumerate}
	\item $\saug'(b) = 0$,
	\item $\saug'(a) = \saug(v)$ where $\df b = a + v$, and
	\item $\saug' = \saug$ on all other crossings. 
\end{enumerate}
\end{lem}

\begin{proof}
We are left to show $\saug' = \saug$ on $y_1, \hdots, y_M$. This follows from observing that $\saug(\beta) = 0$ implies $\saug \circ H \circ \psi (\df y_j) = 0$ for all $j$. 
\end{proof}

\begin{lem}
	\label{lem:dip-aug-2}
	Suppose we extend an augmentation $\saug \in Aug(\sLagr_1)$ to an augmentation $\saug \in Aug(S_i (\sLagr_1))$ by $\saug(\beta) = 1$ and $\saug(\alpha) = 0$. Suppose that in $(\aac(\sLagr_2), \df)$ the generator $a$ appears in the boundary of each of the generators $\{ y_{j_1}, \hdots, y_{j_l} \}$. Suppose that for $y \in \{ y_{j_1}, \hdots, y_{j_l} \}$, each disk contributing $a$ to $\df y$ has the form $QaR$ where $Q, R \in \aac(\sLagr_2)$ and $Q$ and $R$ do not contain $a$ or $b$. Then the augmentation $\saug' \in Aug(\sLagr_2)$ satisfies:
	
\begin{enumerate}
	\item $\saug'(b) = 1$;
	\item $\saug'(a) = \saug(v)$ where $\df b = a + v$;
	\item For each $y \in \{ y_{j_1}, \hdots, y_{j_l} \}$, $\saug' (y) = \saug (y)$ if and only if the generator $a$ appears in an even number of terms in $\df y$ that are of the form $QaR$ with $\saug(Q) = \saug(R) = 1$; and
	\item $\saug' = \saug$ on all other crossings.  
\end{enumerate}

\end{lem}

\begin{proof}
Cases (1), (2), and (4) follow as in Lemma~\ref{lem:dip-aug-1}. Let $y \in \{ y_1, \hdots, y_M \}$. Note that $\saug' (y) = \saug(y) + \saug \circ H \circ \psi (\df y)$. Suppose $a$ does not appear in $\df y$. Then  $H \circ \psi (\df y) = 0$, so $\saug' (y) = \saug(y)$. Suppose $a$ does appear in $\df y$. By assumption, each disk contributing $a$ to $\df y$ has the form $QaR$ where $Q, R \in \aac(\sLagr_2)$ and $Q$ and $R$ do not contain $a$ or $b$. Thus $H \circ \psi (QaR) = Q \beta R$ for each disk of the form $QaR$ in $\df y$ and $H \circ \psi$ is 0 on all of the other disks in $\df y$. Now $\saug \circ H \circ \psi (QaR) = \saug(Q) \saug(R)$, so $\saug \circ H \circ \psi (\df y) = 0$ if and only if $a$ appears in an even number of terms in $\df y$ that are of the form $QaR$ with $\saug(Q) = \saug(R) = 1$. 
\end{proof}

Now we follow an augmentation through the complete process of adding a dip to $\sNgres$. Suppose $\sNgresD$ is a dipped diagram of $\sNgres$ with dips $D_1, \hdots, D_m$. We would like to add a new dip $D$ to $\sNgresD$ away from $D_1, \hdots, D_m$. We will denote the resulting dipped diagram by $\sNgres^{d'}$. Our goal is to understand how an augmentation in $Aug(\sNgresD)$ extends to an augmentation in $Aug(\sNgres^{d'})$ using the stable tame isomorphisms defined above. This work is motivated by the construction defined in Section 3.3 of \cite{Sabloff2005}. 

We set the following notation so that the arguments below are not as cluttered. Let $L = \sNgresD$ and $L' = \sNgres^{d'}$. We will assume that, with respect to the ordering on dips coming from the $x$-axis, $D$ is not the left-most dip in $L'$. Let $I$ denote the insert in $L'$ bounded on the right by $D$. Let $D_j$ denote the dip bounding $I$ on the left. At the location of the dip $D$, label the strands of $\sNgresD$ from bottom to top with the integers $1, \hdots, n$. The creation of the dip $D$ gives a sequence of Lagrangian projections $L, L_{2,1}, \hdots, L_{n, n-1} = L'$, where $L_{k,l}$ denotes the result of pushing strand $k$ over strand $l$. Let $\df_{k,l}$ denote the boundary map of the CE-DGA of $L_{k,l}$. Let $D_{k,l}$ denote the partial dip in $L_{k,l}$. In each Lagrangian projection $L_{k,l}$, the insert $I$ and dip $D_j$ sit to the left of the partial dip $D_{k,l}$. 

Suppose $(r,s)$ denotes the type II move immediately preceding the type II move $(k,l)$. Then $L_{k,l}$ is the result of pushing strand $k$ over strand $l$ in $L_{r,s}$.

\subsubsection{Extending $\saug \in Aug(\sNgresD)$ by 0.}
\label{subs:extend-by-0}

Let us suppose that we have extended $\saug \in Aug(L_{r,s})$ so that $\saug'(b^{k,l}) = 0$. Lemma~\ref{lem:dip-aug-1} implies $\saug'(a^{k,l}) = \saug(v)$ where $\df_{k,l} b^{k,l} = a^{k,l} + v$, and $\saug' = \saug$ on all other crossings. If the insert $I$ is of one of the four types in Figure~\ref{f:inserts}, then we can describe the disks in $v$. Note that $a^{k,l}$ is the only disk in $\df_{k,l} b^{k,l}$ with the lower right corner of $b^{k,l}$ as its positive convex corner, thus we need only understand the disks with the upper left corner of $b^{k,l}$ as their positive convex corner. In the proof of Lemma~\ref{lem:dipped-df}, we let $L(\df b^{k,l})$ denote the disks in $\df b^{k,l}$ that include the upper left corner of $b^{k,l}$ as their positive convex corner. Here $\df b^{k,l}$ refers to the boundary map in the CE-DGA of $\sNgres^{d'}$. The order in which we perform the type II moves that create $D$ ensures that when the crossing $ b^{k,l}$ is created, all of the disks in $L(\df b^{k,l})$ appear in $\df_{k,l} b^{k,l}$. In fact, the restrictions placed on convex immersed polygons by the height function imply that any disk appearing in $\df_{k,l} b^{k,l}$ must also appear in $\df b^{k,l}$. Thus, we have $v = L(\df b^{k,l})$ and so $\saug'(a^{k,l}) = \saug(L(\df b^{k,l}))$, which is equal to the $(k,l)$-entry in the matrix $\saug(\widetilde{A_j}(I + B))$. Since $\saug$ is an algebra map and $\saug'(B) = 0$, we see that $\saug(\widetilde{A_j}(I + B)) = \saug(\widetilde{A}_{j})$. Thus we have the following:

\begin{lem}
	\label{lem:extend-by-0}
Suppose $\sNgresD$ is a dipped diagram of the Ng resolution $\sNgres$. Suppose we use the dipping procedure to add a dip $D$ between the existing dips $D_j$ and $D_{j+1}$ and thus create a new dipped diagram $\sNgres^{d'}$. Suppose the insert $I$ between $D_j$ and $D$ is of one of the four types in Figure~\ref{f:inserts}. Let $\saug \in Aug(\sNgresD)$. Then if at every type II move in the creation of the dip $D$ we choose to extend $\saug$ so that $\saug'$ is 0 on the new crossing in the $b$-lattice, then the stable tame isomorphism from Section~\ref{sec:tracking-augs} maps $\saug$ to $\saug' \in Aug(\sNgres^{d'})$ satisfying:

\begin{enumerate}
	\item $\saug'(B) = 0$, 
	\item $\saug'(A) = \saug(\widetilde{A}_{j})$, and
	\item $\saug' = \saug$ on all other crossings.
\end{enumerate}
\end{lem}

\begin{defn}
	\label{defn:extend-by-0}
	If $\saug \in Aug(\sNgresD)$ then we say that $\saug$ is \emph{extended by 0} if after each type II move in the creation of $D$, we extend $\saug$ so that $\saug$ sends the new crossing of the $b$-lattice of $D$ to 0.
\end{defn} 

In the next corollary, we investigate $\saug(\widetilde{A}_{j})$ by revisiting the definition of $\widetilde{A}_{j}$ from Section~\ref{sec:boundary-dipped}. The assumptions on $\saug (q)$ and $\saug (z)$ in cases (2) and (3) simplify the matrices $\saug(\widetilde{A}_{j})$ considerably, although verifying case (3) is still a slightly tedious matrix calculation. 

\begin{cor}
	\label{cor:extend-by-0}
	Suppose we are in the setup of Lemma~\ref{lem:extend-by-0}. Let $\saug \in Aug(\sNgresD)$ and let $\saug' \in Aug(\sNgres^{d'})$ be the extension of $\saug$ described in Lemma~\ref{lem:extend-by-0}. Then:

\begin{enumerate}
	\item If $I$ is of type (1), then $\saug'(A) = \saug(A_j)$.
	
	\item Suppose $I$ is of type (2) with crossing $q$ between strands $i+1$ and $i$. If $\saug (q) = 0$, then $\saug'(A) = P_{i+1,i} \saug(A_j) P_{i+1,i}$.
	
	\item Suppose $I$ is of type (3) with crossing $z$ between strands $i+1$ and $i$.  
	
		\begin{enumerate}
			\item Let $i+1 < u_1 < u_2 < \hdots < u_s$ denote the strands at dip $D_j$ that satisfy $\saug(a^{u_1,i}_{j}) = \hdots = \saug(a^{u_s,i}_{j}) = 1$; 
			\item Let  $v_r <  \hdots < v_1 $ denote the strands at dip $D_j$ that satisfy $\saug(a^{i+1,v_1}_{j}) = \saug(a^{i+1,v_2}_{j}) = \hdots = \saug(a^{i+1,v_r}_{j}) = 1$; and
			\item Let $E = E_{i, v_r} \hdots E_{i, v_1} E_{u_1, i+1} \hdots E_{u_s, i+1}$.
		\end{enumerate}

If $\saug (z) = 0$ and $\saug(a^{i+1,i}_{j}) = 1$, then, as matrices, $ \saug'(A) = J_{i-1} E \saug(A_j) E^{-1} J_{i-1}^{T}$.
	
	\item Suppose $I$ is of type (4) and the resolved birth is between strands $i+1$ and $i$. Then, as matrices, $\saug'(A)$ is obtained from $\saug(A_j)$ by inserting two rows (columns) of zeros after row (column) $i-1$ in $\saug(A_j)$ and then changing the $(i+1,i)$ entry to 1.
\end{enumerate}
	
\end{cor}

Besides Case (1), these matrix equations correspond to the chain maps connecting consecutive ordered chain complexes in an MCS with simple births; see Definition~\ref{defn:MCS-maps}. Thus we see the first hint of an explicit connection between MCSs and augmentations. 

\subsubsection{Extending $\saug \in Aug(\sNgresD)$ by $\sHSM_{i+1,i}$.}

In this section, we consider extending $\saug \in Aug(\sNgresD)$ to an augmentation $\saug' \in Aug(\sNgres^{d'})$ so that $\saug'(B) = \sHSM_{i+1,i}$. Recall $\sHSM_{i+1,i}$ denotes a square matrix with 1 in the $(i+1,i)$ position and zeros everywhere else. During the type II move that pushes strand $i+1$ over strand $i$ we will choose to extend $\saug$ so that $\saug' (b^{i+1,i}) = 1$. By carefully using Lemma~\ref{lem:dip-aug-2}, we are able to determine the extended augmentation $\saug'$.

\begin{defn}
	\label{defn:extend-by-1}
	Let $\saug \in Aug(\sNgresD)$. We say that $\saug$ is \emph{extended by $\sHSM_{i+1,i}$} if $\smu(i+1) = \smu(i)$ and after each type II move in the creation of a new dip $D$, we extend $\saug$ so that the extended augmentation $\saug' \in Aug(\sNgres^{d'})$ satisfies $\saug'(B) = \sHSM_{i+1,i}$.
\end{defn} 

Understanding $\saug'$ when $\saug'(b^{i+1,i}) = 1$ requires that we understand all of the crossings $y$ such that $a^{i+1,i}$ appears in $\df_{i+1,i} y$. If we add the following restrictions on $I$, $D$, and the pair $(i+1,i)$, then we can identify all disks containing $a^{i+1,i}$ as a negative convex corner. Suppose:

\begin{enumerate}
	\item $|b^{i+1,i}| = 0$,
	\item $I$ is of type (1),
	\item $D$ occurs to the immediate left of a crossing $q$ in $\sNgres$ between strands $i+1$ and $i$.
\end{enumerate}

\begin{figure}[t]
\centering
\includegraphics[scale=.3]{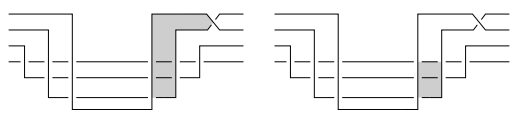}
\caption[Disks in which $a^{i+1,i}$ is a negative convex corner.]{Disks appearing in $\df_{i+1,i}$ with $a^{i+1,i}$ as a negative convex corner.}
\label{f:extend-by-1}
\end{figure}

Given these conditions, $a^{i+1,i}$ only appears in $\df_{i+1,i} q$ and $\df_{i+1,i} a^{i+1,l}$ for $l < i$; see Figure~\ref{f:extend-by-1}. Applying Lemma~\ref{lem:dip-aug-2}, we conclude that after the type II move that pushes strand $i+1$ over strand $i$ the augmentation $\saug'$ satisfies:

\begin{enumerate}
	\item $\saug'(b^{i+1,i}) = 1$,
	\item $\saug'(a^{i+1,i}) = \saug(v)$ where $\df b = a + v$, 
	\item $\saug' (q) \neq \saug(q)$,
	\item For all $l < i$, $\saug'(a^{i+1,l}) = \saug(a^{i+1,l}_j)$ if and only if $\saug'(a^{i,l}) = 0$, and
	\item $\saug' = \saug$ on all other crossings.  
\end{enumerate}

In Section~\ref{subs:extend-by-0} we showed that when $I$ is an insert of type (1) - (4), $v = L(\df b^{i+1,i})$ where $\df$ is the boundary in $\sNgres^{d'}$. Since $I$ is an insert of type (1), we conclude $\saug'(a^{i+1,i}) = \saug(a^{i+1,i}_j)$.

Suppose we continue creating the dip $D$ and with each new type II move we extend the augmentation so that it sends the new crossing in the $b$-lattice to 0. By Lemma~\ref{lem:extend-by-0}, we need only compute the value of the extended augmentation on the new crossing in the $a$-lattice. Consider the type II move pushing strand $k$ over $l$ where $(i+1,i) \prec (k,l)$. By Lemma~\ref{lem:extend-by-0}, we have 

\begin{equation*}
\saug'(a^{k,l}) = \saug(v) = \saug(L(\df b^{k,l})) = \saug(a^{k,l}_{j}) + \sum_{l<p<k} \saug(a^{k,p}_{j}) \saug'(b^{p,l}). 
\end{equation*}

\noindent Since $\saug'(b^{p,l}) = 1$ if and only if $(p,l) = (i+1,i)$, we know that $\saug'(a^{k,l}) = \saug(a^{k,l}_j)$ if and only if $l \neq i$ or $\saug(a^{k,i+1}_j) = 0$. Thus, if we extend $\saug \in Aug(\sNgresD)$ to $\saug' \in Aug(\sNgres^{d'})$ so that $\saug'(B) = \sHSM_{i+1,i}$ then $\saug'$ satisfies:

\begin{equation*} 
  \saug'(a^{k,l}) = \begin{cases}
    \saug(a^{i+1,l}_j) + \saug'(a^{i,l}) & k = i+1 \mbox{, and } l < i \\
    \saug(a^{k,i}_j) + \saug(a^{k,i+1}_j) & k > i+1 \mbox{, and } l = i \\
    \saug(a^{k,l}_j) & \mbox{otherwise.} 
  \end{cases}
\end{equation*}

\noindent This equation is equivalent to $\saug'(A) = E_{i+1,i} \saug(A_j) E_{i+1,i}^{-1}$. Pulling this all together, we have the following lemma.

\begin{lem}
	\label{lem:extend-by-1}
Suppose $\sNgresD$ is a dipped diagram of the Ng resolution $\sNgres$. Let $q$ be a crossing in $\sNgresD$ corresponding to a resolved crossing of $\sfront$ and with $|q| = 0$. Suppose we add a dip $D$ to the right of the existing dip $D_j$ and just to the left  of $q$, thus creating a new dipped diagram $\sNgres^{d'}$. Suppose the insert $I$ between $D$ and $D_j$ is of type (1). Let $\saug \in Aug(\sNgresD)$. If $\saug$ is extended by $\sHSM_{i+1,i}$, then the stable tame isomorphism from Section~\ref{sec:tracking-augs} maps $\saug$ to $\saug' \in Aug(\sNgres^{d'})$ where $\saug'$ satisfies:	

\begin{enumerate}
	\item $\saug'(B) = \sHSM_{i+1,i}$
	\item $\saug'(A) = E_{i+1,i} \saug(A_j) E_{i+1,i}^{-1}$
	\item $\saug'(q) \neq \saug(q)$
	\item $\saug' = \saug$ on all other crossings.
\end{enumerate}

\end{lem}

We now have sufficient tools to begin connecting MCSs and augmentations. 

\section{Relating MCSs and Augmentations}
\label{ch:MCS-Aug}

In this section, we construct a surjective map $\widehat{\Psi} : \sDMCSeq \to Aug^{ch}(\sNgres)$ and define a simple construction that associates an MCS $\sMCS$ to an augmentation $\saug \in Aug(\sNgres)$ so that $\widehat{\Psi}([\sMCS]) = [\saug]$. We also detail two algorithms that use MCS moves to place an arbitrary MCS in one of two standard forms. 

\subsection{Augmentations on sufficiently dipped diagrams}
\label{sec:aug-on-dipped}

Let $\sNgresD$ be a sufficiently dipped diagram of $\sNgres$ with dips $\sdip_1, \hdots, \sdip_m$ and inserts $I_1, \hdots, I_{m+1}$. Let $q_i, \hdots, q_M$ and $z_j, \hdots, z_N$ denote the crossings found in the inserts of type (2) and (3) respectively. The subscripts on $q$ and $z$ correspond to the subscripts of the inserts in which the crossings appear. The formulae in Lemma~\ref{lem:dipped-df} allow us to write the augmentation condition $\saug \circ \df$ as a system of local equations involving the dips and inserts of $\sNgresD$. 

\begin{lem}
	\label{lem:dipped-aug}
	An algebra homomorphism $\saug : \aac(\sNgresD) \to \zz_2$ on a sufficiently dipped diagram $\sNgresD$ of a Ng resolution $\sNgres$ with $\saug(1)=1$ satisfies $\saug \circ \df = 0$ if and only if:
	
\begin{enumerate}
	\item $\saug(a_{s-1}^{i+1,i})=0 $, where $q_s$ is a crossing between strands $i+1$ and $i$, 
	\item $\saug(a_{r-1}^{k+1,k})= 1$, where $z_r$ is a crossing between strands $i+1$ and $i$,
	\item $\saug(A_j)^2 = 0$, and
	\item $\saug(A_j) = (I + \saug(B_j)) \saug(\widetilde{A}_{j-1})(I + \saug(B_j))^{-1}$.
\end{enumerate}
\end{lem}

We will primarily concern ourselves with the following types of augmentations. 

\begin{defn}
	\label{defn:simple-aug}
	Given a sufficiently dipped diagram $\sNgresD$, we say an augmentation $\saug \in Aug(\sNgresD)$ is \emph{occ-simple} if:
	
\begin{enumerate}
	\item $\saug$ sends all of the crossings of type $q_s$ and $z_r$ to 0, 
	\item For $I_j$ of type (1), either $\saug (B_j) = 0$ or $\saug (B_j) = \sHSM_{k,l}$ for some $k > l$, and
	\item For $I_j$ of type (2), (3), or (4), $\saug (B_j) = 0$. 	
\end{enumerate}
	
	$\sAugsimple$ denotes the set of all such augmentations in $Aug(\sNgresD)$. We say $\saug \in \sAugsimple$ is \emph{minimal occ-simple} if for all $I_j$ of type (1), $\saug (B_j) = \sHSM_{k,l}$ for some $k > l$. We let $\sAugminsimple$ denote the set of all minimal occ-simple augmentations over all possible sufficiently dipped diagrams of $\sNgres$. 
\end{defn}

The matrices $\saug(A_1), \hdots, \saug(A_m)$ of an occ-simple augmentation determine a sequence of ordered chain complexes. This will be made explicit in Lemma~\ref{lem:mcs-aug-dipped}.

The following result is a consequence of Lemma~\ref{lem:dipped-aug}, Definition~\ref{defn:simple-aug} and the definition of $\widetilde{A}_{j-1}$ given in Section~\ref{sec:boundary-dipped}. Recall that the matrices $E_{k,l}$, $P_{i+1,i}$, $J_i$, and $\sHSM_{k,l}$ were defined in Section~\ref{sec:matrices}. The proof of Lemma~\ref{lem:aug-eq} is essentially identical to the proof of Corollary~\ref{cor:extend-by-0}. 

\begin{lem}
	\label{lem:aug-eq}
If $\sNgresD$ is a sufficiently dipped diagram and $\saug \in \sAugsimple$, then for each insert $I_j$: 

\begin{enumerate}
	\item Suppose $I_j$ is of type (1), then either $\saug(A_j) = \saug(A_{j-1})$ or $\saug(A_j) = E_{k,l} \saug(A_{j-1})E^{-1}_{k,l}$ for some $k>l$.	
	\item Suppose $I_j$ is of type (2) with crossing $q$ between strands $i+1$ and $i$, then $\saug(A_j) = P_{i+1,i} \saug(A_{j-1}) P^{-1}_{i+1,i}$.
	\item Suppose $I_j$ is of type (3) with crossing $z$ between strands $i+1$ and $i$.  
	
		\begin{enumerate}
			\item Let $i+1 < u_1 < u_2 < \hdots < u_s$ denote the strands at dip $D_{j-1}$ that satisfy $\saug(a^{u_1,i}_{j-1}) = \hdots = \saug(a^{u_s,i}_{j-1}) = 1$; 
			\item Let  $v_r <  \hdots < v_1 $ denote the strands at dip $D_{j-1}$ that satisfy $\saug(a^{i+1,v_1}_{j-1}) = \saug(a^{i+1,v_2}_{j-1}) = \hdots = \saug(a^{i+1,v_r}_{j-1}) = 1$; and
			\item Let $E = E_{i, v_r} \hdots E_{i, v_1} E_{u_1, i+1} \hdots E_{u_s, i+1}$.
		\end{enumerate}

Then $\saug(A_j) = J_{i-1} E \saug(A_{j-1}) E^{-1} J_{i-1}^{T}$.
	
	\item Suppose $I_j$ is of type (4) and the resolved birth is between strands $i+1$ and $i$. Then, as matrices, $\saug(A_j)$ is obtained from $\saug(A_{j-1})$ by inserting two rows (columns) of zeros after row (column) $i-1$ in $\saug(A_{j-1})$ and then changing the $(i+1,i)$ entry to 1.
\end{enumerate}

\end{lem}

The equations in Lemma~\ref{lem:aug-eq} are identical to the equations found in Definition~\ref{defn:MCS-maps}, with the exception of the first case of a type (1) insert. Indeed, given a front diagram $\sfront$ with resolution $\sNgres$, MCSs on $\sfront$ correspond to minimal occ-simple augmentations of $\sNgres$. We assign a minimal occ-simple augmentation $\saug_{\sMCS}$ to an MCS $\sMCS$ using an argument of Fuchs and Rutherford from \cite{Fuchs2008}. 

\begin{lem}
	\label{lem:mcs-aug-dipped}
	The set $\sDMCS$ of MCSs of $\sfront$ with simple births is in bijection with the set $\sAugminsimple$ of minimal occ-simple augmentations.
\end{lem}

\begin{proof}
We begin by assigning an MCS to a minimal occ-simple augmentation $\saug \in Aug(\sNgresD)$, where $\sNgresD$ is a sufficiently dipped diagram of $\sNgres$. Let $\saug \in \sAugminsimple$. For each dip $D_j$, we form an ordered chain complex $(C_j, \df_j)$ as follows. Let $t_j$ denote the $x$-coordinate of the vertical lines of symmetry of $D_j$ in $\sNgresD$. Label the $m_j$ points of intersection in $\sNgresD \cap (\{t_j\} \times \rr)$ by $y_1^j, \hdots, y_{m_j}^j$ and let $C_j$ be a $\zz_2$ vector space generated by $y_1^j, \hdots, y_{m_j}^j$. We label the generators $y_1^j, \hdots, y_{m_j}^j$ based on their $y$-coordinate so that $y_1^j > \hdots > y_{m_j}^j$. Each generator is graded by the Maslov potential of its corresponding strand in $\sfront$. The grading condition on $\saug$ and the fact that $\saug(A_j)^2 = 0$ implies that $\saug(A_j)$ is a matrix representative of a differential on $C_j$. Thus $(C_j, \df_j)$ with $\sdmatrix_j = \saug(A_j)$ is an ordered chain complex. Recall that the notation $\sdmatrix_j$ was established in Remark~\ref{rem:matrix-rep}. The relationship between consecutive differentials $\df_{j-1}$ and $\df_j$ is defined in Lemma~\ref{lem:aug-eq}. Since $\saug$ is minimal occ-simple, each insert $I_{j}$ of type (1) satisfies $\sdmatrix_j = E_{k,l}\sdmatrix_{j-1}E_{k,l}^{-1}$ where $\saug(B_j) = \sHSM_{k,l}$ and $k > l$. Thus, $(C_{j-1}, \df_{j-1})$ and $(C_{j}, \df_{j})$ are related by a handleslide between $k$ and $l$. Since the matrix equations in Lemma~\ref{lem:aug-eq} satisfied by $\saug$ are equivalent to the matrix equations in Definition~\ref{defn:MCS-maps}, the sequence $(C_1, \df_1), \hdots, (C_m, \df_m)$ forms an MCS $\sMCS$ on $\sfront$.

\begin{figure}
\labellist
\small\hair 2pt
\pinlabel {$D_1$} [tl] at 82 20
\pinlabel {$D_2$} [tl] at 132 20
\pinlabel {$D_3$} [tl] at 226 20
\pinlabel {$D_4$} [tl] at 295 20
\pinlabel {$D_5$} [tl] at 339 20
\pinlabel {$D_6$} [tl] at 385 20
\pinlabel {$D_7$} [tl] at 424 20
\pinlabel {$D_8$} [tl] at 465 20
\pinlabel {$\sfront$} [tl] at 490 517
\pinlabel {$\sNgres$} [tl] at 490 220
\pinlabel {$x$} [tl] at 530 44
\pinlabel {$y$} [tl] at 23 300
\pinlabel {$x$} [tl] at 530 341
\pinlabel {$z$} [tl] at 23 590
\endlabellist
\centering
\includegraphics[scale=.4]{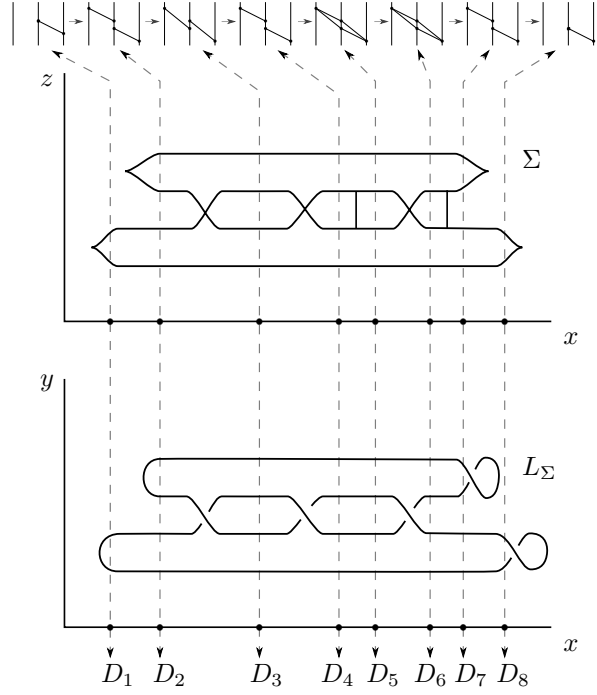}
\caption[Assigning an augmentation to an MCS.]{Assigning an augmentation to an MCS.}
\label{f:mcs-aug-example}
\end{figure}

This process is invertible. Let $\sMCS \in \sDMCS$ be a Morse complex sequence of the front projection $\sfront$ with chain complexes $(C_1, \df_1) \hdots (C_m, \df_m)$. We will use the marked front projection associated to $\sMCS$ to define the placement of dips creating $\sNgresD$. Afterwards, we define an algebra homomorphism $\saug_{\sMCS} : \aac(\sNgresD) \to \zz_2$ and show that it is an augmentation. Figure~\ref{f:mcs-aug-example} gives an example of this process.

Add a dip to $\sNgres$ just to the right of the corresponding location of each resolved cusp, resolved crossing or handleslide mark. The resulting dipped diagram $\sNgresD$ is sufficiently dipped with $m$ dips. We define a $\zz_2$-valued map $\saug_{\sMCS}$ on the crossings of $\sNgresD$ by:

\begin{enumerate}
	\item $\saug_{\sMCS}(q_s) = 0 $ for all crossings $q_s$ coming from a resolved crossing of $\sfront$;
	\item $\saug_{\sMCS}(z_r) = 0 $ for all crossings $z_r$ coming from a resolved right cusp;
	\item $\saug_{\sMCS}(A_j) = \sdmatrix_j$ for all $j$;
	\item If $I_j$ is of type (1), then there exists $k>l$ such that $\sMCS$ has a handleslide mark in $I_j$ between strands $k$ and $l$. Let $\saug_{\sMCS}(B_j) = \sHSM_{k,l}$; and
	\item If $I_j$ is of type (2), (3), or (4), then let $\saug_{\sMCS}(B_j) = 0$.
\end{enumerate}

We define $\saug_{\sMCS} (1) = 1$ and extend $\saug_{\sMCS}$ by linearity to an algebra homomorphism on $\aac(\sNgresD)$. If $\saug_{\sMCS}$ is an augmentation, then it is minimal occ-simple by construction. We must show $\saug_{\sMCS} \circ \df = 0$ and, for all crossings $q$ of $\sNgresD$, $\saug_{\sMCS} (q) = 1$ implies $|q| = 0$.

If $\saug_{\sMCS}(a_j^{k,l}) = 1$, then $\langle \df_j y_k^j | y_l^j \rangle = 1$ where the notation for the generators of $(C_j, \df_j)$ corresponds to the notation in Definition~\ref{defn:MCS}. Thus $\smu(k) = \smu(l) + 1$, where $\smu(i)$ denotes the Maslov potential of the strand in $\sfront$ corresponding to the generator $y_i^j$. Recall $|a_j^{k,l}| = \smu(k) - \mu(l) - 1$ and so $\smu(k) = \smu(l) + 1$ implies $|a_j^{k,l}| = 0$. If $\saug_{\sMCS}(b_j^{k,l}) = 1$, then in the marked front projection of $\sMCS$, a handleslide mark occurs in $I_j$ between strands $k$ and $l$. Thus $\smu(k) = \smu(l)$ and, since $|b_j^{k,l}| = \smu(k) - \mu(l)$, $|b_j^{k,l}| = 0$.

It remains to show that $\saug_{\sMCS} \circ \df = 0$. Let $q_r$ denote a crossing corresponding to a resolved crossing of $\sfront$ between strands $i+1$ and $i$. By Remark~\ref{rem:mcs-defn}~(2), the $(i+1,i)$ entry of the matrix $\df_{r-1}$ is 0, hence $\saug_{\sMCS}(a_{r-1}^{i+1,i}) = 0$. Thus, $\saug_{\sMCS}(\df q_r) = \saug_{\sMCS}(a_{r-1}^{i+1,i}) = 0$. Let $z_s$ denote a crossing corresponding to a resolved right cusp of $\sfront$ between strands $i+1$ and $i$. By Remark~\ref{rem:mcs-defn}~(2), the $(i+1,i)$ entry of the matrix $\df_{s-1}$ is 1, hence $\saug_{\sMCS}(a_{s-1}^{i+1,i}) = 1$. Thus, $\saug_{\sMCS}(\df z_s) = 1 + \saug_{\sMCS}(a_{s-1}^{i+1,i}) = 0$. Since each $\sdmatrix_j$ is the matrix representative of a differential, we see that $\saug_{\sMCS}(\df A_j) = \saug_{\sMCS}(A_j)^2 = 0$. Verifying $\saug_{\sMCS} \circ \df = 0$ on each $b$-lattice $B_j$ is equivalent to verifying (4) in Lemma~\ref{lem:dipped-aug} and this can be done with Lemma~\ref{lem:aug-eq}. Indeed, the matrix equations in Lemma~\ref{lem:aug-eq} correspond to the matrix equations in Definition~\ref{defn:MCS-maps} relating consecutive chain complexes in $\sMCS$. Thus, $\saug_{\sMCS}(\df B_j) = 0$ for all $j$ and so $\saug_{\sMCS}$ is a minimal occ-simple augmentation on $\sNgresD$.
\end{proof}

\subsection{Defining $\Psi : \sDMCS \to Aug^{ch}(\sNgres)$} 

Given $\sMCS \in \sDMCS$, Lemma~\ref{lem:mcs-aug-dipped} gives an explicit construction of an augmentation $\saug_{\sMCS}$ on a sufficiently dipped diagram $\sNgresD$. In this section, we will use this construction to build a map $\Psi : \sDMCS \to Aug^{ch}(\sNgres)$. In order to do so, we must understand the possible stable tame isomorphisms between $\aac(\sNgresD)$ and $\aac(\sNgres)$ coming from sequences of Lagrangian Reidemeister moves. In \cite{Kalman2005}, Kalman proves that two sequences of Lagrangian Reidemeister moves may induce inequivalent maps between the respective contact homologies. However, if we restrict our sequences of moves to adding and removing dips, then the we can avoid this problem and give a well-defined map $\Psi : \sDMCS \to Aug^{ch}(\sNgres)$.

The following results from \cite{Kalman2005} allow us to modify a sequence of Lagrangian Reidemeister moves by removing canceling pairs of moves and commuting pairs of moves that are far away from each other. These modifications do not change the resulting bijection on augmentation chain homotopy classes. 

\begin{prop}[\cite{Kalman2005}]
	\label{prop:II-IIinv}
If we perform a type II move followed by a type $\IIinv$ move that creates and then removes two crossings $b$ and $a$, then the induced DGA morphism from $(\aac(\sLagr), \df)$ to $(\aac(\sLagr), \df)$ is equal to the identity. If we perform a type $\IIinv$ move followed by a type II move that removes and then recreates two crossings $b$ and $a$, then the induced DGA morphism from $(\aac(\sLagr), \df)$ to $(\aac(\sLagr), \df)$ is chain homotopic to the identity. 
\end{prop}

\begin{prop}[\cite{Kalman2005}]
	\label{prop:commute-II}
Suppose $\sLagr_1$ and $\sLagr_2$ are related by two consecutive moves of type II or $\IIinv$. We will call these moves A and B. Suppose the crossings involved in A and B form disjoint sets. Then the composition of DGA morphisms constructed by performing move A and then move B is chain homotopic to the composition of DGA morphisms constructed by performing move B and then move A.
\end{prop}

Proposition~\ref{prop:commute-II} follows from Case 1 of Theorem~3.7 in \cite{Kalman2005}.

\subsubsection{Dipping/undipping paths for $\sNgresD$} 
\label{sec:dip-undip-paths}

Suppose $\sNgresD$ has dips $D_1, \hdots, D_m$. Let $t_1, \hdots, t_m$ denote the $x$-coordinates of the vertical lines of symmetry of the dips $D_1, \hdots, D_m$ in $\sNgresD$.

\begin{defn}
	\label{defn:dip-path}
A \emph{dipping/undipping path for $\sNgresD$} is a finite-length monomial $w$ in the elements of the set $\{s^{\pm}_1, \hdots, s^{\pm}_n, t^{\pm}_1, \hdots, t^{\pm}_m \}$. We require that $w$ satisfies:

\begin{enumerate}
	\item Each $s_i$ in $w$ denotes a point on the $x$-axis away from the dips $D_1, \hdots, D_m$; and
	\item As we read $w$ from left to right, the appearances of $s_i$ alternate between $s^{+}_i$ and $s^{-}_i$, beginning with $s^{+}_i$ and ending with $s^{-}_i$. The appearances of $t_i$ alternate between $t^{-}_i$ and $t^{+}_i$, beginning with $t^{-}_i$ and ending with $t^{-}_i$ and each $t_i$ is required to appear at least once.
\end{enumerate}
\end{defn}

Each dipping/undipping path $w$ is a prescription for adding and removing dips from $\sNgresD$. In particular, $s^{+}_i$ tells us to introduce a dip in a small neighborhood of $s_i$. The letter $s^{-}_i$ tells us to remove the dip that sits in a small neighborhood of $s_i$. The order in which these type $\IIinv$ moves occur is the opposite of the order used in $s^{+}_i$. The elements $t^{+}_i$ and $t^{-}_i$ work similarly. We perform these moves on $\sNgresD$ by reading $w$ from left to right. The conditions we have placed on $w$ ensure that we are left with $\sNgres$ after performing all of the prescribed dips and undips.

Let $w_0 = t^{-}_m t^{-}_{m-1} \hdots t^{-}_1$. Then $w_0$ tells us to undip $D_1, \hdots, D_m$ beginning with $D_m$ and working to $D_1$. Each $w$ determines a stable tame isomorphism $\psi_{w}$ from $\aac(\sNgresD)$ to $\aac(\sNgres)$ which determines a bijection $\Psi_w : Aug^{ch}(\sNgresD) \to Aug^{ch}(\sNgres)$. We are now in a position to define a map $\Psi : \sDMCS \to Aug^{ch}(\sNgres)$.

\begin{defn}
	\label{defn:mcs-aug-ch}
	Define $\Psi : \sDMCS \to Aug^{ch}(\sNgres)$ by $\Psi(\sMCS) = \Psi_{w_0}([\saug_{\sMCS}])$. 
\end{defn}

The following result implies that the definition of $\Psi$ is independent of dipping/undipping paths.

\begin{lem}
	\label{lem:dip-path}
If $w$ is a dipping/undipping paths for $\sNgresD$, then $\Psi_w = \Psi_{w_0}$ and, thus, $\Psi_{w}([\saug_{\sMCS}]) = \Psi_{w_0}([\saug_{\sMCS}]) = \Psi(\sMCS)$.   
\end{lem}

\begin{proof}
By Proposition~\ref{prop:commute-II}, we may reorder type II and $\IIinv$ moves that are ``far apart'' without changing the chain homotopy type of the resulting map from $\aac(\sNgresD)$ and $\aac(\sNgres)$. If $s^{+}_i$ appears in $w$, then the next appearance of $s_i$ is the letter $s^{-}_i$ to the right of $s^{+}_i$. The letters between $s^{-}_i$ and $s^{+}_i$ represent type II and $\IIinv$ moves that are far away from $s^{-}_i$ and $s^{+}_i$. Thus, we may commute $s^{-}_i$ past the other letters so that $s^{-}_i$ immediately follows $s^{+}_i$. By our ordering of the type II moves in $s^{+}_i$ and the type $\IIinv$ moves in $s^{-}_i$, we may remove the type II and $\IIinv$ moves in pairs. By Proposition~\ref{prop:II-IIinv}, this does not change the chain homotopy type of the resulting map from $\aac(\sNgresD)$ and $\aac(\sNgres)$. In this manner, we remove all pairs of letters of the form $s^{+}_i$ and $s^{-}_i$ from $w$. By the same argument we remove pairs of letters of the form $t^{+}_i, t^{-}_i$. The resulting dipping/undipping path, denoted $w'$, only contains the letters $t^{-}_1, \hdots, t^{-}_m$. We may reorder these letters so that $w' = w_0$. By Proposition~\ref{prop:commute-II}, this rearrangement does not change the chain homotopy type of the resulting map from $\aac(\sNgresD)$ and $\aac(\sNgres)$, thus, by Lemma~\ref{lem:aug-ch}, $\Psi_w = \Psi_{w_0}$.
\end{proof}

\subsection{Associating an MCS to $\saug \in Aug(\sNgres)$}
\label{sec:aug-mcs-alg}

In this section, we prove that $\Psi : \sDMCS \to Aug^{ch}(\sNgres)$ is surjective. We do so by giving an explicit construction that assigns to each augmentation in $Aug(\sNgres)$ an MCS in $\sDMCS$. The construction follows from our work in Lemma~\ref{lem:extend-by-1} and Corollary~\ref{cor:extend-by-0}.

\begin{lem}
	\label{lem:psi-surjectivity}
	The map $\Psi : \sDMCS \to Aug^{ch}(\sNgres)$ is surjective.
\end{lem}

\begin{proof}
Let $[\saug] \in Aug^{ch}(\sNgres)$ and let $\saug$ denote a representative of $[\saug]$. We will find an MCS mapping to $[\saug]$ by constructing a minimal occ-simple augmentation and applying Lemma~\ref{lem:mcs-aug-dipped}. We add dips to $\sNgres$ beginning at the left-most resolved left cusp and working to the right. As we add dips, we extend $\saug$ using Corollary~\ref{cor:extend-by-0} and Lemma~\ref{lem:extend-by-1}. In a slight abuse of notation, we will let $\saug$ also denote the extended augmentation at every step.

Begin by adding the dip $D_1$ just to the right of the left-most resolved left cusp. This requires a single type II move. We extend $\saug$ by requiring $\saug(b_1^{2,1}) = 0$. In the dipped projection, $\df (b_1^{2,1}) = 1 + a_1^{2,1}$. Hence, $\saug(a_1^{2,1})=1$. 

Suppose we have added dips $D_1, \hdots, D_{j-1}$ to $\sNgres$ and let $x_j$ denote the next resolved cusp or crossing appearing to the right of $D_{j-1}$. We introduce the dip $D_j$ and extend $\saug$ as follows. 

\begin{enumerate}
	\item If $x_j$ is a resolved left cusp, right cusp, or crossing with $\saug(x_j)=0$, then we introduce $D_j$ just to the right of $x_j$ and extend $\saug$ by $0$ as in Definition~\ref{defn:extend-by-0}. 
	\item  If $x_j$ is a resolved crossing between strands $i+1$ and $i$ and $\saug(x_j)=1$, then we introduce $D_j$ just to the left of $x_j$ and extend $\saug$ by $\sHSM_{i+1,i}$ as in Definition~\ref{defn:extend-by-1}. Note that after the extension $\saug(x_j) = 0$. 
\end{enumerate}

\noindent In the first case, Corollary~\ref{cor:extend-by-0} details the relationship between $\saug(A_j)$ and $\saug(A_{j-1})$. In the second case, Lemma~\ref{lem:extend-by-1} details the relationship between $\saug(A_j)$ and $\saug(A_{j-1})$. The resulting augmentation is minimal occ-simple by construction. Let $\widetilde{\saug}$ denote the resulting augmentation on $\sNgresD$. By Lemma~\ref{lem:mcs-aug-dipped}, $\widetilde{\saug}$ has an associated MCS $\sMCS_{\saug}$. By definition, $\Psi (\sMCS_{\saug}) =  \Psi_{w_0}([\widetilde{\saug}])$. The dipping/undipping path $w_0$ tells us to undip $t_1, \hdots, t_n$ in reverse order. In particular, this is the inverse of the process we used to create $\widetilde{\saug}$ on $\sNgresD$. Thus $\Psi (\sMCS_{\saug}) =  \Psi_{w_0}([\widetilde{\saug}]) = [\saug]$ and so $\Psi : \sDMCS \to Aug^{ch}(\sNgres)$ is surjective.
\end{proof} 

\begin{rem}
The marked front projection for $\sMCS_{\saug}$ is constructed as follows. Let $q_{j_1}, \hdots, q_{j_l}$ denote the resolved crossings of $\sNgres$ satisfying $\saug(q)=1$. Then $\sMCS_{\saug}$ has a handleslide mark just to the left of each of the crossings $q_{j_1}, \hdots, q_{j_l}$ in $\sfront$; see Figure~\ref{f:mcs-aug-trefoil}. Each handleslide mark begins and ends on the two strands involved in the crossing.
\end{rem}

\begin{figure}
\labellist
\small\hair 2pt
\endlabellist
\centering
\includegraphics[scale=.3]{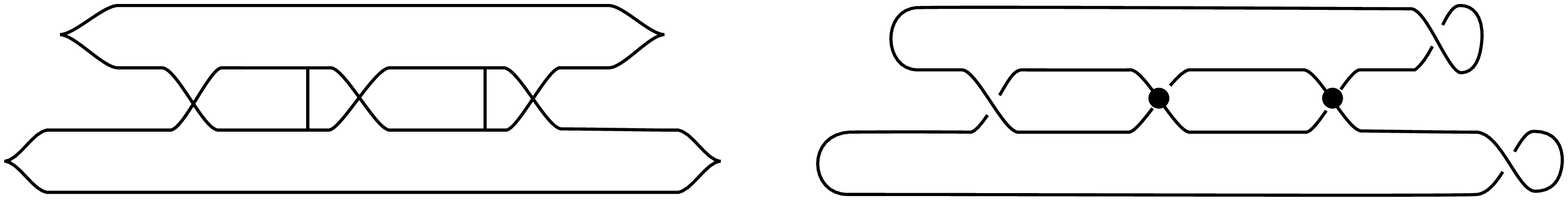}
\caption[Assigning an MCS to an augmentation.]{The dots in the right figure indicate the augmented crossings of $\saug \in Aug(\sNgres)$. The left figure shows the resulting MCS $\sMCS_{\saug}$ from the construction in Lemma~\ref{lem:psi-surjectivity}.}
\label{f:mcs-aug-trefoil}
\end{figure}

\subsection{Defining $\widehat{\Psi} : \sDMCSeq \to Aug^{ch}(\sNgres)$} 

In this section, we prove Theorem~\ref{thm:intro-main-result} from Section~\ref{ch:intro}. In particular,

\begin{thm}
	\label{thm:big-map}
	The map $\widehat{\Psi} : \sDMCSeq \to Aug^{ch}(\sNgres)$ defined by $\widehat{\Psi} ([\sMCS]) = \Psi(\sMCS)$ is a well-defined, surjective map.
\end{thm}

Given $\sMCS_1 \sMCSeq \sMCS_2$ we construct explicit chain homotopies between $\Psi(\sMCS_1)$ and $\Psi(\sMCS_2)$. We begin by restating the chain homotopy property $\saug_1 - \saug_2 = \sH \circ \df$ as a system of local equations.

Let $\saug_1, \saug_2 \in Aug(\sNgresD)$ where $\sNgresD$ is a sufficiently dipped diagram of the Ng resolution $\sNgres$. Let $Q$ denote the set of crossings of $\sNgresD$. By Lemma~\ref{lem:chain-htpy}, a map $\sH : Q \to \zz_2$ that has support on the crossings of grading $-1$ may be extended by linearity and the derivation product property to a map $\sH : (\aac(\sNgresD), \df) \to \zz_2$ with support on $\aac_{-1}(\sNgresD)$. Then $\sH$ is a chain homotopy between $\saug_1$ and $\saug_2$ if and only if $\saug_1 - \saug_2 = \sH \circ \df$ on $Q$.

\begin{lem}
	\label{lem:dipped-ch-htpy}
	Let $\saug_1, \saug_2 \in Aug(\sNgresD)$ where $\sNgresD$ is a sufficiently dipped diagram of the Ng resolution $\sNgres$. Suppose the linear map $\sH : (\aac(\sNgresD), \df) \to \zz_2$ satisfies the derivation product property and has support on crossings of grading $-1$. Then $\sH$ is a chain homotopy between $\saug_1$ and $\saug_2$ if and only if:

\begin{enumerate}
	\item For all $q_r$ and $z_s$, $\saug_1 - \saug_2 = \sH \circ \df$,
	\item For all $A_j$, 
		\begin{equation*}
		\saug_1 (A_j) = (I + \sH (A_j)) \saug_2(A_j) (I + \sH (A_j))^{-1},
		\end{equation*}
	\item For all $B_j$,
		\begin{eqnarray*}
			\sH (B_j) \saug_2 (A_j) + \saug_1 (A_{j-1})\sH (B_j) &=& \saug_1(B_j) + (I + \saug_1(B_j)) \sH (A_j) \nonumber \\ 
			&+&  \sH (\widetilde{A}_{j-1})(I + \saug_2(B_j)) + \saug_2 (B_j).		
		\end{eqnarray*}		
	
\end{enumerate}
\end{lem}

\begin{proof}

For each $A_j$, $\df A_j = A_j^2$, so the derivation property of $H$ implies $\saug_1 - \saug_2(A_j) = \sH \circ \df (A_j)$ is equivalent to $\saug_1 - \saug_2 (A_j) = \sH (A_j) \saug_2(A_j) + \saug_1(A_j) \sH (A_j)$. By rearranging terms, we obtain part (2) of the Lemma.

For each $B_j$, $\df B_j = (I + B_j)A_j + \widetilde{A}_{j-1}(I + B_j)$, so the derivation property of $H$ implies $\saug_1 - \saug_2(B_j) = \sH \circ \df (B_j)$ is equivalent to:

\begin{eqnarray*}
	\saug_1 - \saug_2 (B_j) &=& (I + \saug_1(B_j)) \sH (A_j) + \sH (\widetilde{A}_{j-1})(I + \saug_2(B_j)) + \nonumber \\
		&+& \sH (B_j) \saug_2 (A_j) + \saug_1 (\widetilde{A}_{j-1}) \sH (B_j). 			
\end{eqnarray*}

\noindent By rearranging terms, we obtain part (3) of the Lemma.

\end{proof}

We will use the matrix equations of Lemma~\ref{lem:dipped-ch-htpy} to show that $\Psi$ maps equivalent MCSs to the same augmentation class. We will restrict our attention to occ-simple augmentations and in nearly every situation build chain homotopies that satisfy $\sH(B_j) = 0$ for all $j$. These added assumptions simplify the task of constructing chain homotopies. The next Corollary follows immediately from Lemma~\ref{lem:dipped-ch-htpy} and Definition~\ref{defn:simple-aug}.

\begin{cor}
	\label{cor:dipped-ch-htpy}
	Suppose $\sNgresD, \saug_1, \saug_2$ and $H$ are as in Lemma~\ref{lem:dipped-ch-htpy}. If $\saug_1$ and $\saug_2$ are occ-simple and $H(B_j) = 0$ for all $j$, then $\sH$ is a chain homotopy between $\saug_1$ and $\saug_2$ if and only if:

\begin{enumerate}
	\item For all $q_r$ and $z_s$, $\saug_1 - \saug_2 = \sH \circ \df$,
	\item For all $A_j$, 
		\begin{equation*}
		\saug_1 (A_j) = (I + \sH (A_j)) \saug_2(A_j) (I + \sH (A_j))^{-1},
		\end{equation*}
	\item For $B_j$ with $I_j$ of type (1),
		\begin{equation*}
			\saug_1(B_j) + (I + \saug_1(B_j)) \sH (A_j) = \sH (A_{j-1})(I + \saug_2(B_j)) + \saug_2 (B_j),	
		\end{equation*}		
	
	\item For $B_j$ with $I_j$ of type (2), (3), or (4),
		\begin{equation*}
		\sH (A_j) = \sH (\widetilde{A}_{j-1}).
		\end{equation*}	
\end{enumerate}
\end{cor}

We are now in a position to prove that $\widehat{\Psi} : \sDMCSeq \to Aug^{ch}(\sNgres)$ defined by $\widehat{\Psi} ([\sMCS]) = \Psi(\sMCS)$ is well-defined.

\begin{figure}
\labellist
\small\hair 2pt
\pinlabel {1} [br] at 114 366
\pinlabel {3} [br] at 114 287
\pinlabel {5} [br] at 114 207
\pinlabel {7} [br] at 114 127
\pinlabel {9} [br] at 114 46
\pinlabel {2} [br] at 344 366
\pinlabel {4} [br] at 344 287
\pinlabel {6} [br] at 344 207
\pinlabel {8} [br] at 344 127
\pinlabel {10} [br] at 346 46
\pinlabel {$j$} [tl] at 44 331
\pinlabel {$j$} [tl] at 44 251
\pinlabel {$j$} [tl] at 44 171
\pinlabel {$j$} [tl] at 28 91
\pinlabel {$j$} [tl] at 28 11
\pinlabel {$j$} [tl] at 164 331
\pinlabel {$j$} [tl] at 164 251
\pinlabel {$j$} [tl] at 164 171
\pinlabel {$j$} [tl] at 150 91
\pinlabel {$j$} [tl] at 150 11
\pinlabel {$j$} [tl] at 273 331
\pinlabel {$j$} [tl] at 273 251
\pinlabel {$j$} [tl] at 263 171
\pinlabel {$j$} [tl] at 258 91
\pinlabel {$j$} [tl] at 258 11
\pinlabel {$j$} [tl] at 394 331
\pinlabel {$j$} [tl] at 394 251
\pinlabel {$j$} [tl] at 386 171
\pinlabel {$j$} [tl] at 380 91
\pinlabel {$j$} [tl] at 380 11
\endlabellist
\centering
\includegraphics[scale=.55]{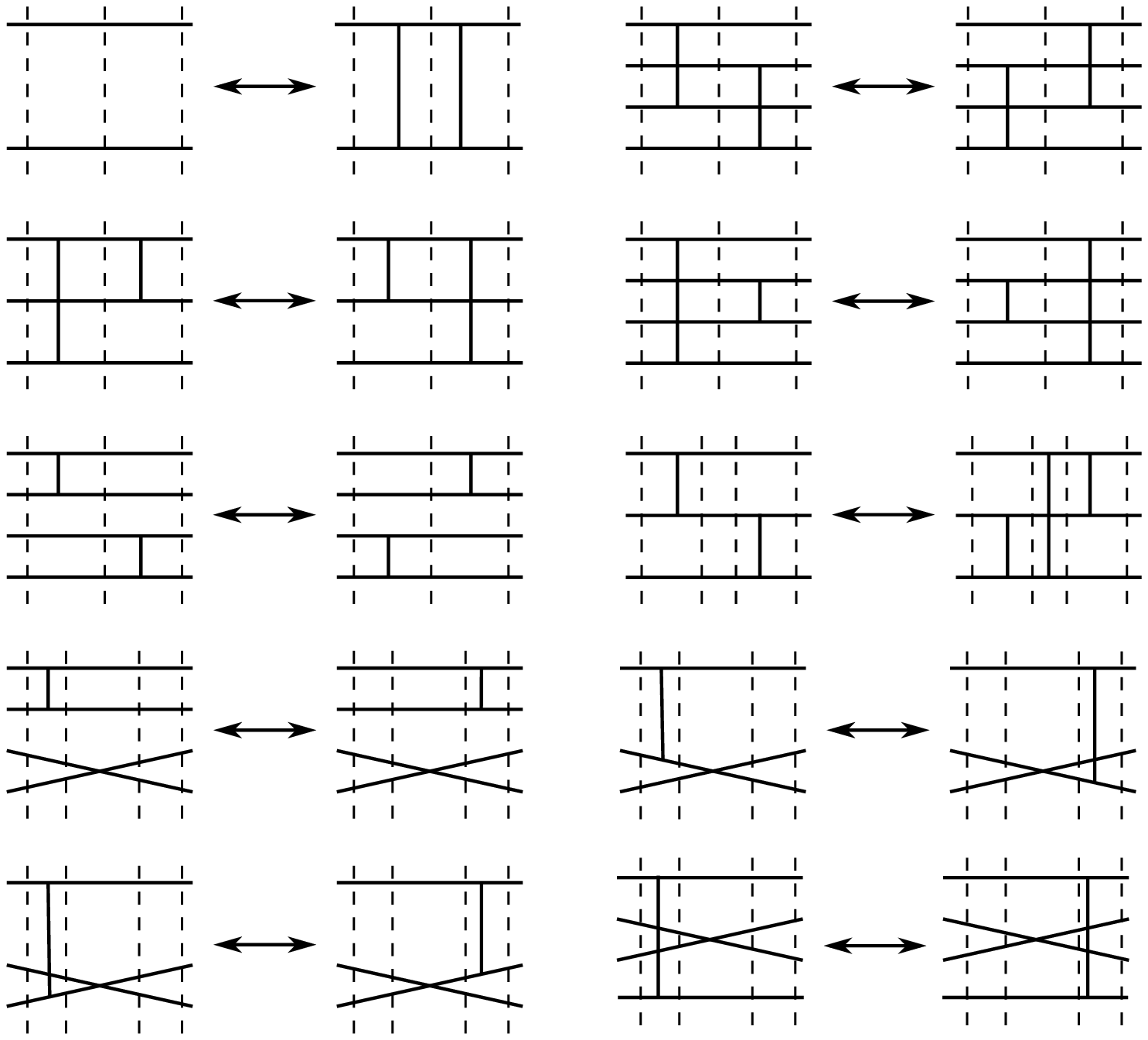}
\caption[MCS equivalence moves 1 - 10 with dips indicated.]{MCS equivalence moves 1 - 10 with dip locations indicated.}
\label{f:MCS-equiv-2-dotted}
\end{figure}

\begin{figure}
\labellist
\small\hair 2pt
\pinlabel {11} [br] at 117 185
\pinlabel {12} [br] at 345 185
\pinlabel {13} [br] at 117 110
\pinlabel {14} [br] at 345 110
\pinlabel {15} [br] at 117 30
\pinlabel {16} [br] at 345 37
\pinlabel {$j$} [tl] at 26 150
\pinlabel {$j$} [tl] at 26 70
\pinlabel {$j$} [tl] at 34 5
\pinlabel {$j$} [tl] at 148 150
\pinlabel {$j$} [tl] at 148 70
\pinlabel {$j$} [tl] at 156 5
\pinlabel {$j$} [tl] at 270 150
\pinlabel {$j$} [tl] at 270 70
\pinlabel {$j$} [tl] at 265 5
\pinlabel {$j$} [tl] at 391 150
\pinlabel {$j$} [tl] at 391 70
\pinlabel {$j$} [tl] at 388 5
\endlabellist
\centering
\includegraphics[scale=.55]{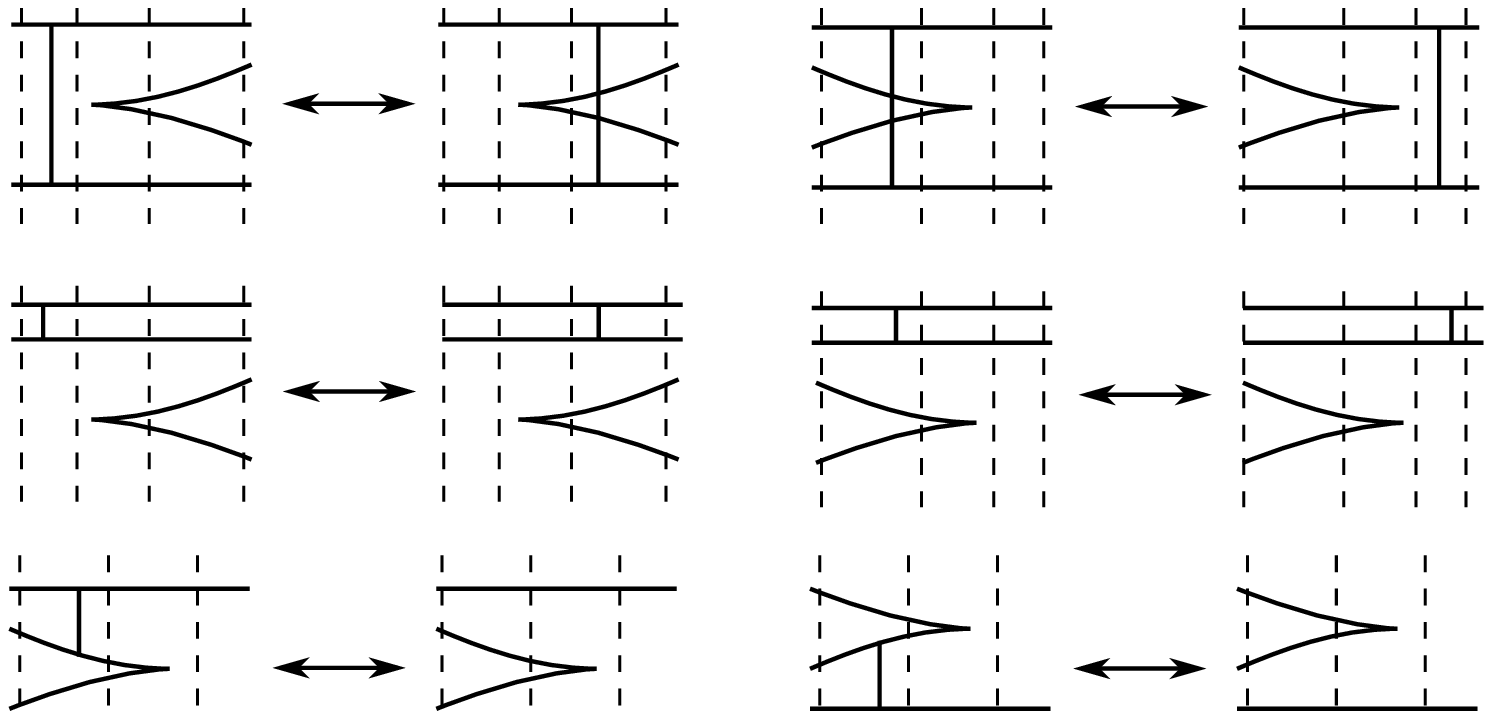}
\caption[MCS equivalence moves 11 - 16 with dips indicated.]{MCS equivalence moves 11 - 16 with dip locations indicated.}
\label{f:MCS-equiv-3-dotted}
\end{figure}

\begin{figure}
\labellist
\small\hair 2pt
\pinlabel {17} [tl] at 190 112
\pinlabel {$j$} [tl] at 11 18
\pinlabel {$j$} [tl] at 242 18
\endlabellist
\centering
\includegraphics[scale=.4]{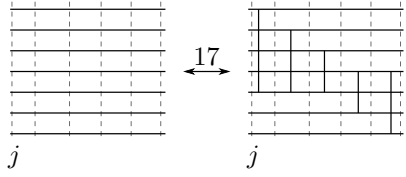}
\caption[MCS move 17 with dips indicated.]{MCS move 17 with dip locations indicated.}
\label{f:MCS-equiv-explosion-dotted}
\end{figure}

\begin{lem}
	\label{lem:mcs-equiv-aug-equiv}
If $\sMCS_1$ and $\sMCS_2$ are equivalent, then $\Psi(\sMCS_1) = \Psi(\sMCS_2)$.
\end{lem}

\begin{proof}
It is sufficient to suppose $\sMCS_1$ and $\sMCS_2$ differ by a single MCS move. Let $\saug_{\sMCS_1} \in Aug(\sNgres^{d_1})$ and $\saug_{\sMCS_2} \in Aug(\sNgres^{d_2})$ be as defined in Lemma~\ref{lem:mcs-aug-dipped}. Since each MCS move is local, we may assume $\sNgres^{d_1}$ and $\sNgres^{d_2}$ have identical dips outside of the region of the MCS move. 

We compare the chain homotopy classes of $\saug_{\sMCS_1}$ and $\saug_{\sMCS_2}$ by extending $\saug_{\sMCS_1}$ and $\saug_{\sMCS_2}$ to augmentations on a third dipped diagram $\sNgres^{d}$. In the case of MCS moves 1-16, we form $\sNgres^{d}$ by adding 0, 1, or 2 additional dips to $\sNgres^{d_1}$ and $\sNgres^{d_2}$. MCS move 17 may require the addition of many more dips. The dotted lines in Figures~\ref{f:MCS-equiv-2-dotted}, \ref{f:MCS-equiv-3-dotted} and \ref{f:MCS-equiv-explosion-dotted} indicate the locations of dips in $\sNgres^{d_1}$ and $\sNgres^{d_2}$ and the additional dips needed to form $\sNgres^{d}$. The index $j$ denotes the dip $D_j$. Table~\ref{t:add-dips} indicates which dotted lines in Figures~\ref{f:MCS-equiv-2-dotted} and \ref{f:MCS-equiv-3-dotted} represent new dips added to form $\sNgres^{d}$. 

\begin{table}
\centering
\begin{tabular}{|c|l|l|}
\hline
\multicolumn{3}{|c|}{Adding dips to $\sNgres^{d_1}$ and $\sNgres^{d_2}$.} \\
\hline
Move & $\sNgres^{d_1}$ & $\sNgres^{d_2}$ \\ \hline
1 & $j, j+1$ & -- \\
6 & $j+1$ & -- \\
7 - 14 & $j+2$ & $j$ \\
15 & -- & $j$ \\
16 & -- & $j$ \\
\hline
\end{tabular}
\caption{Entries indicate the dotted lines in Figure~\ref{f:MCS-equiv-2-dotted} and Figure~\ref{f:MCS-equiv-3-dotted} that correspond to dips added to form $\sNgresD$.}
\label{t:add-dips}
\end{table}

As we add dips to $\sNgres^{d_1}$ and $\sNgres^{d_2}$, we extend $\saug_{\sMCS_1}$ and $\saug_{\sMCS_2}$ by 0. By Corollary~\ref{lem:extend-by-0} we can keep track of the extensions of $\saug_{\sMCS_1}$ and $\saug_{\sMCS_2}$ after each dip. In fact, the extensions will be occ-simple. We let $\widetilde{\saug}_{\sMCS_1}, \widetilde{\saug}_{\sMCS_2} \in Aug(\sNgres^d)$ denote the extensions of $\saug_{\sMCS_1}$ and $\saug_{\sMCS_2}$. We use the matrix equations in Corollary~\ref{cor:dipped-ch-htpy} and Lemma~\ref{lem:dipped-ch-htpy} to construct chain homotopies between $\widetilde{\saug}_{\sMCS_1}$ and $\widetilde{\saug}_{\sMCS_2}$. In particular, for all MCS moves except move 6 and 17, the chain homotopy $H$ is given in Table~\ref{t:MCS-moves}. The chain homotopy $H$ for MCS move 6 sends all of the crossings to $0$ except in the case of $A_{j-1}$ and $A_j$, where:

\begin{equation*}
\sH(A_{j-1}) = \saug_1 (B_{j-1}) +  \saug_2 (B_{j-1}) +  \saug_1 (B_{j-1}) \widetilde{\saug_2} (B_{j-1}) \mbox{, and }
\end{equation*}
\begin{equation*}
\sH(A_j) = \sH(A_{j-1}) +  \saug_2 (B_{j-1}) + \sH(A_{j-1}) \saug_2 (B_{j-1}).
\end{equation*}

The case of MCS move 17 is slightly more complicated. In particular, just to the left of the region of the MCS move, the chain complex $(C_j, \df_j)$ in $\sMCS_1$ and $\sMCS_2$ has a pair of generators $y_l < y_k$ such that $|y_l| = |y_k| + 1$. Let $y_{u_1} < y_{u_2} < \hdots < y_{u_s}$ denote the generators of $C_{j}$ satisfying  $\langle \df y | y_{k} \rangle = 1$. Let  $y_{v_r} <  \hdots < y_{v_1} < y_i$ denote the generators of $C_{j}$ appearing in $\df_j y_{l}$. MCS move 17 introduces the handleslide marks $E_{k, v_r}, \hdots E_{k, v_1}, E_{u_1, l}, \hdots, E_{u_s, l}$. We assume that $\sMCS_2$ includes these marks and $\sMCS_1$ does not. Then we introduce $r + s$ dips to the right of $D_j$ in $\sNgres^{d_1}$. As we add dips to $\sNgres^{d_1}$, we extend $\saug_{\sMCS_1}$ by 0. By Corollary~\ref{lem:extend-by-0} we can keep track of the extensions of $\saug_{\sMCS_1}$ after each dip. The chain homotopy $H$ between $\widetilde{\saug}_{\sMCS_1}$ and $\widetilde{\saug}_{\sMCS_2}$ is defined as follows. For $1 \leq i \leq r + s - 1$, $\sH(A_{j+ i}) = \sum_{k=1}^{i} \saug_{\sMCS_2}(B_{j+k})$, $H (B_{j+r+s}) = \sHSM_{k,l}$, and all of the other crossings are sent to $0$.

Since each $\sH$ defined above has few non-zero entries, it is easy to check that $H$ has support on generators with grading $-1$. Thus we need only check that the extension of $H$ by linearity and the derivation product property satisfies $\widetilde{\saug}_{\sMCS_1} - \widetilde{\saug}_{\sMCS_2} = \sH \circ \df$. In the case of MCS moves 1-16, this is equivalent to checking that $H$ solves the matrix equations in Corollary~\ref{cor:dipped-ch-htpy}. In the case of MCS move 17, we must check that $H$ solves the matrix equations in Lemma~\ref{lem:dipped-ch-htpy}. We leave this task to the reader. 

The maps $\sH$ given above and in Table~\ref{t:MCS-moves} were constructed using the following process. The augmentations $\widetilde{\saug}_{\sMCS_1}$ and $\widetilde{\saug}_{\sMCS_2}$ are equal on crossings outside of the region of the MCS move, thus $\widetilde{\saug}_{\sMCS_1} - \widetilde{\saug}_{\sMCS_2} = 0$ on these crossings. Hence, $\sH = 0$ satisfies $\widetilde{\saug}_{\sMCS_1} - \widetilde{\saug}_{\sMCS_2} = \sH \circ \df$ on these crossings. Within the region of the MCS moves we can write down explicit matrix equations relating $\widetilde{\saug}_{\sMCS_1}$ and $\widetilde{\saug}_{\sMCS_2}$. In particular, $\widetilde{\saug}_{\sMCS_1} = \widetilde{\saug}_{\sMCS_2}$ to the left of the MCS move and, within the region of the move, $\widetilde{\saug}_{\sMCS_1}$ and $\widetilde{\saug}_{\sMCS_2}$ are related by the matrix equations in Lemma~\ref{lem:aug-eq}. From these equations, we are able to define $\sH$.

We now show that $[\widetilde{\saug}_{\sMCS_1}] =  [\widetilde{\saug}_{\sMCS_2}]$ implies $\Psi(\sMCS_1) = \Psi(\sMCS_2)$. Let $v$ denote any dipping/undipping path from $\sNgres^{d}$ to $\sNgres$. By Lemma~\ref{lem:dip-path}, the definition of $\Psi : \sDMCS \to Aug^{ch}(\sNgres)$ is independent of dipping/undipping paths. We calculate $\Psi(\sMCS_1)$ (resp. $\Psi(\sMCS_2)$) using the path that travels from $\sNgres^{d_1}$ (resp. $\sNgres^{d_2}$) to $\sNgres$ by adding dips as specified above to create $\sNgres^{d}$ and then traveling along $v$. The first segment of this path maps $\saug_{\sMCS_1}$ (resp. $\saug_{\sMCS_2}$) to $\widetilde{\saug}_{\sMCS_1}$ (resp. $ \widetilde{\saug}_{\sMCS_2}$). Since $\widetilde{\saug}_{\sMCS_1} \schequiv \widetilde{\saug}_{\sMCS_2}$, the stable-tame isomorphism associated to the path $v$ maps $\widetilde{\saug}_{\sMCS_1}$ and $ \widetilde{\saug}_{\sMCS_2}$ to the same chain homotopy class. Thus, $\Psi(\sMCS_1) = \Psi(\sMCS_1)$, as desired. 
\end{proof}

Finally, note that Lemma~\ref{lem:mcs-equiv-aug-equiv} and Theorem~\ref{lem:psi-surjectivity} imply that the map $\widehat{\Psi} : \sDMCSeq \to Aug^{ch}(\sNgres)$ defined by $\widehat{\Psi} ([\sMCS]) = \Psi(\sMCS)$ is a well-defined, surjective map. Thus we have proven Theorem~\ref{thm:big-map}.

\begin{table}
\centering
\begin{tabular}{|c|l|l|}
\hline
\multicolumn{3}{|c|}{Chain Homotopy $\sH$ between $\saug_1$ and $\saug_2$.} \\
\hline
Move & $\sH (A_j)$ & $\sH(A_{j+1})$ \\ \hline
1 & $\saug_2(B_j)$ & 0 \\
2 & $\saug_1(B_j) + \saug_2(B_j)$ & 0 \\
3 & $\saug_1(B_j) + \saug_2(B_j)$ & 0 \\
4 & $\saug_1(B_j) + \saug_2(B_j)$ & 0 \\
5 & $\saug_1(B_j) + \saug_2(B_j)$ & 0 \\
7 & $\saug_1(B_j)$ & $\saug_1(B_j)$ \\
8 & $\saug_1(B_j)$ & $\saug_2(B_{j+2})$ \\
9 & $\saug_1(B_j)$ & $\saug_2(B_{j+2})$ \\
10 & $\saug_1(B_j)$ & $\saug_1(B_j)$ \\
11 & $\saug_1(B_j)$ & $\saug_2(B_{j+2})$ \\
12 & $\saug_1(B_j)$ & $\saug_2(B_{j+2})$ \\
13 & $\saug_1(B_j)$ & $\saug_2(B_{j+2})$ \\
14 & $\saug_1(B_j)$ & $\saug_2(B_{j+2})$ \\
15 & $\saug_1(B_j)$ & 0 \\
16 & $\saug_1(B_j)$ & 0 \\
\hline
\end{tabular}
\caption{In all cases, $\sH = 0$ on all other crossings.}
\label{t:MCS-moves}
\end{table}

\subsection{Two standard forms for MCSs}
\label{sec:std-forms}

An MCS $\sMCS$ encodes both a graded normal ruling $\sgnr_{\sMCS}$ on $\sfront$ and an augmentation $\saug_{\sMCS}$ on $\sNgresD$. In this section, we formulate two algorithms using MCS moves that highlight these connections. When combined the algorithms provide a map from $\sDMCS$ to $Aug(\sNgres)$ which, when passed to equivalence classes, corresponds to $\widehat{\Psi}$. The upside is that this map does not require dipped diagrams or keeping track of chain homotopy classes of augmentations. We are also able to reprove the many-to-one correspondence between augmentations and graded normal rulings found in \cite{Ng2006}.

\subsubsection{Sweeping collections of handleslide marks}
\label{sec:sweeping}

The algorithms we define sweep handleslide marks from left to right in $\sfront$. During this process, we let $V$ denote the collection of handleslide marks being swept. Suppose $V$ sits away from the crossings and cusps of $\sfront$. Near $V$, label the strands of $\sfront$ from bottom to top with the integers $1, \hdots, n$. We define $v_{k,l} \in \zz_2$ to be $1$ if and only if the collection $V$ includes a handleslide mark between strands $k$ and $l$, $k > l$. We abuse notation slightly by allowing $v_{k,l}$ to denote the handleslide mark as well. We order the marks in $V$ as follows. If $k' < k$ and $v_{k,l} = v_{k',l'} = 1$, then $v_{k',l'}$ appears to the left of $v_{k,l}$ in $V$. If $l' < l$ and $v_{k,l} = v_{k,l'} = 1$, then $v_{k,l'}$ appears to the left of $v_{k,l}$ in $V$. The following moves are used in the $S\bar{R}$ and $\sAugform$-algorithms and describe interactions between $V$ and handleslide marks, cusps, and crossings. 

\begin{figure}
\centering
\includegraphics[scale=.4]{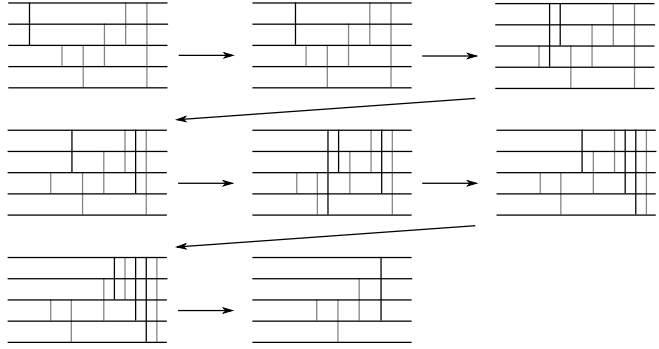}
\caption[Sweeping into $V$ from the left.]{Sweeping a handleslide mark into $V$ from the left. The grey marks are those contained in the original $V$.}
\label{f:SR-algorithm-left-hs}
\end{figure}

{\bf Move I:} (Sweeping a handleslide from the left of $V$ into/out of $V$.) Suppose a handleslide mark $h$ sits to the left of $V$ between strands $k$ and $l$, with $k > l$. We use MCS moves 2 - 6 to commute $h$ past the handleslides in $V$. In order to commute $h$ past a handleslide of the form $v_{l, i}$, we must use MCS move 6 and, thus, create a new handleslide mark $h'$ between strands $k$ and $i$. The ordering on $V$ allows us to commute $h'$ to the right so that it becomes properly ordered with the other marks in $V$. If $v_{k,i}=1$, $h'$ cancels the handleslide $v_{k,i}$ by MCS move 1. We continue this process of commuting $h$ past each of the strands that begin on strand $l$. Once we have done so, we can commute $h$ past the other marks in $V$, without introducing new marks, until it becomes properly ordered in $V$; see Figure~\ref{f:SR-algorithm-left-hs}. Since each of MCS moves 1 - 6 is reversible, we are able to sweep an existing handleslide mark $v_{k,l}$ out of $V$ to the left using the same process.

{\bf Move II:} (Sweeping a handleslide from the right of $V$ into/out of $V$.) The process to sweep a handleslide from the right of $V$ into/out of $V$ is analogous to the process in move I.
	
{\bf Move III:} (Sweeping $V$ past a crossing $q$ between strands $i+1$ and $i$ assuming $v_{i+1,i} = 0$.) Sweep all of the handleslides of $V$ past $q$ using MCS moves 7 - 10 and, if necessary, reorder the handleslide marks so that $V$ is properly ordered.
 
{\bf Move IV:} (Sweeping $V$ past a right cusp between strands $i+1$ and $i$ assuming there are no handleslide marks in $V$ beginning or ending on $i+1$ or $i$.) The handleslides in $V$ sweep past the right cusp using MCS moves 12 and 14. $V$ stays properly ordered during this process. 

\subsubsection{The $S\bar{R}$-form of an MCS}
\label{sec:SR-form}

Given an MCS $\sMCS$, the $S\bar{R}$-algorithm results in an MCS with handleslide marks near switched crossings and some graded returns of $N_{\sMCS}$ and, away from these crossings, the chain complexes are in simple form. The resulting MCS is called the \emph{$S\bar{R}$-form of $\sMCS$}. The $S \bar{R}$-form was inspired by discussions at the September 2008 AIM workshop, the work of Fuchs and Rutherford in \cite{Fuchs2008}, and the work of Ng and Sabloff in \cite{Ng2006}.

\begin{defn}
Given an MCS $\sMCS$, we say a switch or return in $N_{\sMCS}$ is \emph{simple} if the ordered chain complexes of $\sMCS$ are in simple form before and after the switch or return and in a neighborhood of the crossing the MCS is arranged as in Figure~\ref{f:marked-switches-occ} or Figure~\ref{f:marked-returns}. A simple return is \emph{marked} if it corresponds to one of the three arrangements in the top row of Figure~\ref{f:marked-returns}.

An MCS is in \emph{$S\bar{R}$-form} if all of its switched crossings and graded returns are simple, and, besides the handleslides near simple switches and returns, no other handleslides appear in $\sMCS$
\end{defn}

\begin{figure}
\centering
\includegraphics[scale=.5]{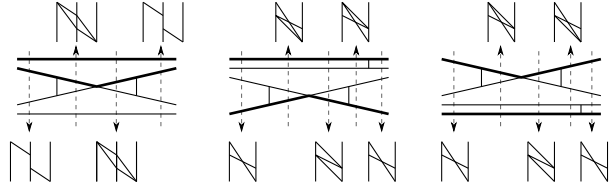}
\caption[Simple switches.]{Simple switches.}
\label{f:marked-switches-occ}
\end{figure}

\begin{figure}
\centering
\includegraphics[scale=.5]{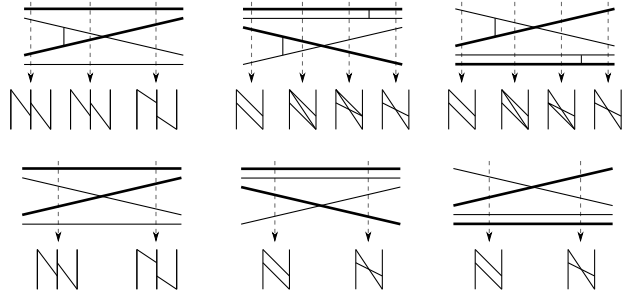}
\caption[Simple returns.]{Simple returns.}
\label{f:marked-returns}
\end{figure}

\begin{thm}
	\label{thm:SR-form}
	Every $\sMCS \in \sDMCS$ is equivalent to an MCS in $S\bar{R}$-form.
\end{thm}

\begin{proof}

We define an algorithm to sweep handleslide marks of $\sMCS$ from left to right in the front projection, beginning at the left-most cusp and working to the right. As we sweep handleslides marks to the right, we ensure that the MCS we leave behind is in $S\bar{R}$-form. By definition, the first chain complex of $\sMCS$ is in simple form and so the collection $V$ begins empty.

We consider each case of $V$ encountering a handleslide mark, cusp, or crossing. In the case of a handleslide mark, we incorporate the mark into $V$ using move II. In the case of a left cusp, we sweep $V$ past the cusp using MCS moves 11 and 13. In either of these cases, if the MCS is in $S\bar{R}$-form to the left of $V$ before the handleslide mark or left cusp, then it is also in $S\bar{R}$-form after we sweep past. 

Moves III and IV ensure that in certain situations we may sweep $V$ past a crossing or right cusp so that the resulting MCS is in $S\bar{R}$-form to the left of the singularity. Suppose we arrive at a crossing $q$ between strands $i+1$ and $i$ and $v_{i+1,i} = 1$. We use the following algorithm to ensure that $q$ is a simple switch or simple return after we push $V$ past $q$. We sweep the handleslide $v_{i+1,i}$ to the left of $V$ using move I. We let $h$ denote this handleslide. Now $v_{i+1,i} = 0$ and so we sweep $V$ past $q$ using move III. If we suppose the ordered chain complex of $\sMCS$ is in simple form just before $h$, then around $h$, $\sMCS$ looks like one of the 6 cases in Figure~\ref{f:marked-switches-returns}. The top three cases indicate that $q$ is a switch and the bottom three cases indicate that $q$ is a return. 

\begin{enumerate}
	\item[\textbf{Switches:}] 
	If $q$ is a switch, we use MCS move 1 to introduce two handleslides just to the right of the crossing between strands $i+1$ and $i$. If $q$ is a switch of type (2) or (3), we also introduce two handleslides between the companion strands of strands $i+1$ and $i$. Finally, we move the right-hand mark of each new pair of marks into $V$ using move I. The resulting MCS now has a simple switch at $q$; see Figure~\ref{f:marked-switches-occ}. 	
	\item[\textbf{Returns:}] If $q$ is a return of type (1), then $q$ is a simple return as in Figure~\ref{f:marked-returns}. If $q$ is a return of type (2) or (3), we use MCS move 1 to introduce two handleslides just to the right of the crossing between the companion strands of strands $i+1$ and $i$ and use move I to sweep the right-hand mark into $V$. The resulting MCS now has a simple return at $q$; Figure~\ref{f:marked-returns}. 			
\end{enumerate}

\begin{figure}
\centering
\includegraphics[scale=.5]{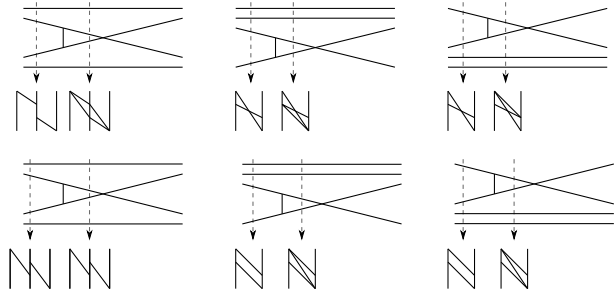}
\caption[Sweeping $V$ past a crossing with a handleslide.]{Sweeping $V$ past a crossing with a handleslide immediately to the left. The top row will be switches and the bottom will be marked returns. In each row, we label the cases (1), (2), (3) from left to right.}
\label{f:marked-switches-returns}
\end{figure} 

Suppose $V$ arrives at a right cusp between strands $i+1$ and $i$ and there are handleslide marks in $V$ beginning or ending on $i+1$ or $i$. Suppose also that the MCS is in $S \bar{R}$-form to the left of $V$. By applying move I possibly many times, we sweep the marks that end on strand $i$, begin on strand $i+1$, begin on strand $i$, or end on strand $i+1$ out of $V$. The order in which these moves occur corresponds to the order given in the previous sentence. The Maslov potentials on strands $i+1$ and $i$ differ by one, thus $v_{i+1,i}=0$. As a consequence, these moves do not introduce new handleslides ending or beginning on $i+1$ or $i$. Next we use move IV to sweep $V$ past the right cusp. 

Now we remove the marks that have accumulated at the right cusp. We remove the marks ending on $i + 1$ and beginning on $i$ using MCS moves 15 and 16. Let $h_1, \hdots, h_n$ denote the handleslides ending on strand $i$, ordered from left to right, and let $g_1, \hdots, g_m$ denote the handleslides beginning on strand $i+1$. By assumption, the MCS is simple just to the left of $h_1$. Thus just before and after $h_1$ the pairing and MCS must look like one of the three cases in Figure~\ref{f:SR-algorithm-death-1} (a). If $h_1$ is of type 1 or 2, we can eliminate it using MCS move 17. Suppose $h_1$ is of type 3, $h_1$ begins on strand $l$, and generator $l$ is paired with generator $k$ in the ordered chain complex just before $h$. Introduce a new handleslide, denoted $h'$, between strands $k$ and $i+1$ using MCS move 15. Move $h_1'$ past each of $g_1, \hdots, g_m$ using MCS move 6. Each time we perform such a move, we create a new handleslide mark which is then swept past the cusp using MCS move 12 and incorporated into $V$ using move I. Commute $h_1'$ past $h_2, \hdots, h_m$ using MCS moves 2 - 4. Now $h_1$ and $h_1'$ look like Figure~\ref{f:SR-algorithm-death-1} (b) and we can remove both using MCS move 17.

\begin{figure}
\labellist
\small\hair 2pt
\pinlabel {(a)} [tl] at 207 9
\pinlabel {(b)} [tl] at 580 9
\endlabellist
\centering
\includegraphics[scale=.5]{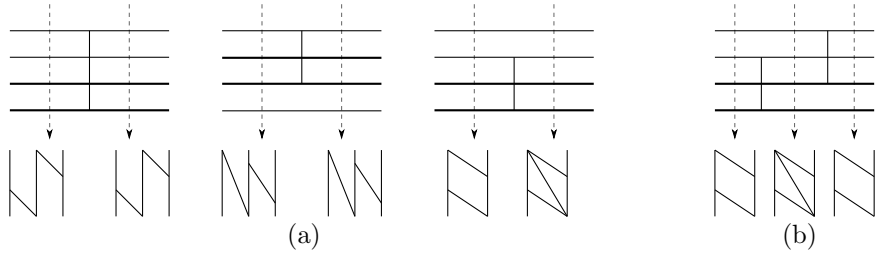}
\caption[Removing a handleslide at a cusp.]{(a) The 3 possible local neighborhoods of the handleslide $h_1$. (b) MCS move 17 at a right cusp after introducing a new handleslide mark. In both (a) and (b), the two dark lines correspond to the strands entering the right cusp.}
\label{f:SR-algorithm-death-1}
\end{figure}

Using this procedure, we eliminate all of $h_1, \hdots, h_n$. The argument to eliminate $g_1, \hdots, g_m$ is essentially identical. Either we can eliminate $g_1$ using MCS move 17 or we can push on a handleslide using MCS move 15 and then eliminate the pair with MCS move 17. After eliminating $h_1, \hdots, h_n$ and $g_1, \hdots, g_m$, the MCS is in simple form just before and just after the right cusp. Hence, the resulting MCS is in $S\bar{R}$-form to the left of $V$. 

As we carry out this algorithm from left to right, each time we encounter a handleslide, crossing or cusp we are able to sweep $V$ past so that the MCS we leave behind is in $S\bar{R}$-form.
\end{proof}

As a corollary, we have the following:

\begin{cor}
	\label{cor:gnr-SR-form}
	Let $\sgnr$ be a graded normal ruling on $\sfront$ with switched crossings $q_1, \hdots, q_n$ and graded returns $p_1, \hdots, p_m$. Then $\sgnr$ is the graded normal ruling of $2^{m}$ MCSs in $S\bar{R}$-form. Hence, $\sgnr$ is the graded normal ruling of at most $2^{m}$ MCS classes in $\sDMCSeq$. 
\end{cor}

\begin{proof}
By Lemmata~\ref{lem:MCS-ruling} and \ref{prop:gnr-uniq}, each MCS equivalence class has an associated graded normal ruling and, by Theorem~\ref{thm:SR-form}, each MCS equivalence class has at least one representative in $S\bar{R}$-form. Thus, it is sufficient to show there are exactly $2^{m}$ MCSs in $S\bar{R}$-form with graded normal ruling $\sgnr$. The set of marked returns of any MCS in $S\bar{R}$-form with graded normal ruling $\sgnr$ gives a subset of $\{ p_1, \hdots, p_m \}$. Conversely, by adding handleslide marks to $\sfront$ as dictated by $\sgnr$, Figure~\ref{f:marked-switches-occ}, and Figure~\ref{f:marked-returns}, any $ R \subset \{ p_1, \hdots, p_m \}$ will determine an MCS in $S\bar{R}$-form with graded normal ruling $\sgnr$ and marked returns corresponding to $R$.
\end{proof}

\subsubsection{The $\sAugform$-form of an MCS}

The $\sAugform$-algorithm is similar to the sweeping process in the $S\bar{R}$-algorithm. However, we no longer introduce new handleslide marks after crossings. The resulting MCS, called the $\sAugform$-form, has marks to the left of some graded crossings and nowhere else. The set of MCSs in $\sAugform$-form are in bijection with the augmentations on $\sNgres$. In the algorithm, we still have to address the issue of mark accumulating at right cusps. If we begin with an MCS in $S\bar{R}$-form, then this is easy to do.

\begin{defn}
An MCS $\sMCS \in \sDMCS$ is in $\sAugform$-form if:
\begin{enumerate}
	\item Outside a small neighborhood of the crossings of $\sfront$, $\sMCS$ has no handleslide marks; and
	\item Within a small neighborhood of a crossings $q$, either $\sMCS$ has no handleslide marks or it has a single handleslide mark to the left of $q$ between the strands crossings at $q$.
\end{enumerate}
The left figure in Figure~\ref{f:mcs-aug-trefoil} is in $\sAugform$-form.
\end{defn}

\begin{thm}
	\label{thm:C-form}
	Every $\sMCS \in \sDMCS$ is equivalent to an MCS in $\sAugform$-form.
\end{thm}

\begin{proof}

Let $\sMCS \in \sDMCS$ and use the $S\bar{R}$-algorithm to put $\sMCS$ in $S\bar{R}$-form. The $\sAugform$-algorithm sweeps handleslide marks from left to right in the marked front of $\sMCS$. As in the $S\bar{R}$-algorithm, we will use $V$ to keep track of handleslide marks and the moves I - IV to sweep $V$ past handleslides, crossings and cusps. In the case of handleslides and left cusps, the $\sAugform$-algorithm works the same as in the $S\bar{R}$-algorithm.

Suppose $V$ arrives at a crossing $q$ between strands $i+1$ and $i$. If $v_{i+1,i} = 0$, then we sweep $V$ past $q$ as described in move III. If $v_{i+1,i} = 1$, then we sweep the handleslide $v_{i+1,i}$ to the left of $V$ and then sweep $V$ past $q$ using move III.

Suppose $V$ arrives at a right cusp $q$ between strands $i+1$ and $i$. If no handleslides begin or end on $i+1$ or $i$, then we sweep $V$ past the right cusp using move IV. Otherwise, we sweep the handleslides that end on strand $i$, begin on strand $i+1$, begin on strand $i$, or end on strand $i+1$ out of $V$ to the right using move II. The order in which these moves occur corresponds to the order given in the previous sentence. Let $h_1, \hdots, h_n$ denote the handleslides ending on strand $i$, ordered from right to left, and let $g_1, \hdots, g_m$ denote the handleslides beginning on strand $i+1$, also ordered from right to left. Since the MCS is in $S\bar{R}$-form to the right of $h_1$, we know that the chain complex between $h_1$ and the right cusp is simple. Thus, just before and after $h_1$ the pairing and MCS must look like one of the three cases in Figure~\ref{f:C-algorithm-death-1} (a). If $h_1$ is of type (1) or (2) in Figure~\ref{f:C-algorithm-death-1} (a), we can eliminate it using MCS move 17. Suppose $h_1$ is of type (3), $h_1$ begins on strand $l$ and in the ordered chain complex just after $h_1$ generator $l$ is paired with generator $k$. Use MCS move 17 to introduce two new handleslides to the left of $h_1$; see Figure~\ref{f:C-algorithm-death-1} (b). One new handleslide is between strands $l$ and $i$. We remove $h_1$ and this new handleslide using MCS move 1. The second handleslide is between strands $k$ and $i+1$ and can be removed using MCS move 15. We eliminate all of $h_1, \hdots, h_n$ using this process. The argument to eliminate $g_1, \hdots, g_m$ is essentially identical. After eliminating $h_1, \hdots, h_n$ and $g_1, \hdots, g_m$, we use MCS move 15 and 16 to remove all of the handleslides ending on $i + 1$ and beginning on $i$. We sweep $V$ past the right cusp using move IV and continue to the right. 

As we progress from left to right in $\sfront$, marks remain to the immediate left of some graded crossings. Hence, after sweeping $V$ past the right-most cusp, we are left with an MCS in $\sAugform$-form.
		
\begin{figure}
\labellist
\small\hair 2pt
\pinlabel {(a)} [tl] at 200 9
\pinlabel {(b)} [tl] at 610 9
\endlabellist
\centering
\includegraphics[scale=.4]{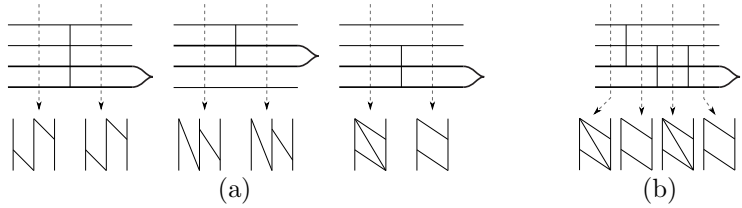}
\caption[Removing a handleslide at a cusp.]{(a) The 3 possible local neighborhoods of the handleslide $h_1$. (b) MCS move 17 at a right cusp introducing two new handleslides. The two dark lines correspond to the strands entering the right cusp.}
\label{f:C-algorithm-death-1}
\end{figure}
		
\end{proof}

The construction in the proof of Lemma~\ref{lem:psi-surjectivity} assigns to an augmentation $\saug \in Aug(\sNgres)$ an MCS $\sMCS_{\saug}$. The MCS $\sMCS_{\saug}$ is in $\sAugform$-form. In fact, the crossings of $\sMCS_{\saug}$ with handleslide marks to their immediate left correspond to the resolved crossings in $\sNgres$ that are augmented by $\saug$. This process is invertible. Suppose $\sMCS$ is in $\sAugform$-form with handleslide marks to the immediate left of crossings $p_1, \hdots, p_k$. From Lemma~\ref{lem:mcs-aug-dipped}, we have an associated augmentation $\saug_{\sMCS}$ in a sufficiently dipped diagram with dips $D_1, \hdots, D_m$. We may undip $D_1, \hdots, D_m$, beginning with $D_m$ and working to the left, so that the resulting augmentation on $\sNgres$ only augments the crossings corresponding to $p_1, \hdots, p_k$. As a result, we have the following corollary.

\begin{cor}
The set of MCSs in $\sAugform$-form are in bijection with the augmentations on $\sNgres$ and if $\sMCS \in \sDMCS$ is in $A$-form, then $\Psi(\sMCS) = [\saug]$ where $\saug(p)=1$ if and only if $p$ is a marked crossing in $\sMCS$.
\end{cor} 

The following corollary uses the $S\bar{R}$-form and the $A$-algorithm to reprove the many-to-one relationship between augmentations and graded normal ruling first noted in \cite{Ng2006}.

\begin{cor}
	\label{cor:gnr-aug-SR-form}
	Let $\sgnr$ be a graded normal ruling on $\sfront$ with switched crossings $q_1, \hdots, q_n$ and graded returns $p_1, \hdots, p_m$. Then the $2^{m}$ MCSs in $S\bar{R}$-form with graded normal ruling $\sgnr$ correspond to $2^{m}$ different augmentations on $\sNgres$.
\end{cor}

\begin{proof}
By Corollary~\ref{cor:gnr-SR-form}, $\sgnr$ corresponds to $2^{m}$ MCSs in $S\bar{R}$-form. In fact, two MCSs $\sMCS_1$ and $\sMCS_2$ in $S\bar{R}$-form corresponding to $\sgnr$ differ only by handleslide marks around graded returns. If $\sMCS_1$ and $ \sMCS_2$ differ at the return $p$, then after applying the $\sAugform$-algorithm, the resulting MCSs $\sMCS'_1$ and $\sMCS'_2$ differ at $p$ as well. Thus the augmentations on $\sNgres$ corresponding to $\sMCS_1$ and $\sMCS_2$ will differ on the resolved crossing corresponding to $p$.
\end{proof}

\section{Two-Bridge Legendrian Knots}
\label{ch:two-bridge}

In this section, we prove that in the case of front projections with two left cusps the map $ \widehat{\Psi} : \sDMCSeq \to \sAugNgresch$ is bijective. 

\begin{defn}
	\label{defn:2-bridge-front}
	A front $\sfront$ of a Legendrian knot $\sK$ with exactly 2 left cusps is called a \emph{$2$-bridge front projection}. 
\end{defn} 

In \cite{Ng2001a}, Ng proves that every smooth knot admitting a 2-bridge knot projection is smoothly isotopic to a Legendrian knot admitting a 2-bridge front projection. Thus, the following results apply to an infinite collection of Legendrian knots.

\begin{defn}
	\label{defn:d-r-pair}
Given a graded normal ruling $\sgnr$ on a front $\sfront$, we say two crossings $q_i < q_j$ of $\sfront$, ordered by the $x$-axis, form a \emph{departure-return pair} $(q_i, q_j)$ if $\sgnr$ has a departure at $q_i$ and a return at $q_j$ and the two ruling disks that depart at $q_i$ are the same disks that return at $q_j$.

A departure-return pair $(q_i, q_j)$ is \emph{graded} if both crossings have grading 0. There are five possible arrangements of $q_i$ and $q_j$ in a graded departure-return pair; see, for example, Figure~\ref{f:departure-return-pairs-g}(a). We let $\nu(\sgnr)$ denote the number of graded departure-return pairs of $\sgnr$.
\end{defn}

Given a fixed graded normal ruling $\sgnr$ on any front projection $\sfront$, each unswitched crossing is either a departure or a return and thus part of a departure-return pair. In the case of a 2-bridge front projection, we can say more. 

\begin{prop}
	\label{prop:dr-pairs}
Suppose $\sfront$ is a 2-bridge front projection with graded normal ruling $\sgnr$. For each departure-return pair $(q_i, q_j)$ of $\sgnr$, no crossings or cusps of $\sfront$ may appear between $q_i$ and $q_j$. In terms of the ordering of crossings by the $x$-axis, this says $j = i+1$.
\end{prop}

\begin{proof}
Since $\sfront$ is a 2-bridge front projection, there are only two ruling disks for $\sgnr$. A switch or right cusp cannot appear after a departure since the two ruling disks overlap. Thus a return must immediately follow a departure.
\end{proof}

\begin{defn}
	\label{defn:SRg-form}
	An MCS $\sMCS \in \sDMCS$ is in \emph{$S \bar{R}_g$-form} if $\sMCS$ is in $S \bar{R}$-form and each marked return is part of a graded departure-return pair.
\end{defn}

\begin{figure}[t]
\labellist
\small\hair 2pt
\pinlabel {(a)} [tl] at 43 8
\pinlabel {(b)} [tl] at 298 8
\endlabellist
\centering
\includegraphics[scale=.5]{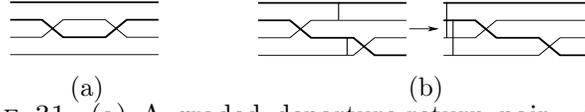}
\caption[Graded departure-return pairs.]{(a) A graded departure-return pair. (b) Unmarking a graded return paired with an ungraded departure.}
\label{f:departure-return-pairs-g}
\end{figure}

\begin{lem}
	\label{lem:srg-existence}
If $\sfront$ is a 2-bridge front projection, every MCS class in $\sDMCSeq$ has a representative in $S \bar{R}_g$-form.
\end{lem}

\begin{proof}

Let $\sbMCS \in \sDMCSeq$ and let $\sMCS$ be an $S \bar{R}$-form representative of $\sbMCS$. Suppose $(q_i, q_{i+1})$ is a departure-return pair of $\sgnr_{\sMCS}$ such that $q_{i+1}$ is a marked graded return with an ungraded departure $q_{i}$. We can push the handleslide mark(s) at $q_{i+1}$ to the left, past the ungraded departure $q_{i}$; see Figure~\ref{f:departure-return-pairs-g}(b). Using MCS move 17, we may remove these handleslide mark(s). In this manner, we eliminate all of the handleslide marks at graded returns that are paired with ungraded departures. Thus, $\sMCS$ is equivalent to an MCS in $S \bar{R}_g$-form. 

\end{proof}

In fact, the $S \bar{R}_g$-form found in the previous proof is unique.

\begin{lem}
	\label{lem:srg-uniqueness}
If $\sfront$ is a 2-bridge front projection, every MCS class in $\sDMCSeq$ has a unique representative in $S \bar{R}_g$-form. Thus, $|\sDMCSeq| = \sum_{N \in \sSgnr} 2^{\nu(\sgnr)}$.
\end{lem}

\begin{proof}
Let $\sMCS_1$ and $\sMCS_2$ be two representatives of $\sbMCS$ in $S \bar{R}_g$-form. 
Since $\sMCS_1 \sMCSeq \sMCS_2$, $\sMCS_1$ and $\sMCS_2$ induce the same graded normal ruling on $\sfront$, which we will denote $\sgnr$. Thus, $\sMCS_1$ and $\sMCS_2$ have the same handleslide marks around switches and only differ on their marked returns. Suppose for contradiction that $\sMCS_1$ and $\sMCS_2$ differ at the graded departure-return pair $(q_{i}, q_{i+1})$. We assume $q_{i+1}$ is a marked return in $\sMCS_1$ and is unmarked in $\sMCS_2$. We will prove that $\Psi(\sMCS_1) \neq \Psi(\sMCS_2)$. Thus, by Lemma~\ref{lem:mcs-equiv-aug-equiv} we will have the desired contradiction.

Recall that $\Psi(\sMCS_1)$ and $ \Psi(\sMCS_2)$ are computed by mapping the augmentations $\saug_{\sMCS_1}$ and $\saug_{\sMCS_2}$ constructed in Lemma~\ref{lem:mcs-aug-dipped} to augmentations in $Aug(\sNgres)$ using a dipping/undipping path. The augmentations $\saug_{\sMCS_1}$ and $\saug_{\sMCS_2}$ occur on sufficiently dipped diagrams $\sNgres^{d_1}$ and $\sNgres^{d_2}$. The dipped diagrams $\sNgres^{d_1}$ and $\sNgres^{d_2}$ differ by one or two dips between the resolved crossings $q_{i}$ and $q_{i+1}$. Add these dip to $\sNgres^{d_2}$ and extend $\saug_{\sMCS_2}$ by 0 using Lemma~\ref{lem:extend-by-0}. We let $\widetilde{\saug}_{\sMCS_2}$ denote the resulting augmentation on $\sNgres^{d_1}$. 

From Lemma~\ref{lem:dip-path}, the definition of $\Psi$ is independent of dipping/undipping paths. Thus $\Psi(\sMCS_1) = \Psi(\sMCS_2)$ if and only if $\saug_{\sMCS_1}$ and $\widetilde{\saug}_{\sMCS_2}$ are chain homotopic as augmentations on $\sNgres^{d_1}$. We will show that they are not chain homotopic. Suppose for contradiction that $H : (\aac(\sNgres^{d_1}), \df) \to \zz_2$ is a chain homotopy between $\saug_{\sMCS_1}$ and $\widetilde{\saug}_{\sMCS_2}$. We will prove a contradiction exists for two of the five possible arrangements of a graded departure-return pair. The arguments for the remaining three cases are essentially identical.

\begin{figure}[t]
\labellist
\small\hair 2pt
\pinlabel {$j$} [tl] at 56 18
\pinlabel {$j$} [tl] at 166 18
\pinlabel {$j$} [tl] at 294 18
\pinlabel {$j$} [tl] at 398 18
\pinlabel {(a)} [tl] at 90 18
\pinlabel {(b)} [tl] at 328 18
\endlabellist
\centering
\includegraphics[scale=.7]{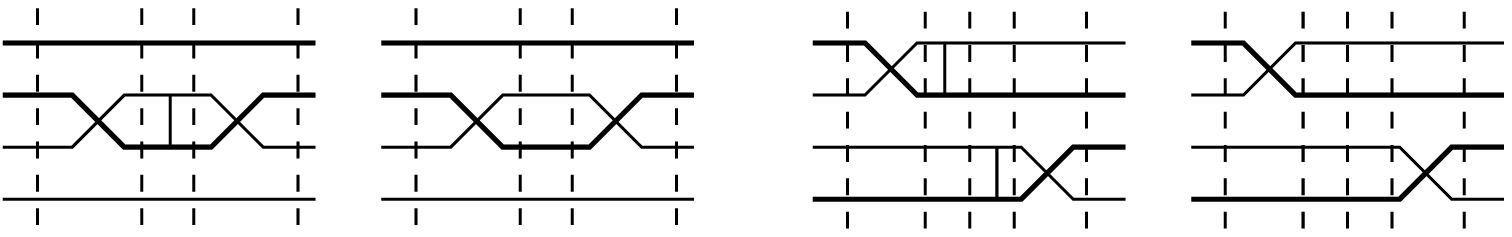}
\caption[A marked graded departure-return pair.]{In both (a) and (b) the MCS $\sMCS_1$ is on the left and $\sMCS_2$ is on the right.}
\label{f:marked-middle}
\end{figure}

{\bf Case 1:} Suppose the graded departure-return pair is arranged as in Figure~\ref{f:marked-middle}(a). The dotted lines in Figure~\ref{f:marked-middle} (a) indicate the location of the dips in $\sNgres^{d_1}$. Let $k+1$ and $k$ denote the strands crossing at $q_{i+1}$. The following calculations use the formulae from Lemma~\ref{lem:dipped-df} and the fact that all of the chain complexes involved are simple in the sense of Definition~\ref{defn:simple-complex} or are only one handleslide away from being simple. The chain homotopy $H$ must satisfy:

\begin{enumerate}
	\item $H ( a_j^{k+1,k} ) = H (\df q_{i+1}) = \saug_{\sMCS_1} - \widetilde{\saug}_{\sMCS_2}(q_{i+1}) = 0$, 
	\item $H ( a_j^{k+1,k} ) + H ( a_{j-1}^{k+1,k} ) = H (\df b_j^{k+1,k}) = \saug_{\sMCS_1} - \widetilde{\saug}_{\sMCS_2}(b_j^{k+1,k}) = 1$, \mbox{ and}
	\item $H ( a_{j-1}^{k+1,k} ) = H (\df b_{j-1}^{k+1,k}) = \saug_{\sMCS_1} - \widetilde{\saug}_{\sMCS_2}(b_{j-1}^{k+1,k}) = 0$.
\end{enumerate}
Combining (1) and (2), we see that $H ( a_{j-1}^{k+1,k} ) = 1$, but this contradicts (3). Thus $\saug_{\sMCS_1}$ and $\widetilde{\saug}_{\sMCS_2}$ are not chain homotopic.

{\bf Case 2:} Suppose the graded departure-return pair is arranged as in Figure~\ref{f:marked-middle}(b). The dotted lines in Figure~\ref{f:marked-middle} (b) indicate the location of the dips in $\sNgres^{d_1}$. Let $k+1$ and $k$ denote the strands crossing at $q_{i}$ and let $l+1$ and $l$ denote the strands crossing at $q_{i+1}$. The chain homotopy $H$ must satisfy:

\begin{enumerate}
	\item $H ( a_j^{l+1,l} ) = H (\df q_{i+1}) = \saug_{\sMCS_1} - \widetilde{\saug}_{\sMCS_2}(q_{i+1}) =0$,
	\item $H(a_j^{k+1,k}) = H (\df a_j^{k+1, l}) = \saug_{\sMCS_1} - \widetilde{\saug}_{\sMCS_2}(a_j^{k+1, l}) = 0$,
	\item $H ( a_j^{k+1,k} ) + H ( a_{j-1}^{k+1,k} ) = H ( \df b_j^{k+1,k}) = \saug_{\sMCS_1} - \widetilde{\saug}_{\sMCS_2}(b_j^{k+1,k}) = 0$,
	\item $H ( a_{j-1}^{k+1,k} ) + H ( a_{j-2}^{k+1,k} ) = H ( \df b_{j-1}^{k+1,k}) = \saug_{\sMCS_1} - \widetilde{\saug}_{\sMCS_2}(b_{j-1}^{k+1,k}) = 1$, \mbox{ and}
	\item $H ( a_{j-2}^{k+1,k} ) = H ( \df b_{j-2}^{k+1,k}) = \saug_{\sMCS_1} - \widetilde{\saug}_{\sMCS_2}(b_{j-2}^{k+1,k}) = 0$.
\end{enumerate}
The second equation follows from $H ( a_j^{l+1,l} )=0$ and the formula for $\df a_j^{k+1, l}$. Combining (2) and (3), we see that $H ( a_{j-1}^{k+1,k} ) = 0$. Combining (4) with $H ( a_{j-1}^{k+1,k} ) = 0$, we see that $H ( a_{j-2}^{k+1,k} ) = 1$. This contradicts (5). Thus $\saug_{\sMCS_1}$ and $\widetilde{\saug}_{\sMCS_2}$ are not chain homotopic.
\end{proof}

Using these two lemmata, we prove Theorem~\ref{thm:intro-two-cusps}.

\begin{thm}
	\label{thm:two-bridge-srg-form}
	If $\sfront$ is a 2-bridge front projection, then $\widehat{\Psi} : \sDMCSeq \to \sAugNgresch$ is a bijection.
\end{thm}

\begin{proof}
The surjectivity of $\widehat{\Psi}$ is the content of Theorem~\ref{thm:big-map}. We need only show injectivity. Suppose $\widehat{\Psi}([\sMCS_1]) = \widehat{\Psi}([\sMCS_2])$ and let $\sMCS_1$ and $\sMCS_2$ be $S \bar{R}_g$-representatives of $[\sMCS_1]$ and $[\sMCS_1]$. Since each MCS class has a unique $S \bar{R}_g$-representative by Lemma~\ref{lem:srg-uniqueness}, we need to show $\sMCS_1 = \sMCS_2$. This is equivalent to showing that $\sMCS_1$ and $\sMCS_2$ induce the same graded normal ruling on $\sfront$ and have the same marked returns.

Recall that $\Psi(\sMCS_1)$ and $ \Psi(\sMCS_2)$ are computed by mapping the augmentations $\saug_{\sMCS_1}$ and $\saug_{\sMCS_2}$ constructed in Lemma~\ref{lem:mcs-aug-dipped} to augmentations in $Aug(\sNgres)$ using a dipping/undipping path. The augmentations $\saug_{\sMCS_1}$ and $\saug_{\sMCS_2}$ occur on sufficiently dipped diagrams $\sNgres^{d_1}$ and $\sNgres^{d_2}$. We may add dips to $\sNgres^{d_1}$ and $\sNgres^{d_2}$ and extend by 0 using Lemma~\ref{lem:extend-by-0} so that the resulting augmentations $\widetilde{\saug}_{\sMCS_1}$ and $\widetilde{\saug}_{\sMCS_2}$ occur on the same sufficiently dipped diagram $\sNgresD$. Then $\widehat{\Psi}([\sMCS_1]) = \widehat{\Psi}([\sMCS_2])$ implies $\widetilde{\saug}_{\sMCS_1}$ and $\widetilde{\saug}_{\sMCS_2}$ are chain homotopic as augmentations on $\sNgresD$. 

Let $\sH : (\aac(\sNgresD), \df) \to \zz_2$ be a chain homotopy between $\widetilde{\saug}_{\sMCS_1}$ and $\widetilde{\saug}_{\sMCS_2}$. Since $\widetilde{\saug}_{\sMCS_1}$ and $\widetilde{\saug}_{\sMCS_2}$ are occ-simple, Corollary~\ref{cor:dipped-ch-htpy} gives $\widetilde{\saug}_{\sMCS_1} (A_j) = (I + \sH (A_j)) \widetilde{\saug}_{\sMCS_2} (I + \sH (A_j))^{-1}$ for all $j$. Recall that $\widetilde{\saug}_{\sMCS_1} (A_j)$ and $\widetilde{\saug}_{\sMCS_2} (A_j)$ encode the differential of a chain complex in $\sMCS_1$ and $\sMCS_2$ respectively. Since $\widetilde{\saug}_{\sMCS_1} (A_j)$ and $\widetilde{\saug}_{\sMCS_2} (A_j)$ are chain isomorphic by a lower triangular matrix, the pairing of the strands of $\sfront$ determined by $\widetilde{\saug}_{\sMCS_1} (A_j)$ and $\widetilde{\saug}_{\sMCS_2} (A_j)$ agree. Thus, $\sMCS_1$ and $\sMCS_2$ determine the same graded normal ruling on $\sfront$.

In the proof of Lemma~\ref{lem:srg-uniqueness} we show two MCSs in $S \bar{R}_g$-form determining the same graded normal ruling on $\sfront$ are mapped to chain homotopic augmentations only if they have the same marked returns. Since $\widetilde{\saug}_{\sMCS_1} (A_j)$ and $\widetilde{\saug}_{\sMCS_2} (A_j)$ are chain homotopic, this implies $\sMCS_1$ and $\sMCS_2$ have the same marked returns and thus $\sMCS_1 = \sMCS_2$ as desired.
\end{proof}

As an immediate corollary, we have:

\begin{cor}
	\label{cor:two-bridge-count}
	If $\sfront$ is a 2-bridge front projection, then $|\sAugNgresch| = |\sDMCSeq| = \sum_{N \in \sSgnr} 2^{\nu(\sgnr)}$ where $\nu(\sgnr)$ denotes the number of graded departure-return pairs of $\sgnr$.
\end{cor}

Corollary~\ref{cor:two-bridge-count} corresponds to Corollary~\ref{cor:intro-two-cusps} in Section~\ref{ch:intro}.

\addcontentsline{toc}{section}{Bibliography} 
\bibliographystyle{amsplain}
\bibliography{MBH-Biblio}

\end{document}